\documentclass{article}

\usepackage[english]{babel}
\usepackage{xcolor}
\usepackage[letterpaper,top=2cm,bottom=2cm,left=3cm,right=3cm,marginparwidth=1.75cm]{geometry}
\usepackage{subcaption}
\usepackage{booktabs}
\usepackage[utf8]{inputenc}
\usepackage{amsmath}
\usepackage{graphicx}
\usepackage[colorlinks=true, allcolors=blue]{hyperref}
\usepackage{algorithm}
\usepackage{algpseudocode}
\usepackage{amsfonts}
\usepackage{amsthm}
\newtheorem{lemma}{Lemma}
\newtheorem{definition}{Definition}
\newtheorem{theorem}{Theorem}
\newtheorem{proposition}{Proposition}
\newtheorem{corollary}{Corollary}
\newtheorem{example}{Example}

\theoremstyle{remark}
\newtheorem{remark}{Remark}
\title{A Barrier Primal Dual Hybrid Gradient Method for Solving Linear Programming Problems}

\author{
Yingxin Zhou, Stefano Cipolla, and Phan T. Vuong\\[0.5em]
\small School of Mathematical Sciences, University of Southampton\\
\small
\texttt{yz11u24@soton.ac.uk},
\texttt{S.Cipolla@soton.ac.uk},
\texttt{T.V.Phan@soton.ac.uk}
}

\date{}

\begin{document}
\maketitle
\begin{abstract}
Primal Dual Hybrid Gradient (PDHG) method has been verified to exhibit a two stage convergence behavior, in
which a prolonged active set identification phase may be a major issue
of slow convergence. In this paper, we propose Barrier PDHG (BPDHG), a nested algorithm which incorporates a logarithmic barrier function into the PDHG framework to
alleviate this problem. We first establish convergence of the inner iterations,
derive an error bound for the corresponding inner problem. Then we prove that the outer sequence generated by BPDHG approaches the optimal solution set of the LP problem we considered.
Furthermore, we 
integrate the barrier technique into the {Primal Dual Linear Programming} (PDLP) framework to develop the
corresponding Barrier PDLP (BPDLP) method. Numerical experiments show that the barrier modification can alleviate prolonged plateaus in the KKT residual on selected instances. We also investigate an empirical
instance-dependent indicator for identifying LP problems on which BPDLP
is more likely to outperform PDLP.
\end{abstract}
\section{Introduction}
Many problems arising in real world applications, such as finance \cite{LP-bankasset,LP-SHELLDIST,LP-YIEMANA}, energy systems \cite{LP-ENER,LP-ENER-2}, and machine learning \cite{LP-MAC-1,LP-MAC-2,MR4784765} can be formulated as linear programming (LP) problems. A standard form of LP problem is given by
\begin{align}\label{p:stand-LP}
    &\min_{x \in \mathbb{R}^n}\;c^{\top}x, \\
    &\text{ s.t. } Ax=b,\; x\geq 0, \nonumber
\end{align}
where $c\in \mathbb{R}^{n},\; A\in\mathbb{R}^{m\times n},\;b\in \mathbb{R}^{m}$. 
As the scale and complexity of practical LP
models continue to grow, there is an increasing demand for algorithms
that can solve large-scale instances efficiently and reliably. 

Though many modern commercial LP solvers are already capable to meet the requirements of practical applications, most of them rely on simplex and interior point methods \cite{simplex63,SIMPLEX-87,IPM89,IPM97,MR4594481},
or combinations of these methods. These
methods are mature and effective for a wide range of LP problems.
Nevertheless, their computational and memory requirements may become substantial for certain large-scale problems.
In particular, interior point methods typically require the solution of large linear systems and sparse matrix factorizations at each iteration, which may
lead to high memory consumption. Likewise, simplex methods may require a large number of iterations on large-scale instances.

In recent years, first order methods have received increasing attention for solving large-scale LP problems. Compared with simplex and
interior point methods, the first order methods \cite{LP-CGM-61,QSQP-17,LP-FOM-OVERVIEW} generally avoid expensive matrix factorizations and use relatively inexpensive iterative updates.
They often have low per iteration costs and memory requirements, and many of them are well suited to parallel computation. These favorable
properties make the first order methods attractive in large-scale optimization settings.

Among these methods, the Primal Dual Hybrid Gradient (PDHG) method \cite{ChambollePDHG} and
its variants \cite{google2022p,KCHPRLP26,LiuCross24,PDLPnew2026,IDSLU2023,GEOMETRYPDHGLP-24,INFEA-PDHG-24,zhou2026andersonacceleratedprimaldualhybrid} have attracted particular attention. Their low per iteration cost, low memory requirements, scalability, and suitability for
parallel implementation make them promising approaches for solving
large-scale LP problems.

A representative example is PDLP, a PDHG-based solver proposed in
\cite{google2022p}. By incorporating several practical enhancements,
PDLP achieves competitive performance compared with classical simplex
and interior point methods. GPU-based
implementations of PDLP were subsequently developed in
\cite{LUGPU25,LU24CUPDLP}, further exploiting the parallel structure of
PDHG. Building on PDLP, several additional acceleration techniques have also been
proposed. For example, \cite{LURHPDHG24} incorporated Halpern iteration
and its reflected variant into PDLP, resulting in improved computational
performance, while \cite{Lu2025OnlineP} introduced an online
preconditioning strategy to further accelerate the method.

Since PDLP is built upon the PDHG framework, we first look at how the
basic PDHG iteration operates when applied to
problem~\eqref{p:stand-LP}. The classical PDHG
iteration is given by
\begin{align}\label{classical-pdhg}
    x^{k+1}
    &=
    \arg\min_{x\geq0}
    \left\{
    c^\top x-(y^k)^\top Ax
    +\frac{1}{2\tau}\|x-x^k\|^2
    \right\}
=
    \max\left\{
    x^k+\tau(A^\top y^k-c),\,0
    \right\},
    \nonumber\\
    y^{k+1}
    &=
    \arg\min_{y\in\mathbb R^m}
    \left\{
    -b^\top y
    +y^\top A(2x^{k+1}-x^k)
    +\frac{1}{2\sigma}\|y-y^k\|^2
    \right\}
=
    y^k-\sigma\left(A(2x^{k+1}-x^k)-b\right),
\end{align}
where \(\tau>0\) and \(\sigma>0\) are the primal and dual step sizes,
respectively.

In particular, the primal update involves an
explicit projection onto the nonnegative orthant, which determines how
the iterates behave near the boundary \(x_i=0\). When the first argument of the projection operator is negative in a
given coordinate, the projection sets that coordinate to zero.
At later iterations, the same
coordinate may become positive again if the search direction changes.
Therefore, some coordinates may repeatedly switch between zero and
small positive values before the algorithm finally determines their behavior
near an optimal solution.

This behavior is related to the first stage of  two phase convergence behavior of PDHG
reported in \cite{LiuCross24,GEOMETRYPDHGLP-24}, consisting of an
active set identification phase followed by a convergence phase.
During the identification phase, the iterates may mainly adjust variables near the boundary. Such adjustments do not necessarily reduce the dominant components of the KKT residual immediately. Consequently, the residual history may exhibit a prolonged plateau before the method enters the faster convergence phase.
Related work of two stage behavior has also been studied from a geometric perspective recently, using condition measures such as LP sharpness and limiting error ratios \cite{Xiong-SHARP-LP}.

The difficulty of this identification process may be further understood clearer through
the local structure of the optimal solution. As discussed in \cite[Theorem~4]{GEOMETRYPDHGLP-24}, given an optimal primal dual
solution \((x^{**},y^{**})\) of problem~\eqref{p:stand-LP}, the primal
coordinates can be partitioned into three sets:
\begin{align*}
  &N = \{ i\in [n] : c_i - A_i^\top y^{**} > 0 \},\\
&B_1 = \{ i\in [n] : c_i - A_i^\top y^{**} = 0,\ x_i^{**}>0 \},\\
&B_2 = \{ i\in [n] : c_i - A_i^\top y^{**} = 0,\ x_i^{**}=0 \}.
\end{align*}
By complementary slackness, coordinates in \(N\) are active at the
optimal solution, whereas those in \(B_1\) are strictly positive and
hence nonactive. Coordinates in \(B_2\) are degenerate, since both the
primal variable and the corresponding reduced cost vanish at the
optimal solution. 

This partition provides a possible perspective for explaining a prolonged active-set identification phase. A coordinate may become
nearly degenerate along the algorithmic trajectory if both
\(x_i^k\) and \(c_i-A_i^\top y^k\) are small (i.e., a coordinate may satisfy
\(x_i\approx0\) and \(c_i-A_i^\top y\approx0\) simultaneously). Consequently, the argument of the primal projection may lie close to the projection threshold. Small changes in the $x$ update
may then cause the coordinate to switch between zero and a small
positive value, leading to prolonged plateau in the KKT residual. This observation further reveals a potential limitation of the standard PDHG $x$-update near the boundary: the explicit projection onto the nonnegative space may repeatedly push nearly degenerate coordinates onto the boundary and subsequently allow them to return to the interior region.

Inspired by interior point methods
\cite{Vanderbei2020LP}, we seek to avoid the direct use of the projection operator for handling the
constraint \(x\geq 0\). Instead of projecting the iterates onto the
nonnegative space, interior point methods introduce a logarithmic
barrier that keeps the primal variables strictly positive. When the
variables approach the boundary \(x_i=0\), the barrier term becomes increasingly large and penalizes further movement toward the boundary.
The iterates therefore approach the solution through the interior of
the feasible region, avoiding the repeated boundary-interior behavior induced by the projection based mechanism.

However, in classical interior point methods \cite{IPM97}, the iterates approximately follow the central path through Newton type steps applied to the perturbed KKT
system. Although this smooth interior trajectory helps avoid the projection induced
switching behavior described above, such Newton type steps
usually require the solution of large linear systems, which can be
computationally expensive for large-scale problems. We therefore do not adopt the full Newton based interior point
framework. Instead, we retain the logarithmic barrier mechanism and
incorporate it into the PDHG iteration directly.

Specifically, we introduce the logarithmic barrier term \(-\mu\sum_{i=1}^n\log(x_i)\) to handle the nonnegativity constraint, and replace \eqref{p:stand-LP} by the following barrier problem:
\begin{align}\label{p:barrLP}
    \min_{x\in \mathbb{R}^n}\quad & c^\top x - \mu\sum_{i=1}^{n}\log(x_i)\\
    \text{s.t.}\quad & Ax=b,\nonumber
\end{align}
where the logarithmic barrier implicitly enforces \(x>0\), and
\(\mu>0\) is a barrier parameter.

We then apply PDHG to solve this barrier problem and call the resulting method as Barrier PDHG (BPDHG). The reformulated
problem~\eqref{p:barrLP} can be viewed as a smooth interior approximation of the original LP \eqref{p:stand-LP}. The parameter \(\mu\) controls the strength of the approximation:
a larger value of \(\mu\) keeps the solution far away from the
boundary, whereas as \(\mu\to0\), the solutions of
problem~\eqref{p:barrLP} approach the solution set of the original LP. 

We note that \cite{IPM-PDHG26} also considered combining interior point
ideas with PDHG. However, their method uses an interior point procedure
to generate a warm start for PDHG, whereas our approach directly applies
PDHG to a sequence of logarithmic barrier subproblems. 
In addition, we note that {ADMM-based Interior Point Method (ABIP) \cite{LMY18IADMM} and its enhanced version (ABIP+) \cite{DQ24EIADMM} were} among the first methods to incorporate
a logarithmic barrier mechanism into the first order framework, using ADMM to solve the resulting barrier subproblems. However, the two approaches differ in both
formulation and algorithmic structure. ABIP/ABIP+ operates on a homogeneous self-dual embedding and uses ADMM as its inner solver, whereas our method applies PDHG directly to a logarithmic barrier reformulation of the original LP. The resulting primal update of our method admits an explicit coordinatewise solution and preserves the basic primal dual structure of PDHG. Moreover, our main focus is on understanding and alleviating the projection induced boundary behavior of PDHG and on trying to incorporate such a barrier mechanism into PDLP.

\subsection{Contribution}

We now summarize our main contributions as follows.
\begin{enumerate}
\item 
We propose a Barrier PDHG (BPDHG), a nested method that incorporates a
logarithmic barrier term into the PDHG framework. For each fixed barrier
parameter $\mu$, the resulting primal subproblem admits a closed form solution which keeps the low cost property and parallel structure of PDHG.
\item
For each fixed barrier parameter, we establish an error bound for
the corresponding inner problem of BPDHG. Based on
this result, we analyze the convergence of the inner BPDHG iterations
and the overall framework as the barrier parameter decreases. 
We further derive finite accuracy complexity bounds on both the number of outer iterations and the total number of inner BPDHG iterations required to a prescribed accuracy~$\iota$. Building on these results, we incorporate the barrier mechanism into the PDLP framework and develop the corresponding Barrier PDLP (BPDLP) method.

\item 
We conduct numerical experiments comparing BPDHG with PDHG, and BPDLP
with PDLP. The results show that the barrier modification can alleviate
the prolonged plateau in the KKT residual caused by long term active set identification stage on some LP instances. However, for the instance which PDLP already solved rapidly, the
additional outer barrier loop may provide only limited benefit.

We further investigate problem characteristics associated with the
relative performance of BPDLP and PDLP. Motivated by our theoretical
analysis, we propose a computable instance-dependent quantity as a
one sided proxy for potentially weak sharpness \cite{fasterPDHG-23} (i.e., weaker sharpness may be associated with slower
progress of PDHG, whereas stronger sharpness generally provides
more effective control of the distance to the optimal solution set). Although this quantity
is only a rough indicator and does not provide an exact characterization
of sharpness, our numerical results suggest that it may still contain useful empirical information for identifying LP instances on which BPDLP is more likely to outperform PDLP.
\end{enumerate}

The structure of the paper are as follows. In Section~\ref{SEC:PRELI}, we introduce the necessary preliminaries. In Section~\ref{SEC:BPDHG}, we present the algorithmic framework of BPDHG and show that the distance from its outer iterates to the optimal solution set of the original LP converges to zero. In Section~\ref{SEC:EXPER}, we first compare the numerical performance
of BPDHG and PDHG on selected LP instances. We then incorporate the
barrier mechanism into the PDLP framework to develop Barrier PDLP (BPDLP), and evaluate BPDLP against
PDLP on a larger collection of LP instances. Finally, we investigate
the problem characteristics associated with their relative performance, try to identify LP instances on which BPDLP
is more likely to be advantageous.

\section{Preliminaries}\label{SEC:PRELI}
In this section, we present the notation and some important preliminaries used throughout the paper.

We use \(\mathbb{N}:=\{0,1,2,\ldots\}\) to denote the set of nonnegative integers.
Let \(\mathbb{R}^n\) denote the \(n\)-dimensional Euclidean space, and
\(
\mathbb{R}^n_+
:=
\left\{
x\in\mathbb{R}^n
\mid x\geq 0 
\right\},\;\mathbb{R}^n_{++}
:=
\left\{
x\in\mathbb{R}^n
\mid x> 0 
\right\}.
\) 
Denote \(x^\top\) as the transpose of vector \(x\in \mathbb{R}^n\), its $i$-th component is denoted by $x_i$, $i = 1, \ldots, n$. Besides, for such a $x$, the {positive part} of $x$ is denoted by $(x)_+$ that satisfies $(x)_+=\max\{x,\;0\}$.
For \(a,b\in\mathbb{R}\) with \(a\leq b\), we denote \([a,b]^n\) as the \(n\)-dimensional box set
\[
[a,b]^n:=\{x\in\mathbb{R}^n\mid a\leq x_i\leq b,\ i=1,\ldots,n\}.
\]

We use $\mathbf{e}\in \mathbb{R}^n$ as the vector of all ones, and use \(\langle\cdot,\cdot\rangle\) and \(\|\cdot\|\) to denote the standard Euclidean inner product and its associated norm, respectively. For a nonempty closed set
\(\Omega\subseteq\mathbb{R}^n\), the distance from a point \(x\) to \(\Omega\) is defined by
\(
\operatorname{dist}(x,\Omega)
:=
\min_{z\in\Omega}\|x-z\|.
\)

For a nonsingular matrix \(A\in\mathbb{R}^{n\times n}\), its inverse is denoted by \(A^{-1}\). Let \(I\) be the identity matrix. Let \(\operatorname{diag}(A)\in\mathbb{R}^n\) be the vector consisting of the diagonal entries of \(A\), that is, \((\operatorname{diag}(A))_i=A_{ii}\) for \(i=1,\ldots,n\). As for a vector \(x\in\mathbb{R}^n\), \(\operatorname{diag}(x)\in\mathbb{R}^{n\times n}\) denotes the diagonal matrix whose \(i\)-th diagonal entry is \(x_i\).

For a symmetric matrix \(Q\in\mathbb{R}^{n\times n}\), we write \(Q\succ0\) and \(Q\succeq0\) to indicate that \(Q\) is positive definite and positive semidefinite, respectively. Equivalently, \(Q\succ0\) if \(x^\top Qx>0\) for all \(x\in\mathbb{R}^n\setminus\{0\}\), whereas \(Q\succeq0\) if \(x^\top Qx\geq0\) for all \(x\in\mathbb{R}^n\).
Furthermore, if $Q$ is a symmetric positive definite matrix. The \(Q\)-inner product is defined as
\(
\langle x,y\rangle_Q:=x^\top Qy,
\)
with the associated norm
\(
\|x\|_Q:=\sqrt{x^\top Qx}.
\)
When \(Q=I\), the \(Q\)-inner product and \(Q\)-norm reduce to the standard Euclidean inner product and norm, respectively. 
\begin{definition}
Let matrix $G \in \mathbb{R}^{m\times n}$ with $m \le n$. We say $G$ has \emph{full row rank}
if $\mathrm{rank}(G) = m$, i.e., the $m$ rows of $G$ are linearly independent.
\end{definition}
For an operator \(\mathcal{T}:\mathbb{R}^n\to\mathbb{R}^n\), the fixed point set of \(\mathcal{T}\) is defined by
\(\operatorname{Fix}(\mathcal{T}):=\{x\in\mathbb{R}^n\mid \mathcal{T}(x)=x\}\).
Any \(x\in\operatorname{Fix}(\mathcal{T})\) is called a fixed point of \(\mathcal{T}\). Moreover, in the following, \(f(\mu)=\Theta(g(\mu))\) means that there exist positive constants \(c_1,c_2\), independent of \(\mu\), such that
\[
c_1|g(\mu)|\le |f(\mu)|\le c_2|g(\mu)|.
\]

\begin{definition}[Proximal Operator \cite{first-order-beck}]\label{de:proximal-op}
Given a proper function $f:\mathbb{R}^n \to \mathbb{R}\cup\{+\infty\}$ and a positive constant $\tau$, the proximal operator (mapping) of $f$ is defined as
\[
\mathrm{prox}_{\tau f}(x) := \arg\min_{u \in \mathbb{R}^n} \left\{ f(u) + \frac{1}{2\tau}\|u - x\|^2 \right\},\; \text{ for any $x\in\mathbb{R}^n$}.
\]
\end{definition}
Let $\mathcal{I}$ be the identity operator. Let $\mathcal{C} \subseteq \mathbb{R}^n$ be a nonempty, closed, and convex set. The indicator function of $\mathcal{C}$ is defined as
\[
\mathcal{I}_{\mathcal{C}}(x) =
\begin{cases}
0, & x \in \mathcal{C},\\
+\infty, & x \notin \mathcal{C}.
\end{cases}
\]

\begin{definition}[Projection Operator \cite{first-order-beck}]\label{de:projection}
Let $f$ in Definition~\ref{de:proximal-op} be given by $f(x)= \mathcal{I}_{\mathcal{C}}(x)$, then 
\[
\mathrm{prox}_{\tau f}(x) := \arg\min_{u \in \mathbb{R}^n} \left\{  \mathcal{I}_{\mathcal{C}}(u) + \frac{1}{2\tau}\|u - x\|^2 \right\}
= \arg\min_{u \in \mathcal{C}} \|u-x\|=\mathcal{P}_{\mathcal{C}}(x)
,\; \text{ for any $x\in\mathbb{R}^n$}.
\]
Here \( \mathcal{P}_{\mathcal{C}}(x)\)
denotes the projection of \(x\) onto \(\mathcal{C}\).
\end{definition} The projection operator  $\mathcal{P}_{\mathcal{C}}$ has nonexpansiveness property, i.e.,
\[
\|\mathcal{P}_{\mathcal{C}}(x) - \mathcal{P}_{\mathcal{C}}(y)\| \leq \|x - y\|, \quad \forall x,\;y \in \mathbb{R}^n.
\]

\begin{definition}[Subgradient and Subdifferential \cite{first-order-beck}]\label{de:subgradient}
Let $f:\mathbb{R}^n \to (-\infty,\infty]$ be a proper convex function and $x \in \mathrm{dom}(f)$. A vector $g \in \mathbb{R}^n$ is called a subgradient of $f$ at $x$ if
\[
f(y) \geq f(x) + \langle g, y - x \rangle, \quad \forall y \in \mathbb{R}^n.
\]
The set of all subgradients at $x$ is called the subdifferential of $f$ at $x$, denoted by $\partial f(x)$.
\[
\partial f(x) = \{g\mid f(y) \geq f(x) + \langle g, y - x \rangle, \quad \forall y \in \mathbb{R}^n\}.
\]
\end{definition}

\begin{definition}[Differentiability]\label{de:differentiability}
Let \(f:\mathbb{R}^n\to(-\infty,+\infty]\) and
\(x\in\operatorname{int}(\operatorname{dom}f)\).
The function \(f\) is said to be differentiable at \(x\) if there exists
\(g\in\mathbb{R}^n\) such that
\[
\lim_{h\to 0}
\frac{f(x+h)-f(x)-\langle g,h\rangle}{\|h\|}
=0.
\]
The unique vector \(g\) satisfying this condition is called the gradient
of \(f\) at \(x\) and is denoted by \(\nabla f(x)\).
\end{definition}
If \(f\) is convex and differentiable at \(x\), then
\(\partial f(x)=\{\nabla f(x)\}\).

\begin{definition}[Fej\'er Monotonicity \cite{convex-monotone}]\label{de:fejer-monotonicity}
Let \(\{x^k\}_{k\in\mathbb{N}}\) be a sequence in $\mathbb{R}^n$. Let $\mathcal{R}$ be a nonempty subset of $\mathbb{R}^n$. The sequence \(\{x^k\}_{k\in\mathbb{N}}\) is said to be Fej\'er monotone with respect to \(\mathcal{R}\) if
\[
\|x^{k+1}-a\|
\leq
\|x^k-a\|,
\qquad
\text{for all }a\in \mathcal{R},\; k\in\mathbb{N}.
\]
\end{definition}

\begin{proposition}[Convergence of Fej\'er Monotone Sequences \cite{convex-monotone}]
\label{prop:fejer-strong-convergence}
Let \(\{x^k\}_{k\in\mathbb{N}}\) be a sequence in \(\mathbb{R}^n\), and let
\(\mathcal{R}\subseteq\mathbb{R}^n\) be nonempty and closed. Suppose that
\(\{x^k\}_{k\in\mathbb{N}}\) is Fej\'er monotone with respect to
\(\mathcal{R}\). If every cluster point of \(\{x^k\}_{k\in\mathbb{N}}\)
belongs to \(\mathcal{R}\), then \(\{x^k\}_{k\in\mathbb{N}}\) converges
to a point in \(\mathcal{R}\).
\end{proposition}

\section{Barrier Primal Dual Hybrid Gradient method (BPDHG)}\label{SEC:BPDHG}
In this section, we first describe how PDHG is applied to solve
problem~\eqref{p:barrLP} and then provide the BPDHG framework, together with its convergence analysis.

Recall that for problem~\eqref{p:barrLP}, the corresponding Lagrange function has the following form:
\begin{align*}
    \mathcal{L}_{\mu}(x,y)=c^{\top}x-\mu\sum_{i=1}^n log(x_i)-y^{\top}(Ax-b).
\end{align*}
We then get the PDHG subproblems update framework as follows,
\begin{align}
    x^{k+1}&=\arg\min_{{x \in \mathbb{R}_{++}^n}}\Big\{
    c^{\top}x - \mu\sum_{i=1}^{n}log (x_i)-{(y^k)}^{\top}(Ax)
    +\frac{1}{2\tau}\|x-x^k\|^2
    \Big\}\nonumber\\
    &=\operatorname{prox}_{(-\tau \mu\sum_{i=1}^{n}log (x_i))}(x^k+\tau(A^{\top}y^k-c)),\label{bpdhg-x}\\
    y^{k+1}&=\arg\min_{{y \in \mathbb{R}^m}}\Big\{-b^{\top}y+
    y^{\top}(A(2x^{k+1}-x^k))+\frac{1}{2\sigma}\|y-y^k\|^2
    \Big\}=y^k - \sigma(A(2x^{k+1}-x^k)-b).\label{bpdhg-y}
\end{align}
Although the $x$-subproblem may appear difficult to handle, the presence of the logarithmic term makes it admit a closed form solution \cite{first-order-beck}. In particular,
\begin{align}\label{bpdhg-x-expilic}
    x^{k+1}=\dfrac{z^k+\sqrt{(z^k)^2+4\tau\mu}}{2}, \text{ where $z^k=x^k+\tau A^{\top}y^k-\tau c.$}
\end{align}
From \eqref{bpdhg-x-expilic},
since $\sqrt{(z^k)^2+4\tau\mu}>|z^k|$,
we observe that, on the one hand, avoiding the projection operator helps keeping the $x$ iterate away from the boundary; on the other hand, introducing the logarithmic barrier makes the $x$-subproblem objective function differentiable, which will be beneficial for theoretical analysis.

On top of that, comparing the PDHG updates for problems~\eqref{p:stand-LP} and \eqref{p:barrLP}, we observe that
their \(y\)-updates are identical and that the only modification occurs in the \(x\)-update. The introduction of the logarithmic barrier
transforms the optimality conditions of the original LP into a \(\mu\) dependent perturbed Karush–Kuhn–Tucker (KKT) system. In interior point methods, such KKT systems characterize the
central path concept and are typically solved approximately using Newton type
iterations. Therefore, to better understand the role of the barrier parameter \(\mu\) and its relationship with the original problem~\eqref{p:stand-LP}, we compare the KKT conditions of problem~\eqref{p:stand-LP} and the barrier problem~\eqref{p:barrLP}, which are denoted by \(\mathrm{KKT}_{\mathrm{standard}}\) and \(\mathrm{KKT}_{\mathrm{barrier}}\), respectively. Let $X = \mathrm{diag}(x)$, we have

\begin{align}
 \mathrm{KKT}_{\mathrm{standard}}=&
\begin{cases}
Ax = b,\\
A^\top y + s_1 = c,\\
{X} s_1 = 0,\\
x \ge 0, s_1 \ge 0,\\
\end{cases}\label{kkt-lp} \text{\quad where $s_1$ is the dual slack variable,}
\\
\mathrm{KKT}_{\mathrm{barrier}}=
&\begin{cases}
Ax = b,\\
A^\top y + s_2 = c,\\
X s_2 = \mu \mathbf{e},\\
x > 0,\; s_2>0,
\end{cases}\label{kkt-barrier} 
\end{align}
Here, $s_2 := \mu X^{-1}\mathbf{e}$ is the dual slack variable implicitly induced by the
logarithmic barrier. As $\mu \to 0$, the perturbed complementarity condition in
$\mathrm{KKT}_{\mathrm{barrier}}$,
\(
    X s_2 = \mu \mathbf{e},
\)
implies that
\[
    x_i (s_{2})_i = \mu \to 0, \;\text{ as }\mu \to 0. \qquad i = 1, \ldots, n.
\]
Therefore, as \(\mu\to0\), the complementarity condition in
\(\mathrm{KKT}_{\mathrm{barrier}}\) reduces to the standard
complementarity condition in
\(\mathrm{KKT}_{\mathrm{standard}}\).

Besides, from \eqref{kkt-barrier}, the barrier parameter \(\mu\) directly controls the primal dual gap, i.e., $$c^{\top}x-b^{\top}y=c^{\top}x-x^{\top}A^{\top}y=
x^{\top}(c-A^{\top}y)=x^{\top}s_2=n\mu.$$
Consequently, to recover a solution of the original problem~\eqref{p:stand-LP}, the barrier parameter $\mu$ should be reduced and driven to zero, while the corresponding barrier subproblems are solved to an appropriate level of accuracy.

For ease of analysis, we next reformulate update \eqref{bpdhg-x}, \eqref{bpdhg-y}
as a fixed point iteration following the idea of \cite{degenerate2021}. This reformulation also provides the basis
for the subsequent development and analysis of the algorithm.
In particular,
denote the operators
$\mathcal{A},\; \mathcal{M}$ as
$$\mathcal{A}:=
\begin{bmatrix}
  \tau\bar{A} & -\tau{A}^{\top} \\
  \sigma A     & \sigma\bar{B}^{-1}
\end{bmatrix},\;\quad
\mathcal{M}:=
\begin{bmatrix}
    I &\tau A^{\top} \\  \sigma A  & I
\end{bmatrix},
$$
where $
\bar{ A}(x)
:=
c-\mu X^{-1}\mathbf e,
\;
\bar{B}^{-1}(y):=-b.$ 
Then following \cite{degenerate2021}, the steps \eqref{bpdhg-x}, \eqref{bpdhg-y} can be written in the following fixed point iteration form:
\begin{align*}
    \begin{bmatrix}
        x^{k+1}\\ y^{k+1} 
    \end{bmatrix}= \mathcal{T}_{\mu}
    \begin{bmatrix}
        x^{k}\\ y^{k} 
    \end{bmatrix}:=(\mathcal{A}+\mathcal{M})^{-1}\mathcal{M}
\begin{bmatrix}
        x^{k}\\ y^{k} 
    \end{bmatrix}.
\end{align*}
For the subsequent analysis, we introduce the following notation as well:
\[
f_{\mu}(x)
:=
c^\top x-\mu\sum_{i=1}^{n}\log(x_i),
\qquad
h(y):=-b^\top y,
\]
Letting $z=(x,\;y)\in \mathbb{R}_{++}^n\times\mathbb{R}^m$ and define 
\begin{align}\label{defi:Q}
G_{\mu}(z):=f_{\mu}(x)+h(y),\;
    F(z):=\begin{bmatrix}
        -A^{\top} y\\ Ax
    \end{bmatrix},\; Q:=
    \begin{bmatrix}
        \frac{1}{\tau} & A^{\top}\\  A& \frac{1}{\sigma} 
    \end{bmatrix},
\end{align}
It is straightforward to verify that, for any
\(z_1,z_2\in\mathbb{R}_{++}^{n}\times\mathbb{R}^{m}\),
the operator \(F\) satisfies
\begin{align}\label{F-PRO}
\left\langle
z_1-z_2,\,
F(z_1)-F(z_2)
\right\rangle
=0.
\end{align}
We emphasize that, throughout this paper, the primal and dual step
sizes are chosen to satisfy
\(
    \tau\sigma\|A\|^2<1,
\)
which is also a standard step size choice for the convergence analysis of PDHG.
Besides, this condition guarantees that \(Q\succ0\).
And we have \[
{\sigma_{\min}(Q)}\|z\|^2\leq\|z\|^2_Q\leq {\sigma_{\max}(Q)}\|z\|^2,
\quad  \dfrac{1}{\sigma_{\max}(Q)}\|z\|^2\leq\|z\|^2_{Q^{-1}}\leq \dfrac{1}{\sigma_{\min}(Q)}\|z\|^2,
\]
where $\sigma_{\min}(Q),\;\sigma_{\max}(Q)$ are the minimum and maximum singular value of matrix $Q$.

The following lemma characterizes the equivalence between the optimal solution of problem~\eqref{p:barrLP} and the fixed points of $\mathcal{T}_{\mu}$.
\begin{lemma}\label{le:optimal-fix}
Let $\mu>0$, and $\mathcal{T}_\mu$ denote the PDHG operator associated with problem~\eqref{p:barrLP}. Then \[ z_{\mu}^*=(x_{\mu}^*,y_{\mu}^*)\in\operatorname{Fix}(\mathcal{T}_\mu) \] if and only if $z_{\mu}^*$ is a primal dual {solution} associated with problem~\eqref{p:barrLP}. Equivalently, \[ \operatorname{Fix}(\mathcal{T}_\mu)=Z_\mu^*, \] where $Z_\mu^*$ denotes the set of primal dual solutions of the barrier problem.
\end{lemma}
\begin{proof}
Assume that $z_{\mu}^*\in \mathrm{Fix}(\mathcal{T}_{\mu})$. Then $z_{\mu}^*=\mathcal{T}_{\mu}(z^*)$ and therefore the fixed point residual $\hat{g}$ vanishes, i.e.,
$$\hat{g}:=z_{\mu}^*-\mathcal{T}_{\mu}(z_{\mu}^*)=0.$$
Writing the PDHG iterations \eqref{bpdhg-x} and \eqref{bpdhg-y} explicitly, we obtain
\begin{align*}
    x_{\mu}^* &= \mathrm{prox}_{\tau f_{\mu}}\bigl(x_{\mu}^*+\tau A^{\top}y_{\mu}^*\bigr),\\
    y_{\mu}^* &= \mathrm{prox}_{\sigma h}\bigl(y_{\mu}^*-\sigma Ax_{\mu}^*\bigr),
\end{align*}
where the functions $ f_{\mu}, h$ are defined in \eqref{defi:Q}. 
By the optimality conditions of the proximal operators, the above relations are equivalent to
\begin{align*}
    0 &= \nabla f_{\mu}(x_{\mu}^*)-A^{\top}y_{\mu}^* = c-\mu \bar{X}^{-1}\mathbf{e}-A^{\top}y_{\mu}^*,\\
    0 &= Ax_{\mu}^*+\nabla h = Ax_{\mu}^*-b,
\end{align*}
where $\bar{X}^{-1}=\operatorname{diag}(1/(x_{\mu}^*)_i),\;(x_{\mu}^*)_i>0,\; i\in\{1,2,\ldots,n\}$.
    Consequently, $(x_{\mu}^*,y_{\mu}^*)$ satisfies \[ c-\mu \bar{X}^{-1}\mathbf{e}-A^\top y_{\mu}^*=0, \qquad Ax_{\mu}^*-b=0, \qquad x_{\mu}^*>0. \] These are precisely the primal dual optimality conditions for problem~\eqref{p:barrLP}. Hence, \[ z_{\mu}^*\in Z_\mu^*. \]
Conversely, if $z_{\mu}^*$ is an optimal primal dual solution of problem~\eqref{p:barrLP}, then it satisfies the KKT conditions in~\eqref{kkt-barrier}. Reversing the above argument, we would obtain
$z_{\mu}^*=\mathcal{T}_{\mu}(z_{\mu}^*)$
and hence
\[
z_{\mu}^*\in \operatorname{Fix}(\mathcal{T}_{\mu}).
\]
\end{proof}

\subsection{Convergence of BPDHG}\label{se:con-BPDHG}
In this subsection, we first present the framework of BPDHG by using the fixed point structure, and then establish its convergence analysis.
\begin{algorithm}[]
\caption{Barrier PDHG (BPDHG)}\label{al:BPDHG}
\begin{algorithmic}[1]
\State \text{Given} the fixed $\mu^0>0$, operator $\mathcal{T}_{\mu^0}$, $z^{0,0}=(x^{0,0},y^{0,0})$ and $\theta \in (0,1)$.
\Repeat{ \( j=0,1,2,\ldots\)}
    \Repeat { $k=0,1,2,...$}
  \State {use} PDHG for solving $z^{j,k+1}=\mathcal{T}_{\mu^{j}}(z^{j,k})$\label{al:BPDL4}
 \Until stopping condition holds
\State Compute: $\mu^{j+1}=\theta \mu^{j}$\label{update-mu}
\State Set $z^{j+1,0}\gets z^{j,k+1}$
\Until stopping condition holds
\end{algorithmic}
\end{algorithm}

\begin{remark}
Algorithm~\ref{al:BPDHG} adopts a nested structure to approximate the optimal solution of problem~\eqref{p:stand-LP}.
In the inner loop, the primal dual variables are updated via the PDHG iterations given in \eqref{bpdhg-x} and \eqref{bpdhg-y} with a fixed $\mu$.
In the outer loop, the barrier parameter is decreased to drive the {iterates to move toward} the optimal point of problem~\eqref{p:stand-LP}.
\end{remark}

Next, we proceed to prove the convergence of the algorithm. We first establish a lemma that will be used in the subsequent analysis. For notational simplicity, we fix the outer iteration index \(j\) and consider only the inner iteration index \(k\). In particular, we write
\[
z_\mu^k:=z_\mu^{j,k},\quad
z_\mu^*:=z^*(\mu^j) \text{ for a fixed $j$.}
\]
Here $z_\mu^*\in Z_\mu^*$ is {an optimal primal dual solution of problem \eqref{p:barrLP}, and $Z_\mu^*$ is the set of all optimal primal dual solutions of problem \eqref{p:barrLP}.}

\begin{lemma}\label{le:VI-G}
Fix an outer iteration index \(j\), and let
\(
z_\mu^*=(x_\mu^*,y_\mu^*)\in\operatorname{Fix}(\mathcal{T}_\mu).
\)
Suppose that the sequence
\(
{z_\mu^k=(x_\mu^k,y_\mu^k)}
\)
is generated by the inner loop of Algorithm~\ref{al:BPDHG}. Then, for every \(k\geq 0\),
\[
G_{\mu}(z_\mu^{k+1})-G_{\mu}(z_{\mu}^*)
+\left\langle z_\mu^{k+1}-z_{\mu}^*,F(z_{\mu}^*)\right\rangle
\geq 0.
\]
\end{lemma}
\begin{proof}
{
Since the outer iteration index $j$ is fixed, consider the Lagrangian of
problem~\eqref{p:barrLP},
\[
\mathcal{L}(x,y)=f_\mu(x)-\langle Ax,y\rangle-h(y).
\]
A primal dual optimal solution
$z_\mu^*=(x_\mu^*,y_\mu^*)$ is a saddle point of $\mathcal{L}$ and hence satisfies
\[
\mathcal{L}(x_\mu^*,y)
\leq
\mathcal{L}(x_\mu^*,y_\mu^*)
\leq
\mathcal{L}(x,y_\mu^*)
\]
for any $x\in\mathbb{R}_{++}^n$ and $y\in\mathbb{R}^m$.
}

 From \[ \mathcal{L}(x_\mu^*,y_\mu^*) \leq \mathcal{L}(x,y_\mu^*), \] we obtain \[ f_\mu(x_\mu^*)-f_\mu(x) -\langle Ax_\mu^*,y_\mu^*\rangle +\langle Ax,y_\mu^*\rangle \leq 0. \] 
 Similarly, from \[ \mathcal{L}(x_\mu^*,y) \leq \mathcal{L}(x_\mu^*,y_\mu^*), \] 
 we obtain \[ -\langle Ax_\mu^*,y\rangle +\langle Ax_\mu^*,y_\mu^*\rangle -h(y)+h(y_\mu^*) \leq 0. \] 
 Adding the above two inequalities gives
 \[ f_\mu(x)-f_\mu(x_\mu^*) + h(y)-h(y_\mu^*) - \langle Ax,y_\mu^*\rangle + \langle Ax_\mu^*,y\rangle \geq 0. \] 
 Recall \eqref{defi:Q}, \[ G_{\mu}(z):=f_\mu(x)+h(y), \qquad F(z):= \begin{bmatrix} -A^\top y\\ Ax \end{bmatrix}. \] 
 By these notation, we have \[ \left\langle z-z_\mu^*,F(z_\mu^*)\right\rangle = -\langle x-x_\mu^*,A^\top y_\mu^*\rangle + \langle y-y_\mu^*,Ax_\mu^*\rangle, \] which is equivalent to \[ \left\langle z-z_\mu^*,F(z_\mu^*)\right\rangle = -\langle Ax,y_\mu^*\rangle + \langle Ax_\mu^*,y\rangle. \] 
 Therefore, \[ G_{\mu}(z)-G_{\mu}(z_\mu^*) + \left\langle z-z_\mu^*,F(z_\mu^*)\right\rangle \geq 0 \] holds for every \(z=(x,y)\in\mathbb{R}_{++}^n\times\mathbb{R}^m\). Taking \(z=z_\mu^{k+1}\) yields \[ G_{\mu}(z_\mu^{k+1})-G_{\mu}(z_\mu^*) + \left\langle z_\mu^{k+1}-z_\mu^*, F(z_\mu^*) \right\rangle \geq 0. \] \end{proof}
 
Again, fixing the outer iteration index \(j\), in the following, we use Lemma~\ref{le:VI-G}, to establish the non-increasing property of the sequence \({\|z_{\mu}^{k+1}-z_{\mu}^*\|_Q}\). And we further show that the distance between two successive inner iterates is bounded in terms of the initial point and the optimal solution of the inner loop problem.
\begin{theorem}\label{Th:Z-ZSTAR}
    Fix the outer iteration $j$, suppose $z_{\mu}^*=(x_{\mu}^*,y_{\mu}^*)$ is the fixed point of $\mathcal{T}_{\mu}$. Then for the sequence $ \{z_{\mu}^k\}$ generated by the inner loop of Algorithm \ref{al:BPDHG}, it
     holds for every \(k\geq 0\),
\begin{align}\label{ineq:zk+1-zstartleqzk-zstat}
\|z_{\mu}^{k+1}-z_{\mu}^*\|_Q\leq \|z_{\mu}^k-z_{\mu}^*\|_Q,\quad \|z_{\mu}^{k+1}-z_{\mu}^{k}\|_Q\leq \sqrt{\dfrac{1}{k+1}}\|z_{\mu}^0-z_{\mu}^*\|_Q, \end{align}
where $Q$ is defined in \eqref{defi:Q}.
\end{theorem}
\begin{proof}
Since the logarithmic barrier term ensures that the minimizer of subproblem \eqref{bpdhg-x} lies in the interior of $\mathbb{R}_{++}^n$, i.e.,
$x_{\mu}^{k+1}\in\mathbb{R}_{++}^n$. So no normal cone term is required in the first order optimality condition. Therefore, from the first order optimality condition of subproblem \eqref{bpdhg-x}, we have
\begin{align}\label{op-con-x}
0=\nabla f_{\mu}(x_{\mu}^{k+1})-A^{\top}y_{\mu}^k+\frac{1}{\tau}(x_{\mu}^{k+1}-x_{\mu}^k).
\end{align}
Since \(f_{\mu}\) is convex, for any \(x>0\), we have
\[
f_{\mu}(x)-f_{\mu}(x_{\mu}^{k+1})\geq \langle\nabla f_{\mu}(x_{\mu}^{k+1}),x-x_{\mu}^{k+1}\rangle.
\]
Therefore, adding \(\langle -A^{\top}y_{\mu}^k+\frac{1}{\tau}(x_{\mu}^{k+1}-x_{\mu}^k), x-x_{\mu}^{k+1}\rangle\) to both sides and using \eqref{op-con-x}, we obtain
\begin{align}\label{ineq:fx-fxx}
&f_{\mu}(x)-f_{\mu}(x_{\mu}^{k+1})+\langle -A^{\top}y_{\mu}^k+\frac{1}{\tau}(x_{\mu}^{k+1}-x_{\mu}^k),x-x_{\mu}^{k+1}\rangle\nonumber\\
&\geq \langle\nabla f_{\mu}(x_{\mu}^{k+1}) -A^{\top}y_{\mu}^k+\frac{1}{\tau}(x_{\mu}^{k+1}-x_{\mu}^k),x-x_{\mu}^{k+1}\rangle=0.
\end{align}
Similarly, from the first-order optimality condition of subproblem \eqref{bpdhg-y}, we have
\begin{align}\label{op-con-y}
0=-b + A(2x_{\mu}^{k+1}-x_{\mu}^k)+\frac{1}{\sigma}(y_{\mu}^{k+1}-y_{\mu}^k).
\end{align}
Combining this equality with the convexity of $h$, we obtain, for any
$y\in\mathbb{R}^m$,
    \begin{align}\label{ineq:gy-gyy}
        &h(y)-h(y_{\mu}^{k+1})+\langle A(2x_{\mu}^{k+1}-x_{\mu}^k)+\frac{1}{\sigma}(y_{\mu}^{k+1}-y_{\mu}^k),\;y-y_{\mu}^{k+1}\rangle\nonumber\\
        &\geq \langle -b + A(2x_{\mu}^{k+1}-x_{\mu}^k)+\frac{1}{\sigma}(y_{\mu}^{k+1}-y_{\mu}^k),\;y-y_{\mu}^{k+1}\rangle=0.
    \end{align}
    Summing \eqref{ineq:fx-fxx} and \eqref{ineq:gy-gyy}, shows that for any $z \in \mathbb{R}_{++}^n\times \mathbb{R}^m$,
   \begin{align}\label{VI-final}
        G_{\mu}(z)-G_{\mu}(z_{\mu}^{k+1})+\langle z-z_{\mu}^{k+1}, F(z_{\mu}^{k+1})\rangle\geq \langle z-z_{\mu}^{k+1}, Q(z_{\mu}^k-z_{\mu}^{k+1}) \rangle.
    \end{align}
    By using $\langle a-b, Q(c-d)\rangle = \frac{1}{2}(\|a-d\|_Q^2-\|a-c\|_Q^2)+\frac{1}{2}(\|c-b\|_Q^2-\|d-b\|_Q^2)$ for $\langle z-z_{\mu}^{k+1}, Q(z_{\mu}^k-z_{\mu}^{k+1}) \rangle$, we further get,
\begin{align*}
        G_{\mu}(z)-G_{\mu}(z_{\mu}^{k+1})+\langle z-z_{\mu}^{k+1}, F(z_{\mu}^{k+1})\rangle \geq
        \frac{1}{2}(\|z-z_{\mu}^{k+1}\|_Q^2-\|z-z_{\mu}^k\|_Q^2)+\frac{1}{2}\|z_{\mu}^k-z_{\mu}^{k+1}\|_Q^2.
\end{align*}
Letting $z=z_{\mu}^*$, yields,
\begin{align*}
        G_{\mu}(z_{\mu}^*)-G_{\mu}(z_{\mu}^{k+1})+\langle z_{\mu}^*-z_{\mu}^{k+1}, F(z_{\mu}^{k+1})\rangle \geq
        \frac{1}{2}(\|z_{\mu}^*-z_{\mu}^{k+1}\|_Q^2-\|z_{\mu}^*-z_{\mu}^k\|_Q^2)+\frac{1}{2}\|z_{\mu}^k-z_{\mu}^{k+1}\|_Q^2.
\end{align*}
From \eqref{F-PRO}, we know that $\langle z_{\mu}^*-z_{\mu}^{k+1}, F(z_{\mu}^{k+1})\rangle=\langle z_{\mu}^*-z_{\mu}^{k+1}, F(z_{\mu}^{*})\rangle$, then the inequality becomes,
\begin{align*}
      \frac{1}{2} \|z_{\mu}^*-z_{\mu}^k\|_Q^2-
        \frac{1}{2}\|z_{\mu}^*-z_{\mu}^{k+1}\|_Q^2-\frac{1}{2}\|z_{\mu}^k-z_{\mu}^{k+1}\|_Q^2&\geq 
         G_{\mu}(z_{\mu}^{k+1})-G_{\mu}(z_{\mu}^*)+\langle z_{\mu}^{k+1}-z_{\mu}^*, F(z_{\mu}^{k+1})\rangle\\
         &=G_{\mu}(z_{\mu}^{k+1})-G_{\mu}(z_{\mu}^*)+\langle z_{\mu}^{k+1}-z_{\mu}^*, F(z_{\mu}^{*})\rangle \geq 0.
\end{align*}
The last inequality holds from Lemma \ref{le:VI-G}.
Then we have 
\begin{align}\label{ineq:z-decre}
   0\leq \|z_{\mu}^k-z_{\mu}^{k+1}\|_Q^2\leq \|z_{\mu}^*-z_{\mu}^k\|_Q^2- 
\|z_{\mu}^*-z_{\mu}^{k+1}\|_Q^2, 
\end{align} which implies $\|z_{\mu}^{k+1}-z_{\mu}^*\|_Q\leq \|z_{\mu}^k-z_{\mu}^*\|_Q.$ This establishes the first conclusion.

Summing \eqref{ineq:z-decre} from $k=0,1,\ldots,K$, yields,
\begin{align}\label{sum-se}
    \sum_{k=0}^K \|z_{\mu}^k-z_{\mu}^{k+1}\|^2_Q\leq \sum_{k=0}^K(\|z_{\mu}^*-z_{\mu}^k\|_Q^2-\|z_{\mu}^*-z_{\mu}^{k+1}\|_Q^2)=\|z_{\mu}^0-z_{\mu}^*\|^2_Q-\|z_{\mu}^*-z_{\mu}^{K+1}\|_Q^2.
\end{align}
From the inequality \eqref{VI-final}, letting $k=k+1$, it still holds,
    \begin{align*}
        G_{\mu}(z)-G_{\mu}(z_{\mu}^{k+2})+\langle z-z_{\mu}^{k+2}, F(z_{\mu}^{k+2})+Q(z_{\mu}^{k+2}-z_{\mu}^{k+1})\rangle\geq 0, \text{ for any $z \in \mathbb{R}_{++}^n\times \mathbb{R}^m$.}
    \end{align*}
Taking $z=z_{\mu}^{k+1},$
    \begin{align*}
        G_{\mu}(z_{\mu}^{k+1})-G_{\mu}(z_{\mu}^{k+2})+\langle z_{\mu}^{k+1}-z_{\mu}^{k+2}, F(z_{\mu}^{k+2})+Q(z_{\mu}^{k+2}-z_{\mu}^{k+1})\rangle\geq 0, 
    \end{align*}
Taking  $z=z_{\mu}^{k+2}$ in  \eqref{VI-final}, we have
    \begin{align*}
        G(z_{\mu}^{k+2})-G_{\mu}(z_{\mu}^{k+1})+\langle z_{\mu}^{k+2}-z_{\mu}^{k+1}, F(z_{\mu}^{k+1})\rangle\geq\langle z_{\mu}^{k+2}-z_{\mu}^{k+1}, Q(z_{\mu}^{k}-z_{\mu}^{k+1})\rangle
    \end{align*}
Summing the above two inequalities and using the property \eqref{F-PRO}, we have
\begin{align*}
     \langle z_{\mu}^{k+1}-z_{\mu}^{k+2},Q(z_{\mu}^{k}-z_{\mu}^{k+1}-(z_{\mu}^{k+1}-z_{\mu}^{k+2})) \rangle\geq 0.
\end{align*}
Adding the term $\langle (z_{\mu}^{k}-z_{\mu}^{k+1}-(z_{\mu}^{k+1}-z_{\mu}^{k+2}),Q((z_{\mu}^{k}-z_{\mu}^{k+1}-(z_{\mu}^{k+1}-z_{\mu}^{k+2})) \rangle$ into both sides, we have
\begin{align*}
    \langle z_{\mu}^{k}-z_{\mu}^{k+1},Q(z_{\mu}^{k}-z_{\mu}^{k+1}-(z_{\mu}^{k+1}-z_{\mu}^{k+2}) \rangle\geq\langle (z_{\mu}^{k}-z_{\mu}^{k+1}-(z_{\mu}^{k+1}-z_{\mu}^{k+2}),Q((z_{\mu}^{k}-z_{\mu}^{k+1}-(z_{\mu}^{k+1}-z_{\mu}^{k+2})) \rangle
\end{align*}
By using $\|a\|^2_Q-\|b\|^2_Q=2a^{\top}Q(a-b)-\|a-b\|^2_Q$, we have,
\begin{align*}
    &\|z_{\mu}^{k}-z_{\mu}^{k+1}\|^2_Q-\|z_{\mu}^{k+1}-z_{\mu}^{k+2}\|^2_Q\\
    &=2\langle z_{\mu}^{k}-z_{\mu}^{k+1},Q(z_{\mu}^{k}-z_{\mu}^{k+1}-(z_{\mu}^{k+1}-z_{\mu}^{k+2}))\rangle-\|z_{\mu}^{k}-z_{\mu}^{k+1}-(z_{\mu}^{k+1}-z_{\mu}^{k+2})\|_Q^2.
\end{align*}
Then it yields,
\begin{align}\label{eq:znon-increasing}
    \|z_{\mu}^{k}-z_{\mu}^{k+1}\|^2_Q-\|z_{\mu}^{k+1}-z_{\mu}^{k+2}\|^2_Q\geq \|z_{\mu}^{k}-z_{\mu}^{k+1}-(z_{\mu}^{k+1}-z_{\mu}^{k+2})\|_Q^2\geq0
    ,
\end{align}
which implies $\{\|z_{\mu}^k-z_{\mu}^{k+1}\|^2_Q\}$ is non-increasing.
By \eqref{sum-se}, we have
\begin{align}
    \sum_{k=0}^K \|z_{\mu}^k-z_{\mu}^{k+1}\|_Q^2 \leq\|z_{\mu}^0-z_{\mu}^*\|_Q^2,
\end{align}
letting $K\to \infty$, \begin{align}
    \sum_{k=0}^\infty \|z_{\mu}^k-z_{\mu}^{k+1}\|_Q^2 \leq\|z_{\mu}^0-z_{\mu}^*\|_Q^2<+\infty,
\end{align}
Since the sequence
$\{\|z_{\mu}^k-z_{\mu}^{k+1}\|_Q^2\}$ is nonincreasing, for any
$t\geq0$, we have
\begin{align*}
    (t+1)\|z_{\mu}^t-z_{\mu}^{t+1}\|_Q^2
    \leq
    \sum_{k=0}^t
    \|z_{\mu}^k-z_{\mu}^{k+1}\|_Q^2\leq
    \|z_{\mu}^0-z_{\mu}^*\|_Q^2.
\end{align*}
Therefore,
\[
\|z_{\mu}^t-z_{\mu}^{t+1}\|_Q
\leq
\frac{1}{\sqrt{t+1}}
\|z_{\mu}^0-z_{\mu}^*\|_Q.
\]
Replacing $t$ with $k$ gives the second desired result.

\end{proof}
As a direct consequence of Theorem~\ref{Th:Z-ZSTAR}, we show that the sequence generated by the inner loop converges to an optimal primal dual solution of the problem~\eqref{p:barrLP}, {when $\mu$ is fixed.}
\begin{corollary}\label{co:z-zstar}
Fix an outer iteration index \(j\). Then the sequence \(\{z_{\mu}^k\}\) generated by the inner loop of Algorithm~\ref{al:BPDHG} converges to the optimal primal dual solution \(z_{\mu}^*\) of problem~\eqref{p:barrLP}.
\end{corollary}

\begin{proof}
From Theorem~\ref{Th:Z-ZSTAR}, we have $\{z_{\mu}^k\}$ is Fej\'er monotone { in $Q$-norm }with respect to $\operatorname{Fix}(\mathcal{T}_{\mu})$ and bounded. By \eqref{ineq:zk+1-zstartleqzk-zstat}, we have
\begin{align}\label{eq:gklim}
  0\leq\lim_{k\to \infty}\|\mathcal{T}_{\mu}(z_{\mu}^k)-z_{\mu}^k\|_Q=\lim_{k\to \infty}\|z_{\mu}^{k+1}-z_{\mu}^k\|_Q\leq\lim_{k\to \infty}\sqrt{\dfrac{1}{k+1}}\|z_{\mu}^0-z_{\mu}^*\|_Q=0.  
\end{align}
From the boundedness of $\{z_{\mu}^k\}$, there exists a subsequence $\{z_{\mu}^{k_n}\}$ such that $\lim_{n\to \infty}z_{\mu}^{k_n} =\bar{z}_{\mu}.$ Since $Q\succ 0$, and by the continuity of $\mathcal{T}_{\mu}$,
the inequality \eqref{eq:gklim} implies that
\begin{align*}
    \mathcal{T}_{\mu}(\bar{z}_{\mu})=\lim_{n\to \infty}\mathcal{T}_{\mu}(z_{\mu}^{k_n})
    =\lim_{n\to \infty}z_{\mu}^{k_n} = \bar{z}_{\mu},
\end{align*}
Hence, we have $\bar{z}_{\mu}\in \operatorname{Fix}(\mathcal{T}_{\mu})$. 
Since $\{z_{\mu}^k\}$ is Fejér monotone with respect to
$\operatorname{Fix}(\mathcal{T}_{\mu})$, and together with the existence of a convergent subsequence.
It follows that
\[
\|z_{\mu}^{k}-\bar z_{\mu}\|_Q\to0, \quad \text{as $k\to \infty$}
\]
and hence
\(
z_{\mu}^k\to\bar z_{\mu}, \; \text{as }{k\to \infty}.
\)
Finally, by Lemma~\ref{le:optimal-fix}, we conclude $\bar{z}_{\mu}\in Z_{\mu}^* $.

\end{proof}

The following lemma prepares the ground for establishing the sharpness condition of our Barrier PDHG, which paves the way for establishing the stopping criterion of inner scheme that ensures the overall convergence of the algorithm.
\begin{lemma}\label{LE:PDHG-Z-TZ}
 Let $\mu>0$ be fixed, and define 
 \begin{align}\label{matrix-Fmu}
     F_\mu(z) := \begin{bmatrix} c-\mu X^{-1}\mathbf{e}-A^\top y\\ Ax-b \end{bmatrix}, \qquad z=(x,y)\in\mathbb{R}^n\times \mathbb{R}^m, \end{align} 
     where $X=\operatorname{diag}(x)$ and $x>0$. For $z_\mu^*=(x_\mu^*,\;y_\mu^*) \in Z_\mu^*
 $ we have \( F_\mu(z_\mu^*)=0. \)
 
 Suppose that the iterates $(x_{\mu}^k,y_{\mu}^k)$ is generated from inner loop of BPDHG. 
 Assume that $A\in\mathbb{R}^{m\times n}$ has full row rank. 
 Then there exists a constant $\varsigma_\mu>0$
such that
\[
\operatorname{dist}(z_\mu^{k+1},Z_\mu^*)
\leq
\varsigma_\mu\|F_\mu(z_\mu^{k+1})\|,
\qquad k\geq0.
\]
\end{lemma}
\begin{proof} 
Here we first claim that under the assumption, $Z^*_{\mu}=\{z^*_{\mu}\}$ holds. 
Since \(f_\mu\) is strictly convex, then the primal barrier solution
\(x_\mu^*\) is unique. Moreover, from the KKT condition, as \(A\) has full row rank,
\(A^\top\) is injective, and hence the associated multiplier
\(y_\mu^*\) is also unique. Therefore,
\(Z_\mu^*=\{z_\mu^*\}\).

Since $F_\mu(z_\mu^*)=0$, for every $k$ we have 
\[ F_\mu(z_\mu^{k+1}) = F_\mu(z_\mu^{k+1})-F_\mu(z_\mu^*). \] 
For each component of $x$, for $i=1,\ldots,n$, \[ -\frac{\mu}{(x^{k+1}_{\mu})_i
} +\frac{\mu}{(x^*_{\mu})_i} = \frac{\mu}{(x^{k+1}_{\mu})_i (x^*_{\mu})_i} \left((x^{k+1}_{\mu})_i-(x^*_{\mu})_i\right). \] 
Therefore,
\[ F_\mu(z_\mu^{k+1}) =F_\mu(z_\mu^{k+1})-F_\mu(z_\mu^*)= P_{\mu}(x_\mu^{k+1})(z_\mu^{k+1}-z_\mu^*), \] where \[ P_{\mu}(x) := \begin{bmatrix} \Phi_\mu(x)&-A^\top\\ A&0 \end{bmatrix}, \qquad \Phi_\mu(x) := \operatorname{diag} \left( \frac{\mu}{(x)_i(x^*_{\mu})_i} \right)_{i=1}^n. \] 

We first show that $P_{\mu}(x)$ is nonsingular for every $x>0$.
Since for barrier problem, it implicitly keep $x>0$,  
for every $i$, both $(x_{\mu})_i$ and $(x^*_{\mu})_i$ are positive, which provide $\Phi_{\mu}(x)$ is positive definite. Togerther with $A$ is of full row rank, Schur Complement Theorem shows that $P_{\mu}$ is invertible and hence $P_{\mu}^{-1}$ exists. 
Next, we aim to prove that
\[
\|P_{\mu}^{-1}(x_{\mu}^{k+1})\|\leq \varsigma_{\mu}.
\] Note that, if this inequality holds, then 
\begin{align}\label{ineq:global-sharpness}
        \operatorname{dist}(z_{\mu}^{k+1},Z_{\mu}^*)=\|z_{\mu}^{k+1}-z_{\mu}^*\|=\|P_{\mu}^{-1}(x^{k+1}_{\mu}){F}_{\mu}(z_{\mu}^{k+1})\|\leq\varsigma_\mu \|{F}_{\mu}(z_{\mu}^{k+1})\|,
\end{align}
which draws the conclusion.

 Since $\sigma_{\max}(P_{\mu}^{-1})=\|P_{\mu}^{-1}\|=\dfrac{1}{\sigma_{\min}(P_{\mu})}$, it comes to establish a uniform lower bound on $\sigma_{\min}(P_{\mu})$.
By Corollary~\ref{co:z-zstar}, $x_{\mu}^{k+1}\to x_\mu^*>0 \text{ as } k\to \infty $, then for every $i$ there exists $K_i$ such that 
\begin{align}\label{ineq:xmui-bound}
    \frac{(x_{\mu}^*)_i}{2} \leq (x_{\mu}^k)_i \leq \frac{3(x_{\mu}^*)_i}{2}, \qquad k\geq K_i. 
\end{align} 
Let \( K:=\max_{1\leq i\leq n}K_i. \)
Then, for every $i$ and every $k\geq K$, \[ (x_{\mu}^{k+1})_i\geq \frac{(x_{\mu}^*)_i}{2}>0. \] 
For the finitely many indices $k=0,\ldots,K-1$, all components $(x_{\mu}^k)_i$ are positive. 
Letting \[ \delta_\mu := \min\left\{ \min_{\substack{0\leq k<K\\1\leq i\leq n}}(x_{\mu}^{k+1})_i, \frac{1}{2}\min_{1\leq i\leq n}(x_{\mu}^*)_i \right\} >0, \] 
we have the uniform lower bound of $x_{\mu}^{k+1}$, i.e., $x_{\mu}^{k+1}\geq
 \delta_{\mu}$.
Since in Theorem~\ref{Th:Z-ZSTAR}, we have shown that $\{x_{\mu}^{k+1}\}$ is bounded, then there exists $M_\mu<\infty$ such that
\[ (x_{\mu}^{k+1})_i\leq M_\mu, \qquad i=1,\ldots,n,\quad k\geq0. \] 
Consequently, \[ x_{\mu}^{k+1}\in D_\mu:=[\delta_\mu,M_\mu]^n \qquad \text{for all }k\geq0. \] 
Since the mapping $ x\mapsto P_{\mu}(x)$ is continuous on the bounded set $D_\mu$, and $ x\mapsto \sigma_{\min}(P_{\mu}(x))$ is also continuous and attains its minimum on $D_\mu$. 
Denote \[ {s}_\mu := \min_{x\in D_\mu} \sigma_{\min}(P_{\mu}(x)). \] Since $P_{\mu}(x)$ is nonsingular for every $x_{\mu}\in D_\mu$, we have \( s_\mu>0. \) 

Hence, 
\[ \|P^{-1}_{\mu}(x_{\mu}^{k+1})\| = \frac{1}{\sigma_{\min}(P_{\mu}(x_{\mu}^{k+1}))} \leq \frac{1}{s
_\mu} =: \varsigma_\mu. \]
This completes the proof. 
\end{proof}
\begin{corollary}\label{co-sharp}
    If Lemma~\ref{LE:PDHG-Z-TZ} holds, combining it with the first order optimality condition of \eqref{op-con-x} and \eqref{op-con-y}, yields,
    {
    \begin{align}\label{inqe:dist-zk-Tzk}
        \|z_{\mu}^{k+1}-z
        ^*_{\mu}\|_Q=\operatorname{dist}_Q(z_{\mu}^{k+1},Z^*_{\mu})\leq \sqrt{\dfrac{\sigma_{\max}(Q)}{\sigma_{\min}(Q)}} \varsigma_\mu\|\mathcal{N}\|\|z_{\mu}^k-\mathcal{T}_{\mu}z_{\mu}^k\|_Q:=\eta_{\mu} \|z_{\mu}^k-\mathcal{T}_{\mu}z_{\mu}^k\|_Q,
    \end{align}
    }
    where $\mathcal{N}=\begin{bmatrix}
    \dfrac{1}{\tau}\quad & A^{\top}  \\
    A\quad & \dfrac{1}{\sigma}
\end{bmatrix}.$
\end{corollary}
\begin{proof}
{
By rewriting the first order optimality conditions of \eqref{op-con-x} and \eqref{op-con-y} and then applying the Cauchy–Schwarz inequality, we have 
\[
\|F_{\mu}(z_{\mu}^{k+1})\|\leq \|\mathcal{N}\|\|z_{\mu}^k-\mathcal{T}_{\mu}z^k_{\mu}\|
\] Connect this result with the conclusion of Lemma~\ref{LE:PDHG-Z-TZ}, we obtain the desired conclusion.
}
\end{proof}
In what follows, we incorporate the outer iteration index \(j\) into the analysis. In particular, we show that the sequence generated by the algorithm approaches the solution set of problem~\eqref{p:stand-LP}.

We first establish an estimate for $\eta_{\mu}$ in \eqref{inqe:dist-zk-Tzk}. The following lemma serves as a preparatory result, i.e., as $\mu^j\to 0 \;(j\to \infty),$ the family of barrier solutions $\{x^*_{\mu^j}\}_{j\geq 0}$ have a uniform upper bound.
\begin{lemma}\label{le:uniform-bounded-barrier-solution}
Suppose that $A\in\mathbb{R}^{m\times n}$ has full row rank, and that $\mu^j$ is updated according to $\mu^{j+1}=\theta\mu^j$ with $\theta\in(0,1)$. Here $j$ is the outer iteration index of Algorithm~\ref{al:BPDHG}. Assume that problem~\eqref{p:stand-LP} and its dual problem both have strictly feasible points, i.e.,
there exists \((\widehat{x},\widehat{y},\widehat{s})\) such that
\[
\widehat{x}>0,\qquad A\widehat{x}=b,
\qquad
\widehat{s}>0,\qquad
A^\top\widehat{y}+\widehat{s}=c.
\]
Let \((x^{**},y^{**},s^{**})\) be a primal dual optimal solution of problem~\eqref{p:stand-LP}. By strong duality, define
\(
p^{**}:=c^\top x^{**}=\;b^\top y^{**}.
\)
For every $\mu^j>0$, let $z_{\mu^j}^*=(x_{\mu^j}^*,y_{\mu^j}^*)$ denote the corresponding point associated with $\mu^j$ which satisfies \eqref{kkt-barrier}. Denote
\[
    s_{\mu^j}^*:=c-A^\top y_{\mu^j}^*.
\]
Then the sequence $\{(x_{\mu^j}^*,y_{\mu^j}^*,s_{\mu^j}^*)\}_{j\ge0}$ is uniformly bounded. 
In particular, there exist constants $\hat{M}>0$ and $\kappa>0$ such that
\[
    \kappa\mu^j
    \leq (x^*_{\mu^j})_i
    \leq \hat{M},
    \qquad
    i=1,\ldots,n,
\]
for $0<\mu^j\leq\mu^0$.
\end{lemma}
\begin{proof}
Since \eqref{p:stand-LP} has an optimal solution $p^{**}=c^{\top}{x}^{**}=b^{\top}{y}^{**}$. And for $z_{\mu^j}^*$, it satisfies 
\[
Ax_{\mu^j}^* = b ,\quad A^{\top}y_{\mu^j}^*+s_{\mu^j}^* = c, \quad X_{\mu^j}^*s_{\mu^j}^* = {\mu^j} \mathbf{e}, \; \text{ and }\; c^{\top}x_{\mu^j}^*-b^{\top}y_{\mu^j}^*=n{\mu^j}.
\]
Weak duality property then gives
\(
b^{\top}y_{\mu^j}^*\leq p^{**}\leq c^{\top}x_{\mu^j}^*.
\)
It follows that
\begin{align}\label{ineq:bycx}
    b^{\top}y_{\mu^j}^* = c^{\top}x_{\mu^j}^* - n{\mu^j} \geq p^{**}- n{\mu^j}, \quad   c^{\top}x_{\mu^j}^* = n{\mu^j} + b^{\top}y_{\mu^j}^* \leq  p^{**}+n{\mu^j}.
\end{align}
From $\widehat{s}=c-A^{\top}\widehat{y}$ and $Ax_{\mu^j}^* = b$, we have
\begin{align*}
    \widehat{s}^{\top}x_{\mu^j}^*=(c-A^{\top}\widehat{y})^{\top}x_{\mu^j}^*=c^{\top}x_{\mu^j}^*-b^{\top}\widehat{y}\leq p^{**}+n{\mu^j}-b^{\top}\widehat{y}\leq p^{**}+n{\mu^0}-b^{\top}\widehat{y},
\end{align*}
where the first inequality holds due to \eqref{ineq:bycx} and the last inequality holds due to the update of $\mu^j$.

Now let $\widehat{s}_{\min}=\min_{1\leq i\leq n} \widehat{s}_i$. Since $x_{\mu^j}^*>0$, we have
\[
    0<\widehat{s}_{\min} \mathbf{e}^\top x_{\mu^j}^*
    \leq
    \widehat{s}^\top x_{\mu^j}^*\leq p^{**}+n{\mu^0}-b^{\top}\widehat{y}.
\]
Then $\|x_{\mu^j}^*\|\leq \dfrac{p^{**}+n{\mu^0}-b^{\top}\widehat{y}}{ \widehat{s}_{\min} }$, and we first establish the uniform boundedness of $x_{\mu^j}^*$. 

Likewise, from $A\widehat{x}=b$ and \eqref{ineq:bycx}, we have
\begin{align*}
    (s_{\mu^j}^*)^{\top}\widehat{x}=(c- A^{\top}y_{\mu^j}^*)^{\top}\widehat{x}=c^{\top}\widehat{x}-b^{\top}y_{\mu^j}^*\leq c^{\top}\widehat{x} + n\mu^j - p^{**}
    \leq  c^{\top}\widehat{x} + n\mu^0 - p^{**}.
\end{align*}
Let $\widehat{x}_{\min}=\min_{1\leq i\leq n} \widehat{x}_i$. We finally get
\begin{align*}
    0< \widehat{x}_{\min} \mathbf{e}^{\top}s_{\mu^j}^* \leq  (s_{\mu^j}^*)^{\top}\widehat{x}\leq  c^{\top}\widehat{x} + n\mu^0 - p^{**}.
\end{align*}
Then we build up the uniform boundedness of $s_{\mu^j}^*$, i.e., $
\|s_{\mu^j}^*\|\leq \dfrac{c^{\top}\widehat{x} + n\mu^0 - p^{**}}{\widehat{x}_{\min}}$.

As for $y_{\mu^j}^*$, since $A^{\top}y_{\mu^j}^* = c-s_{\mu^j}^*$ and $A$ has the full row rank, then 
\[
\sigma_{\min}(A)\|y_{\mu^j}^*\|\leq \|A^{\top}y_{\mu^j}^*\|\leq \|c\|+\|s_{\mu^j}^*\|\leq \|c\|+\dfrac{c^{\top}\widehat{x} + n\mu^0 - p^{**}}{\widehat{x}_{\min}},
\]
where $\sigma_{\min}(A)$ is the minimum singular value of $A$.
Rearrange it, yields,
\[
\|y_{\mu^j}^*\|\leq \dfrac{\|c\|}{{\sigma_{\min}(A)}} +\dfrac{c^{\top}\widehat{x} + n\mu^0 - p^{**}}{{\sigma_{\min}(A)}\widehat{x}_{\min}}.
\]
Hence, the sequence $\{(x_{\mu^j}^*,y_{\mu^j}^*,s_{\mu^j}^*)\}_{j\ge0}$ is uniformly bounded.

Moreover, letting $\hat{M}=\dfrac{p^{**}+n{\mu^0}-b^{\top}\widehat{y}}{ \widehat{s}_{\min} },\;\hat{N}=\dfrac{c^{\top}\widehat{x} + n\mu^0 - p^{**}}{\widehat{x}_{\min}}$. Then from $X_{\mu^j}^* s_{\mu^j}^*=\mu^j \mathbf{e},$
\[
\hat{M}\geq (x^*_{\mu^j})_i=\dfrac{\mu^j}{(s^*_{\mu^j})_i} \geq \dfrac{\mu^j}{\hat{N}}, \quad i=1,\ldots,n
\]
Let $\kappa=\dfrac{1}{\hat{N}}$, we draw the conclusion.
\end{proof}

In Lemma~\ref{LE:PDHG-Z-TZ}, we have for every fixed
\(\mu>0\),\[ \|P^{-1}_{\mu}(x_{\mu}^k)\|  \leq \varsigma_\mu. \] However, this estimate does not characterize the
dependence of \(\varsigma_\mu\) on the barrier parameter \(\mu\).
Such a characterization is needed to obtain a more explicit estimate
in \eqref{inqe:dist-zk-Tzk}.
The following result therefore provides an explicit bound on
\(\varsigma_{\mu}\) in terms of \(\mu\).

\begin{theorem}\label{th:jmu-inverse-bound}
Suppose that the assumptions of
Lemma~\ref{le:uniform-bounded-barrier-solution} hold.
Then there exists a constant \(C_J>0\), independent of \(j\), such that,
for each \(j\geq0\), there exists an index \(K_j\in\mathbb N\) for which
\[
    \left\|P^{-1}_{\mu^j}(x_{\mu^j}^{k+1})\right\|
    \leq \frac{C_J}{\mu^j},
    \qquad \text{for all } k\geq K_j.
\]
\end{theorem}

\begin{proof}

From Lemma~\ref{le:uniform-bounded-barrier-solution}, we have, for $0<\mu^j\leq\mu^0$,
\[
    \kappa\mu^j
    \leq (x^*_{\mu^j})_i
    \leq \hat{M},
    \qquad
    i=1,\ldots,n.
\]
Recall \eqref{ineq:xmui-bound}, for each component $i$ and every $k\geq K_i$,
\begin{align*}
    \dfrac{\kappa\mu^j}{2}
    \leq
    \dfrac{(x_{\mu^j}^*)_i}{2}
    \leq
    (x_{\mu^j}^{k+1})_i
    \leq
    \dfrac{3(x_{\mu^j}^*)_i}{2}
    \leq
    \dfrac{3\hat{M}}{2}.
\end{align*}
Combining this with the definition of $\Phi_{\mu^j}$, we obtain
\begin{align*}
    \dfrac{2\mu^j}{3\hat{M}^2}I
    \preceq
    \Phi_{\mu^j}(x_{\mu^j}^{k+1})
    \preceq
    \dfrac{2}{\mu^j\kappa^2}I.
\end{align*}
Since $\Phi_{\mu^j}$ is a positive definite diagonal matrix, its singular values coincide with its eigenvalues. Denoting them by $0<\gamma_1\leq\gamma_2\leq\cdots\leq\gamma_n$, it follows that
\[
    \dfrac{2\mu^j}{3\hat{M}^2}
    \leq
    \gamma_1
    \leq
    \gamma_n
    \leq
    \dfrac{2}{\mu^j\kappa^2}.
\]
Taking the inverse in the above matrix inequality yields
\begin{align*}
    \dfrac{\mu^j\kappa^2}{2}I
    \preceq
    \Phi_{\mu^j}^{-1}(x_{\mu^j}^{k+1})
    \preceq
    \dfrac{3\hat{M}^2}{2\mu^j}I.
\end{align*}
Consequently,
\begin{align*}
    \dfrac{\mu^j\kappa^2}{2}AA^\top
    \preceq
    A\Phi_{\mu^j}^{-1}(x_{\mu^j}^{k+1})A^\top
    \preceq
    \dfrac{3\hat{M}^2}{2\mu^j}AA^\top.
\end{align*}
Let $\sigma_{\min}(A)$ and $\sigma_{\max}(A)$ denote the smallest and largest singular values of $A$, respectively (which are independent of $\mu^j$). Denote by $0{<} \rho_1\leq\rho_2\leq\cdots\leq\rho_m$ the eigenvalues of $ A\Phi_{\mu^j}(x_{\mu^j}^{k+1})^{-1}A^\top.$ Then
\[
    \rho_1
    \geq
    \dfrac{\mu^j\kappa^2}{2}\sigma^2_{\min}(A)>0,
    \qquad
    \rho_m
    \leq
    \dfrac{3\hat{M}^2}{2\mu^j}\sigma^2_{\max}(A).
\]
Let $P$ be the orthogonal matrix
\[
    P
    :=
    \operatorname{diag}(I_n,-I_m)
\]
and define
\[
    \widetilde P_{\mu^j}
    :=
    P P_{\mu^j}(x_{\mu^j}^{k+1})
    =
    \begin{bmatrix}
        \Phi_{\mu^j}(x_{\mu^j}^{k+1}) & -A^\top\\
        -A                         & 0
    \end{bmatrix}.
\]
The matrix \(\widetilde P_{\mu^j}\) is symmetric. Moreover,
\begin{equation}\label{eq:J-Jtilde-equivalence}
    \sigma_{\min}
    \bigl(
        P_{\mu^j}(x_{\mu^j}^{k+1})
    \bigr)
    =
    \sigma_{\min}
    \bigl(
        \widetilde P_{\mu^j}
    \bigr), \quad    \sigma_{\max}
    \bigl(
        P_{\mu^j}(x_{\mu^j}^{k+1})
    \bigr)
    =
    \sigma_{\max}
    \bigl(
        \widetilde P_{\mu^j}
    \bigr),
\end{equation}
and
\[
    \left\|
        P_{\mu^j}(x_{\mu^j}^{k+1})^{-1}
    \right\|
    =
    \left\|
        \widetilde P_{\mu^j}^{-1}
    \right\|.
\]
By applying \cite[Theorem 1(c)]{EGIVALUE-ESTIMA} to
\(\widetilde P_{\mu^j}\), its eigenvalues satisfy

\[
\lambda^{\widetilde P_{\mu^j}}_i \in \left[
    \dfrac{-\rho_m}{\dfrac{1}{2}\left(1+\sqrt{1+4\dfrac{\rho_m}{\gamma_n}}\right)},
    \ \dfrac{-\rho_1}{1+\dfrac{\rho_1}{\gamma_1}}
\right]
\cup
\left[
    \gamma_1,
    \ \gamma_n\dfrac{1+\sqrt{1+4\dfrac{\rho_m}{\gamma_n}}}{2}
\right].
\]
Since \(\widetilde P_{\mu^j}\) is symmetric, its singular values
are the absolute values of its eigenvalues. Hence,
\[
    \sigma_{\min}(\widetilde P_{\mu^j})
    =
    \min_{1\leq i\leq n+m}
    \left|
        \lambda_i(\widetilde P_{\mu^j})
    \right|.
\]
It follows that
\[
  \sigma_{\min}
    \bigl(
        P_{\mu^j}(x_{\mu^j}^{k+1})
    \bigr)
    =
    \sigma_{\min}
    \bigl(
        \widetilde P_{\mu^j}
    \bigr)\geq\min\Big\{ \dfrac{\rho_1}{1+\dfrac{\rho_1}{\gamma_1}},\;\gamma_1\Big\}=\dfrac{\rho_1}{1+\dfrac{\rho_1}{\gamma_1}}=\dfrac{\rho_1\gamma_1}{\gamma_1+{\rho_1}}.
\]
Then
\[
\left\|
    P^{-1}_{\mu^j}(x_{\mu^j}^{k+1})
\right\|\leq \dfrac{\gamma_1+{\rho_1}}{\rho_1\gamma_1}=\dfrac{1}{\gamma_1}+\dfrac{1}{\rho_1}\leq \dfrac{3\hat{M}^2}{2\mu^j}+\dfrac{2}{\mu^j\kappa^2\sigma_{\min}(A)^2}=O((\mu^j)^{-1}).
\]
\end{proof}

Denote $z^{**}$ as the primal dual optimal point of problem~\eqref{p:stand-LP}. And letting $Z^{**}$ be the set of all such optimal points, i.e., $z^{**}\in Z^{**}$. Denote the fixed point reisdual of $\mathcal{T}_{\mu}$ as {$\hat{g}=z-\mathcal{T}_{\mu}z$.}
In the following, we incorporate the outer iteration into our analysis and prove that the outer iterates of Algorithm \ref{al:BPDHG} finally approach the optimal solution set of problem~\eqref{p:stand-LP}, {provided the inner stopping criteria is sufficiently satisfied.}

\begin{theorem}\label{th:innerstop}
Suppose that the assumptions of Lemma~\ref{LE:PDHG-Z-TZ} and
Lemma~\ref{le:uniform-bounded-barrier-solution} hold.
Suppose that, at each outer iteration $j$, the inner BPDHG loop is terminated at an index $k(j)$ such that
\[
\|\hat{g}^{j,k(j)}\|_Q\leq \delta^j,
\]
where $\{\delta^j\}_{j\geq0}$ is a sequence of positive scalars satisfying
\[
\eta_{\mu^j}\delta^j\to0
\quad\text{as }j\to\infty,
\]
where $\eta_{\mu^j}$ is defined in \eqref{inqe:dist-zk-Tzk}.
Then the sequence $\{z^{j+1,0}\}_{j\geq0}$ generated by Algorithm~\ref{al:BPDHG} approaches the solution set of problem~\eqref{p:stand-LP}, i.e.,
\[
\operatorname{dist}_Q(z^{j+1,0},Z^{**})\to0
\qquad\text{as }j\to\infty,
\]
\end{theorem}
\begin{proof}
Since $Z_{\mu}^*=\{z^*_{\mu}\}$, then
\begin{align}\label{ineq:z-zstar}
 \operatorname{dist}_Q(z^{j+1,0}, Z^{**}) \leq \|z^{j+1,0} - z^*_{\mu^j}\|_Q + \operatorname{dist}_Q(z^*_{\mu^j}, Z^{**})=\operatorname{dist}_Q(z^{j+1,0}, Z^{*}_{\mu^j})+\operatorname{dist}_Q(z^*_{\mu^j}, Z^{**}).
\end{align}
Considering the first term on the right-hand side of \eqref{ineq:z-zstar}, from the sharpness condition in Corollary~\ref{co-sharp}, we have for $z_{\mu^j}^*$,
    \begin{align}\label{ineq-zj-zstar}
        &\operatorname{dist}_Q(z^{j+1,0}, Z^{*}_{\mu^j})
        \leq  \eta_{\mu^j} \|z^{j,k}-\mathcal{T}_{\mu^j}(z^{j,k})\|_Q=\eta_{\mu^j} \|z^{j,k}-z^{j,k+1}\|_Q
        =\eta_{\mu^j} \|\hat{g}^{j,k}\|_Q.
    \end{align}
From Theorem \ref{Th:Z-ZSTAR}, we know that for the fixed $j$, the BPDHG iterations satisfy $\|z^{j,k}-z^{j,k+1}\|_Q=\|\hat{g}^{j,k}\|_Q\to 0,\;$as $k\to \infty$. Thus, for each $j>0$, there exists a finite index $k(j)$ such that $\|\hat{g}^{j,k(j)}\|_Q\leq \delta^j$. With this choice, \eqref{ineq-zj-zstar} shows
\begin{align}\label{ineq:zj+1zmustar}
     &\operatorname{dist}_Q(z^{j+1,0}, Z^{*}_{\mu^j})\leq  \eta_{\mu^j} \|\hat{g}^{j,k(j)}\|_Q \leq \eta_{\mu^j} \delta^j\to 0, \text{ as $j \to \infty$}.
\end{align}
Besides, let \(H(U)\) denote the Hoffman constant of problem~\eqref{p:stand-LP}, then for any $z\in \mathbb{R}_{++}^n\times \mathbb{R}^m$, it holds
\begin{align}\label{ineq:hoffman}
    \operatorname{dist}(z,Z^{**})
    &\leq H(U)\Big(\sqrt{\|Ax-b\|^2+\|(A^{\top}y-c)_+\|^2+(c^{\top}x-b^{\top}y)_+^2}\Big):=H(U)\|(l-Uz)_+\|,
\end{align}
where \[
U :=
\begin{bmatrix}
-A  & 0 \\
A   & 0 \\
0   & -A^\top \\
-c^\top & b^\top
\end{bmatrix},
\qquad
l :=
\begin{bmatrix}
-b \\
b \\
-c \\
0
\end{bmatrix}.
\]
Letting $z=z_{\mu^j}^*$, recall \eqref{kkt-barrier}, we have
$Ax^*_{\mu^j}=b,\; c-A^{\top}y^*_{\mu^j}\geq 0,\; c^{\top}x^*_{\mu^j}-b^{\top}y^*_{\mu^j}=n\mu^j.$ Substitute $z_{\mu^j}^*$ into \eqref{ineq:hoffman}, yields,
\begin{align}\label{ineq:zstarmuj}
     \operatorname{dist}(z^*_{\mu^j},Z^{**})\leq nH(U)\mu^j.
\end{align}
As \eqref{ineq:zstarmuj} can be transformed into,
{
\begin{align*}
    \dfrac{1}{\sqrt{\sigma_{\max}(Q)}} \operatorname{dist}_Q(z^*_{\mu^j},Z^{**})\leq \operatorname{dist}(z^*_{\mu^j},Z^{**})\leq nH(U)\mu^j.
\end{align*}
Hence, we have $\operatorname{dist}_Q(z^*_{\mu^j},Z^{**})\leq\sqrt{\sigma_{\max}(Q)}nH(K)\mu^j.$ Combining this estimate with
\eqref{ineq:z-zstar} and \eqref{ineq-zj-zstar} yields
\begin{align*}
   \operatorname{dist}_Q(z^{j+1,0},Z^{**})\leq \operatorname{dist}_Q(z^{j+1,0}, Z^{*}_{\mu^j})+\operatorname{dist}_Q(z^*_{\mu^j}, Z^{**})\leq \eta_{\mu^j}\delta^j+ \sqrt{\sigma_{\max}(Q)}nH(U)\mu^j.
\end{align*}
In Algorithm \ref{al:BPDHG}, we have $\mu^{j+1}=\theta \mu^j$ with $\theta\in(0,1)$. Then,
\begin{align*}
   \operatorname{dist}_Q(z^{j+1,0},Z^{**})\leq \eta_{\mu^j}\delta^j+\sqrt{\sigma_{\max}(Q)} nH(U)\mu^j\leq \eta_{\mu^j}\delta^j+ \sqrt{\sigma_{\max}(Q)}nH(U)\theta^{j}\mu_0.
\end{align*}
Since when $j\to\infty$, we have $\eta_{\mu^j}\delta^j\to 0$ and
$\mu^{j}=\theta^j\mu_0\to0$ for $\theta\in(0,1)$, it follows that
\begin{align*}
\operatorname{dist}_Q(z^{j+1,0},Z^{**})\to0
\quad\text{as }j\to\infty.
\end{align*}
}
\end{proof}
\begin{remark}\label{re:deltachoose}
The condition
\[
\eta_{\mu^j}\delta^j\to0, \quad\text{as $j\to \infty$}
\]
is a sufficient theoretical requirement for guaranteeing the conclusion of Theorem~\ref{th:innerstop}.
If the growth rate of $\eta_\mu$ is known, it provides guidance for selecting the inner stopping tolerance $\delta^j$. In general, the exact growth rate of \(\eta_{\mu^j}\) may be difficult
to determine. Nevertheless, it is sufficient to obtain an upper bound
on its growth.  More precisely, suppose that for some \(q>0\),
\(
\eta_{\mu^j}=O(({\mu^j})^{-q}),
\)
one may choose
\[
\delta^j=\widehat C(\mu^j)^{q+\xi},
\qquad \xi>0,
\]
where $\widehat C>0$ is a constant. Then \[
\eta_{\mu^j}\delta^j=O(({\mu^j})^{\xi})\to 0, \quad \text{as $j\to\infty$.}
\]
In particular, Theorem~\ref{th:jmu-inverse-bound} shows that, for every
\(0<\mu^j\leq\mu^0\), there exists \(K_{\mu^j}\in\mathbb{N}\) such
that, for all \(k\geq K_{\mu^j}\),
\[
    \left\|
        P^{-1}_{\mu^j}(x_{\mu^j}^k)
    \right\|
    \leq
    \frac{C_J}{\mu^j}.
\] 
It implies that, once the inner iterates are
sufficiently close to the barrier solution, the corresponding local
error bound admits the worst case estimate
\[
\eta_{\mu^j}=O((\mu^j)^{-1}).
\]
Consequently, within this local region, the choice
\[
\delta^j=O((\mu^j)^{1+\xi}),\qquad \xi>0,
\]
is sufficient to ensure
\(\eta_{\mu^j}\delta^j\to0\).



\end{remark}
{
\begin{lemma}\label{lem:inner-complexity}
Suppose that the assumptions of Theorem~\ref{th:innerstop} hold and that
$\{z^{j,k}\}$ is generated by the update framework of
Algorithm~\ref{al:BPDHG}. Let
\[
k(j+1):=
\min\left\{
k\mid
\|\hat g^{j+1,k}\|_Q\leq\delta^{j+1}
\right\}
\]
denote the inner stopping index for the barrier subproblem corresponding
to $\mu^{j+1}$, where $\mu^{j+1}$ satisfies the update \eqref{update-mu}.
By Remark~\ref{re:deltachoose}, assume further that the point \(z^{j+1,0}\) lies in the local region where
\(
\eta_{\mu^j}=O\bigl((\mu^j)^{-1}\bigr),
\;
\delta^j=\widehat C(\mu^j)^{1+\xi},
\; \xi>0.
\)
Then,
\[
k(j+1)
\leq
\left\lceil
\left(
O\bigl((\mu^j)^{-1}\bigr)
+
O\bigl((\mu^j)^{-\xi}\bigr)
\right)^2
\right\rceil.
\]
\end{lemma}

\begin{proof}
By Theorem~\ref{Th:Z-ZSTAR}, for every $k\geq0$,
\(
\|\hat g^{j+1,k}\|_Q
\leq
\frac{1}{\sqrt{k+1}}
\|z^{j+1,0}-z^{j+1,*}\|_Q.
\)
Therefore, the stopping condition
\(
\|\hat g^{j+1,k}\|_Q\leq\delta^{j+1}
\)
is guaranteed whenever
\[
k+1
\geq
\left(
\frac{\|z^{j+1,0}-z^{j+1,*}\|_Q}
{\delta^{j+1}}
\right)^2.
\]
Since $k(j+1)$ is the first index satisfying the stopping condition, it
follows that
\[
k(j+1)
\leq
\left\lceil
\left(
\frac{\|z^{j+1,0}-z^{j+1,*}\|_Q}
{\delta^{j+1}}
\right)^2
\right\rceil.
\]We next estimate
\(\|z^{j+1,0}-z^{j+1,*}\|_Q\). Since
\[
\eta_{\mu^j}\leq\zeta(\mu^j)^{-1},
\qquad
\delta^j=\widehat C(\mu^j)^{1+\xi},
\qquad
\delta^{j+1}
=
\widehat C\theta^{1+\xi}(\mu^j)^{1+\xi},
\]
where constant \(\zeta>0\). From the update framework of
Algorithm~\ref{al:BPDHG}, together with
\eqref{inqe:dist-zk-Tzk}, \eqref{ineq:zstarmuj}, the inner stopping
condition, and \(\mu^{j+1}=\theta\mu^j\), gives
\begin{align*}
\frac{\|z^{j+1,0}-z^{j+1,*}\|_Q}{\delta^{j+1}}
&\leq
\frac{
\|z^{j,k(j)+1}-z^{j,*}\|_Q
+
\|z^{j,*}-z^{j+1,*}\|_Q
}{
\delta^{j+1}
}
\\
&\leq
\frac{
\|z^{j,k(j)+1}-z^{j,*}\|_Q
+
\|z^{j,*}-z^{**}\|_Q
+
\|z^{j+1,*}-z^{**}\|_Q
}{
\delta^{j+1}
}
\\
&\leq
\frac{
\eta_{\mu^j}
\|z^{j,k(j)}-\mathcal{T}_{\mu^j}(z^{j,k(j)})\|_Q
+
nH(U)(\mu^j+\mu^{j+1})
}{
\delta^{j+1}
}
\\
&\leq
\frac{\eta_{\mu^j}\delta^j}{\delta^{j+1}}
+
\frac{nH(U)(1+\theta)\mu^j}{\delta^{j+1}}
\\
&\leq
\frac{
\zeta(\mu^j)^{-1}\widehat C(\mu^j)^{1+\xi}
}{
\widehat C\theta^{1+\xi}(\mu^j)^{1+\xi}
}
+
\frac{
nH(U)(1+\theta)\mu^j
}{
\widehat C\theta^{1+\xi}(\mu^j)^{1+\xi}
}
\\
&=
\frac{\zeta}{\theta^{1+\xi}}(\mu^j)^{-1}
+
\frac{nH(U)(1+\theta)}
{\widehat C\theta^{1+\xi}}
(\mu^j)^{-\xi}
\\
&=
O\bigl((\mu^j)^{-1}\bigr)
+
O\bigl((\mu^j)^{-\xi}\bigr).
\end{align*}

Substituting this estimate into the upper bound for \(k(j+1)\) yields
\[
k(j+1)
\leq
\left\lceil
\left(
O\bigl((\mu^j)^{-1}\bigr)
+
O\bigl((\mu^j)^{-\xi}\bigr)
\right)^2
\right\rceil,
\]
which completes the proof.
\end{proof}
}
Using the reasoning carried out until this moment, we can give and $\varepsilon$-accuracy inner and outer iteration bound:

\begin{proposition}[Finite-accuracy inner and outer iteration bound]
Suppose that the assumptions of Lemma~\ref{lem:inner-complexity} hold. 
{For a prescribed
accuracy \(\iota>0\)}, 
stop the outer loop at the first index \(J_\iota\) such that
\[
    \mu_{J_\iota}\le C_\mu\iota ,
\]
where \(C_\mu>0\) is independent of \(\iota\). Then
\[
    J_\iota
    =
    \left\lceil
    \frac{\log(\mu_0/(C_\mu\iota))}{\log(1/\theta)}
    \right\rceil
    =
    O(\log(1/\iota)).
\]
Assume moreover that the inner loop at outer iteration \(j\) is terminated when
\[
    \|\hat g_{\mu_{j+1}}^{k}\|_Q\le \delta_{j+1},
    \qquad
    \delta_{j+1}={\widehat{C}}\mu_{j+1}^{1+\xi},
    \qquad \xi>0 .
\]
Then, for every \(j<J_\iota\), there exists a constant
\(\widehat C_{\rm in}>0\), independent of \(j\), such that
\[
    k(j+1)
    \le
    \widehat C_{\rm in}
    \left(
        \iota^{-1}+\iota^{-\xi}
    \right)^2 .
\]
Consequently, the total number of inner BPDHG iterations required before
termination satisfies
\[
    \sum_{j=0}^{J_\iota-1} k(j+1)
    \le
    \widehat C_{\rm tot}
    \log(1/\iota)
    \left(
        \iota^{-1}+\iota^{-\xi}
    \right)^2 ,
\]
for a constant \(\widehat C_{\rm tot}>0\) independent of \(\iota\).
\end{proposition}

\section{Experiments on LP Instances}\label{SEC:EXPER}

In this section, we evaluate the proposed Barrier PDHG method on a range of linear programming instances. The numerical study consists of two parts. First, we compare PDHG with Barrier PDHG on several small-scale LP instances to illustrate the effect of the barrier term. We then incorporate the barrier mechanism into the PDLP framework \cite{google2022p} and evaluate its numerical performance on a collection of benchmark LP instances. Besides, we further analyse the numerical results by introducing a theoretically motivated problem dependent measure and an additional empirical indicator, both of which help identify LP instances for which BPDLP is more likely to outperform PDLP.
All experiments are conducted using \textit{Julia 1.6.5} on the University of Southampton's Iridis 6 high-performance computing cluster.

Throughout the numerical experiments, all algorithms are terminated
once the KKT conditions are satisfied to a prescribed
accuracy. For a primal-dual slack point $(x,y,s_1)$, we first consider the
relative KKT residual
\begin{equation*}
    \mathrm{KKT}
    =
    \max\left\{
        \frac{\|Ax-b\|}{1+\|b\|},
        \frac{\|c-A^\top y-s_1\|}{1+\|c\|},
        \frac{|c^\top x-b^\top y|}{1+|c^\top x|+|b^\top y|}
    \right\}.
\end{equation*}
Since the dual slack variable $s_1$ in~\eqref{kkt-lp} is not stored explicitly in the
implementation, we reconstruct it from the dual multiplier by setting
\[
    s_1:=c-A^\top y.
\]
We then define
\begin{align}\label{eq:lp-stopping}
\operatorname{KKT}_{\mathrm{exLP}}(x,y)
:=
\max\left\{
    \frac{\|Ax-b\|}{1+\|b\|},
    \frac{\|(s_1)_-\|}{1+\|c\|},
    \frac{|c^\top x-b^\top y|}
         {1+|c^\top x|+|b^\top y|}
\right\}.
\end{align}
and terminate the algorithm when
\[
    \operatorname{KKT}_{\mathrm{exLP}}(x,y)
    \leq \epsilon,
\]
where \(\epsilon>0\) is the prescribed tolerance.

For the barrier based methods, Remark~\ref{re:deltachoose} suggests that a theoretically sufficient inner tolerance is
$\delta_j
    =
    O\bigl((\mu^j)^{1+\xi}\bigr),$ with 
\(
    \xi>0.
\)
This rule is derived from the worst case estimate
\(\eta_{\mu^j}=O((\mu^j)^{-1})\), and may lead to excessive inner
iterations in practice.

Therefore, in the numerical implementation, we instead use {a} less restrictive
criterion and terminate the inner iteration at the
first index \(k(j)\) satisfying
\begin{align}\label{ineq:stop-innerbarrier}
     \operatorname{KKT}_{\mathrm{exBarrier}}(x_{\mu^j}^k,y_{\mu^j}^k)
    \leq
    \widetilde C\mu^j,
\end{align}
where \(\widetilde C>0\) is a prescribed constant and
\[
    \operatorname{KKT}_{\mathrm{exBarrier}}(x_{\mu^j}^k,y_{\mu^j}^k)
    :=
    \max\left\{
        \frac{\|Ax_{\mu^j}^k-b\|}{1+\|b\|},
        \frac{\|c-A^\top y_{\mu^j}^k-\mu^j(X_{\mu^j}^k)^{-1}\mathbf{e}\|}
             {1+\|c\|}
    \right\}.
\] This practical
choice is not covered by the preceding worst case convergence
argument, and has been designed to avoid solving the barrier subproblem to unnecessarily
high accuracy.

It is important to note that we use the barrier KKT residual rather than the difference between two consecutive iterates considered in Theorem~\ref{th:innerstop}. The main reason is that a small difference between consecutive fixed point iterates does not guarantee that the current iterate is sufficiently close to the KKT solution set. Such a conclusion generally requires an additional error bound relating the fixed point residual to the distance from the solution set, together with suitable
control of the associated error bound constant. In contrast, the barrier KKT residual directly measures the extent to which the primal feasibility, dual feasibility, and perturbed complementarity conditions are satisfied.

\subsection{Diagnostic behaviour on small LPs}
{
Before considering larger scale LP instances, we first use a simple linear programming problem to showcase the characteristic behavior of the barrier PDHG mechanism here proposed. And indeed, the purpose of this experiment is to visualize the projection behavior in the iterate trajectory, and to illustrate how the barrier update changes this phenomenon.
} To this aim, we compare four variants of PDHG: standard PDHG, Barrier PDHG (BPDHG), restarted PDHG (rPDHG), and restarted Barrier PDHG (rBPDHG), on the following LP problem, which has the same form as problem~\eqref{p:barrLP}:
\begin{align}\label{pro:simple}
    \min_{x_1,x_2}\;\;& -2x_1 - x_2\nonumber\\
    \text{s.t. }\;& x_1 + 2x_2 = 3,
    \\& x_1\geq 0,\; x_2
    \geq 0.\nonumber
\end{align}
For a fair comparison, all four methods are initialized from the same
point and use the same primal and dual step sizes. Their overall
termination is assessed using the criterion \eqref{eq:lp-stopping}, while the barrier-based variants additionally
use the inner stopping criterion described above. Figure~\ref{fig:simple_test} shows the iteration trajectories generated by the four methods.

\begin{figure}
    \centering
    \includegraphics[width=0.4\linewidth]{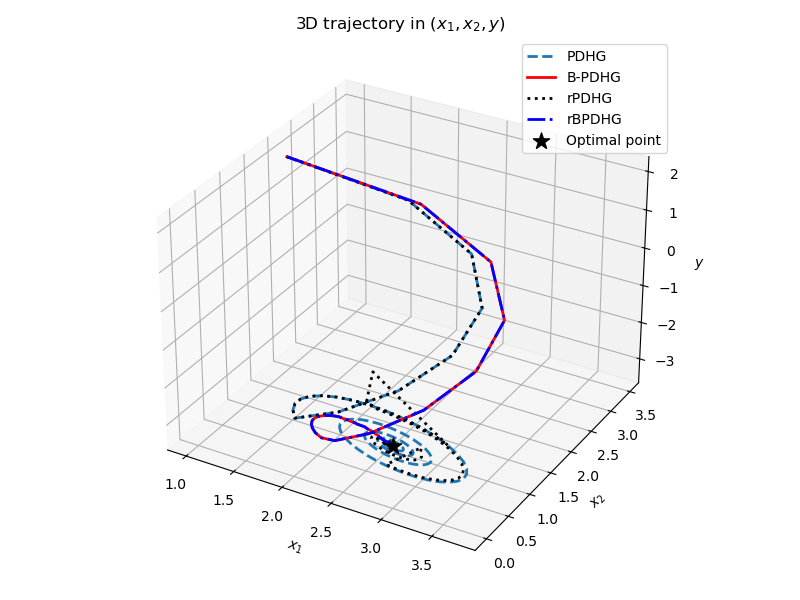}
    \includegraphics[width=0.4\linewidth]{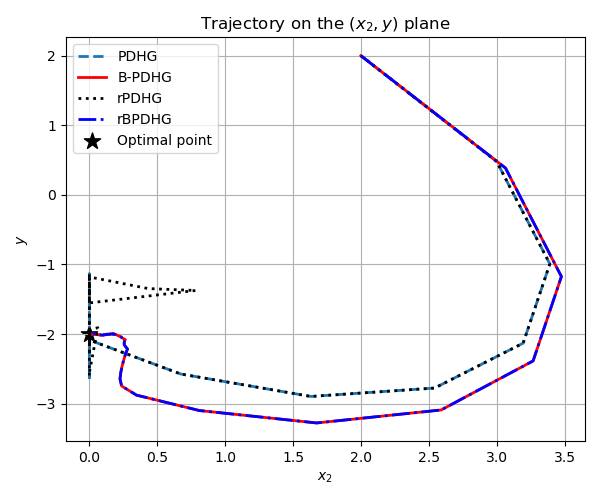}
    \includegraphics[width=0.4\linewidth]{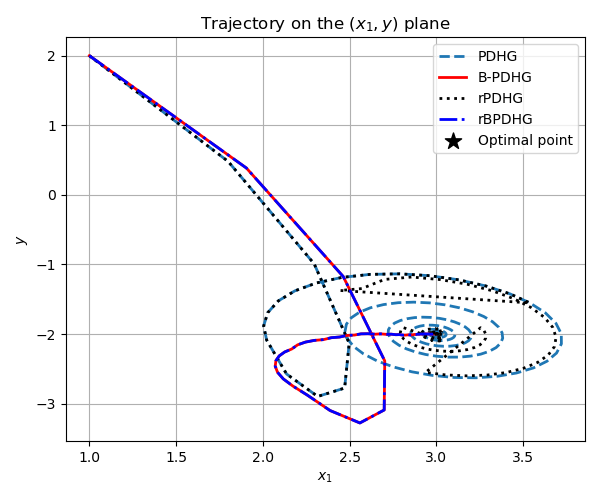}
    \includegraphics[width=0.4\linewidth]{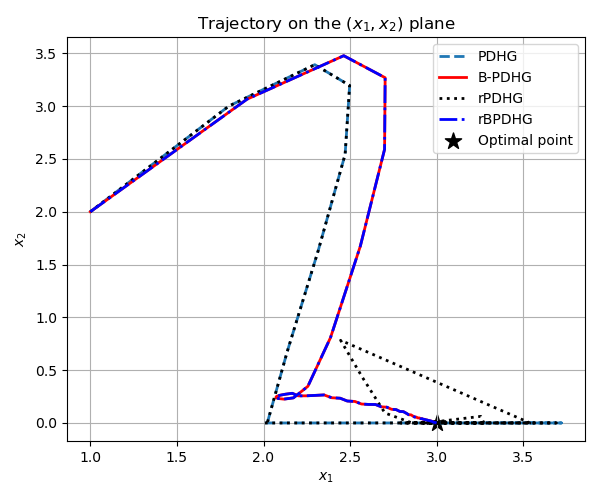}

    \caption{Trajectory comparison between PDHG, BPDHG, rPDHG (fix restart frequency) and rBPDHG in 3D, together with their projected side views. Optimal point $(3,0,-2)$.
   We choose the step size of 0.2 for all methods and set the stopping tolerance to $10^{-12}$. 
    For Barrier PDHG, the barrier parameter is updated according to
    $
    \mu^{k+1}=0.85\mu^k.$
    For the fixed-restart variants, we set the restart length to 20. In this experiment, PDHG converges after 1282 iterations, BPDHG converges after 332 iterations, restarted BPDHG also converges after 332 iterations, and restarted PDHG converges after 507 iterations.}
    \label{fig:simple_test}
    
\end{figure}
{
From Figure~\ref{fig:simple_test}, we observe that the methods incorporating the barrier strategy exhibit a more direct trajectory towards the optimal solution than their counterparts without the barrier term. In particular, standard PDHG and rPDHG show a clear spiral behaviour near the solution. Although the primal iterates approach the boundary
\(x_2=0\) relatively early, the primal and dual variables continue to
oscillate because of their coupling, and the iterates repeatedly move
around the optimal point before convergence.
}

{
In contrast, BPDHG and rBPDHG maintain strict positivity
during each barrier subproblem and approach the solution along a
smoother and more direct trajectory. The barrier term prevents the
primal iterate from reaching the boundary too aggressively at an early
stage and substantially reduces the spiral and zig-zag behaviour near
the solution. 
} We also provide the corresponding KKT behavior for these four methods in Figure~\ref{fi:kktsimple}.
\begin{figure}
    \centering
    \includegraphics[width=0.4\linewidth]{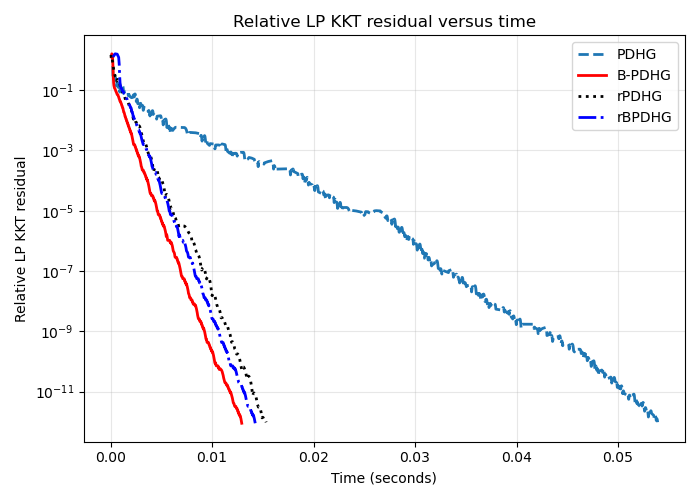}
    \includegraphics[width=0.4\linewidth]{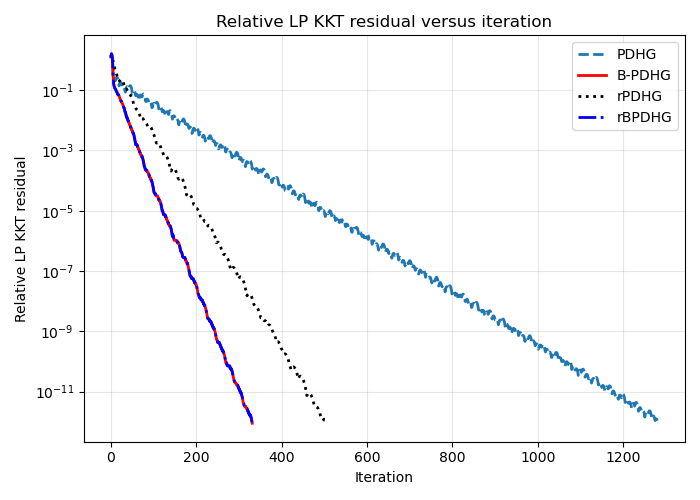}
    \caption{Comparison of the KKT residual trajectories of PDHG, rPDHG, BPDHG, and rBPDHG for problem~\eqref{pro:simple} in terms of computing time and iteration count.}\label{fi:kktsimple}
\end{figure}
Although the residual curves do not exhibit a plateau in this example, the barrier variants still display a favorable convergence behavior. In particular, they reach the prescribed tolerance with fewer iterations.

\subsection{BPDHG versus PDHG on selected Netlib instances}

{
The preceding experiment uses a small-scale LP instance to illustrate the different behaviours of PDHG and its barrier variant. However, the corresponding KKT residual does not exhibit a pronounced plateau, possibly due to the small size and simple structure of the problem. To further investigate this phenomenon, we extend the experiments to two larger scale LP instances selected from the Netlib collection \cite{NETLIB}. The purpose of this subsection is to examine whether the plateau behaviour of PDHG discussed at the beginning of this paper can be alleviated by the barrier mechanism.
}

{Both instances are first preprocessed using Ruiz rescaling followed by $\ell_2$-norm  rescaling~\cite{google2022p}. As in the previous experiment, both methods in this subsection are initialized from the same point and use identical primal and dual step sizes. With the stopping tolerance set to \(10^{-4}\), the corresponding KKT residual trajectories, plotted against computing time and iteration count are shown in Figure~\ref{fig:larger_lp_kkt}.
}


\begin{figure}
    \centering
\includegraphics[width=0.4\linewidth]{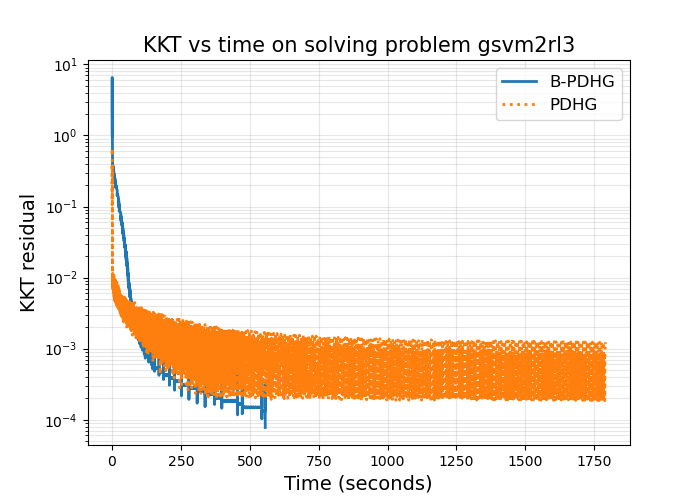}
\includegraphics[width=0.4\linewidth]{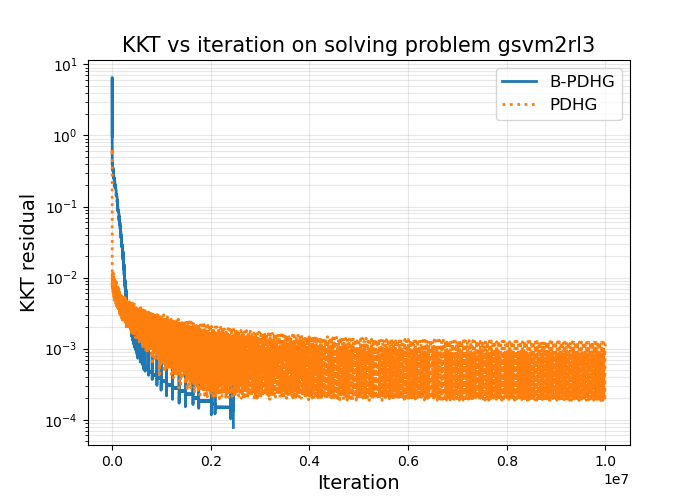}
\includegraphics[width=0.4\linewidth]{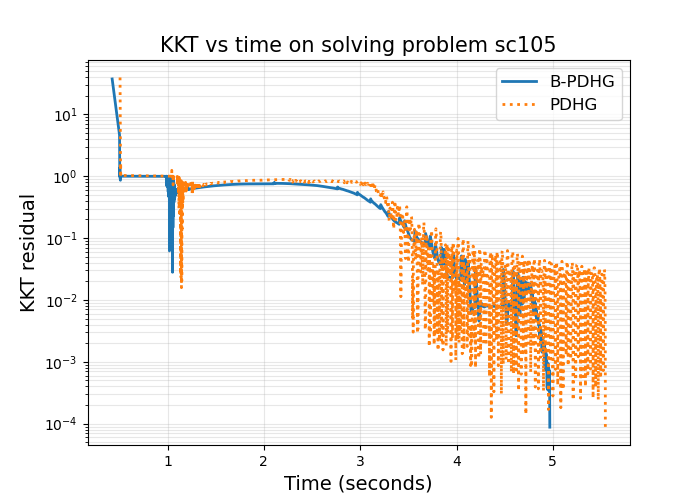}
\includegraphics[width=0.4\linewidth]{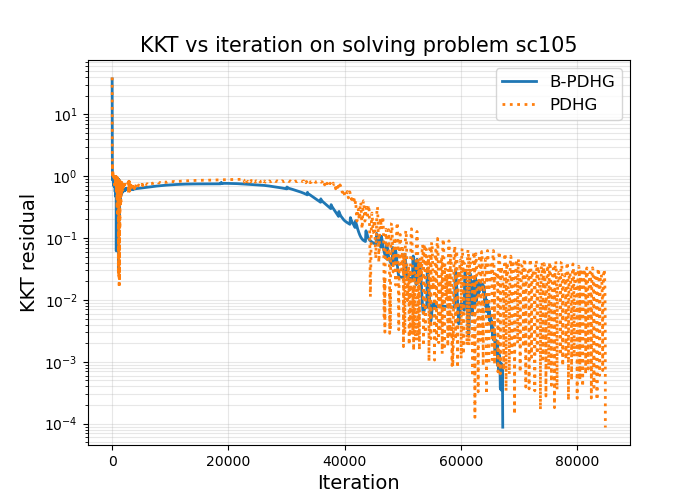}
\caption{Comparison of the KKT residuals of PDHG and Barrier PDHG on
the instances \texttt{gsvm2r13} and \texttt{sc105}, plotted against
computing time and iteration count. For Barrier PDHG, the barrier
parameter is updated according to
\(\mu^{k+1}=0.9\mu^k\).}\label{fig:larger_lp_kkt}
\end{figure}
{
Figure~\ref{fig:larger_lp_kkt} shows that BPDHG reduces the KKT residual more rapidly than standard PDHG on both selected instances, in terms of both iteration count and computing time. For \texttt{gsvm2rl3}, PDHG exhibits a prolonged plateau, with the KKT residual oscillating around $10^{-3}$, whereas BPDHG maintains an overall decreasing trend and reaches the target tolerance much earlier. For \texttt{sc105}, both methods experience a relatively long plateau initially. However, BPDHG leaves this plateau earlier than PDHG, while PDHG continues to exhibit an oscillatory behavior in the later stages of convergence.
}

These observations suggest that the barrier mechanism can alleviate, but not eliminate the plateau behavior of PDHG.

\subsection{Barrier PDLP (BPDLP) on MIPLIB}\label{subse:Comparison with PDLP on the MIPLIB Benchmark}

{
The preceding experiments show that the barrier mechanism can improve the convergence behavior of PDHG and rPDHG on selected instances. This motivates us to investigate whether the same mechanism remains effective when embedded in the practical PDLP framework~\cite{google2022p}.
}

{
PDLP can be viewed as a rPDHG method enhanced with additional components, including adaptive step sizes, primal weight updates, and presolve procedures. The purpose of this subsection is therefore to evaluate the barrier mechanism within this more comprehensive framework. In particular, we incorporate the barrier technique into the PDLP framework directly, and compare the resulting method with the original PDLP on solving MIPLIB benchmark instances \cite{MILP}. 
}

{
We emphasize that the aim is not to claim that BPDLP would uniformly outperform PDLP. Instead, we investigate whether the barrier modification can improve performance on a meaningful subset of instances for which PDLP exhibits prolonged stagnation in the KKT residual.
}


{
We summarize the complete framework in Algorithm~\ref{al:BPDLP}, and refer to the resulting method as BPDLP. Specifically, apart from the additional outer loop for updating the barrier parameter and the replacement of the projected primal update by the barrier counterpart, the main components of PDLP are retained.
}


\begin{algorithm}[]
\caption{Barrier PDLP (BPDLP)}\label{al:BPDLP}
\begin{algorithmic}[1]
\State \textbf{Input:} Initial solution $z^{0,0,0}$, $\mu^0 > 0$, $\theta \in (0,1)$, $0<\theta<\tau_{\text{inner}} <1$.
\State Initialize $j \leftarrow 0$, primal weight $\omega^{0,0} \leftarrow \texttt{InitializePrimalWeight}(c,b)$.
\Repeat \Comment{Outer loop: barrier parameter $\mu$}
    \State $n \leftarrow 0$;
    \Repeat \Comment{Middle loop}
        \State $t \leftarrow 0$;
        \Repeat \Comment{Inner loop}
            \State $z^{j,n,t+1},\, \eta^{j,n,t+1} \leftarrow \texttt{AdaptiveStepOfBarrierPDHG}(z^{j,n,t},\, \omega^{j,n},\, \hat{\eta}^{j,n,t},\, \mu^j,\;k)$;\label{BPDLP:z-x}
            \State $\bar{z}^{j,n,t+1} \leftarrow \frac{1}{\sum_{i=1}^{t+1}\eta^{j,n,i}}\sum_{i=1}^{t+1}\eta^{j,n,i}z^{j,n,i}$;
            \State $z_c^{j,n} \leftarrow \texttt{GetRestartCandidate}(z^{j,n,t},\, \bar{z}^{j,n,t},\, z^{j,n,0})$;
            \State $t \leftarrow t+1,\; k\leftarrow k+1$;
        \Until{restart or LP termination criteria holds}\label{ineq:inner_stop}
        \State $z^{j,n+1,0} \leftarrow z_c^{j,n}$,
        \State $\omega^{j,n+1} \leftarrow \texttt{PrimalWeightUpdate}(z^{j,n+1,0},\, z^{j,n,0},\, \omega^{j,n})$;
        \State $n \leftarrow n+1$;
  \Until{inner stopping condition holds \textbf{or} max inner iterations reached \textbf{or} LP termination criteria holds}
\If{LP termination criteria holds}
    \State \textbf{break}
\EndIf
\If{inner stopping condition holds}
    \State $\mu^{j+1} \leftarrow \theta \mu^j$ \Comment{Inner barrier subproblem solved}
    \State $z^{j+1,0,0} \leftarrow \text{Predictor}_x(z^{j,n,0}, \mu^{j+1})$ \Comment{Predictor step for new $\mu$}\label{al:predx}
\Else
    \State $\mu^{j+1} \leftarrow \tau_{\text{inner}} \mu^j$\label{update-innerfail} \Comment{Max inner iterations reached}
    \State $z^{j+1,0,0} \leftarrow z^{j,n,0}$
\EndIf
\State $j \leftarrow j+1$
\Until{termination criteria holds}
\State \Return $z^{j+1,0,0}$
\end{algorithmic}
\end{algorithm}

{Before presenting the comparison of BPDLP with PDLP, in the following remarks, we explain in detail the algorithmic choices carried out in our implmentation.}

\begin{remark}[Adaptive update of the barrier parameter]
Structurally, Algorithm~\ref{al:BPDLP} can be interpreted as replacing the PDHG update in Line~\ref{al:BPDL4} of Algorithm~\ref{al:BPDHG} with its barrier counterpart within
the PDLP framework. In addition to this modification, BPDLP introduces
an adaptive update rule for the barrier parameter \(\mu\), rather than
using the same reduction factor throughout the algorithm.

More specifically, if the inner problem associated with the current
barrier parameter \(\mu^j\) is solved to the prescribed accuracy, then
the barrier parameter is updated using the standard reduction rule.
Otherwise, we set
\[
    \mu^{j+1}
    \leftarrow
    \tau_{\mathrm{inner}}\mu^j,
    \qquad
    \theta<\tau_{\mathrm{inner}}<1,
\]
so that the barrier parameter is reduced {at a slower pace}.

The motivation for this strategy is that insufficient progress in the
inner iterations may indicate that the reduction in the barrier
parameter is too aggressive. In this case, the current warm start may
be too far from the solution associated with the new barrier
parameter, making the corresponding barrier subproblem more difficult
to solve. A slower reduction of \(\mu\) is therefore used.
\end{remark}

\begin{remark}[Practical inner stopping criterion]
In the BPDLP framework, the original inner stopping criterion
\eqref{ineq:stop-innerbarrier} is replaced by
\begin{align}\label{barrierinnerkkt}
    \operatorname{KKT}_{\mathrm{cxBarrier}}
    \bigl(x_{\mu^j}^k,y_{\mu^j}^k\bigr)
    \leq
    \max\left\{
        \widetilde C\mu^j,\,
        \vartheta\operatorname{KKT}_{\mathrm{exLP}}
    \right\}, \quad \text{where $\vartheta>0$,}  
\end{align}
where $\operatorname{KKT}_{\mathrm{exLP}}$ is  computed to the output of Step~\ref{ineq:inner_stop} of Algorithm~\ref{al:BPDLP} (corresponding to the inner barrier termination.)
In practice, for some LP instances, the criterion based solely on
\(\widetilde C\mu^j\) as in \eqref{ineq:stop-innerbarrier}, can still be overly restrictive. In particular, in this scenario, the
inner barrier subproblem might be solved to high accuracy even when
this does not lead to a meaningful reduction in
the KKT residual of the original LP, see \eqref{eq:lp-stopping}. This can hence result in excessive
inner iterations without corresponding progress in the outer
iteration. To avoid such an over-solving issue, we additionally relate the required inner
accuracy to the current KKT residual of the original LP. Under
\eqref{barrierinnerkkt}, the inner iteration is terminated once its
barrier KKT residual is sufficiently small relative 
to either the
current barrier parameter or to the current outer KKT residual. Thus,
as the outer KKT residual decreases, the required inner accuracy is
also tightened accordingly.
\end{remark}
\begin{remark}[Favorable computational structure]
As shown in Algorithm~\ref{al:BPDLP}, apart from the additional outer loop for updating the barrier parameter, the main algorithmic structure remains unchanged. 
The only modification in the PDHG step is the primal update in Step~\ref{BPDLP:z-x}: the standard projection is replaced by the barrier update. 
Moreover, the barrier update has a closed form solution and can be computed separately for each coordinate. Therefore, the proposed method can be implemented in parallel and remains suitable for large-scale LP problems.
\end{remark}

\begin{remark}[{Central-path predictor}]
The primal predictor step used at Step~\ref{al:predx} of Algorithm \ref{al:BPDLP} takes the form \[
\widehat x^{\,j+1}
=
\psi_{\mu_{j+1}}^{-1}(q^j),
\]
where 
\[
\psi_\nu(t):=t-\frac{\tau\nu}{t},
\]
which is strictly increasing on \(\mathbb R_{++}\) and, 
\[
q^j
=
x(\mu_j)+\tau(A^\top y(\mu_j)-c)
=
\psi_{\mu_j}(x(\mu_j)).
\]
Using $
\psi_{\mu_{j+1}}(\widehat x^{\,j+1})
=
\psi_{\mu_j}(x(\mu_j))$ and  the central-path asymptotics, it is possible to prove
\[
\widehat x_i^{\,j+1}
=
x_i(\mu_j)+O(\mu_j),
\qquad i\in B,
\]
whereas
\[
\widehat x_i^{\,j+1}
=
\theta x_i(\mu_j)\bigl(1+O(\mu_j)\bigr),
\qquad i\in N,
\]
where
\[
B:=\{i:x_i^\star>0\},
\qquad
N:=\{i:s_i^\star>0\}
\]
is the strictly complementary optimal partition. Hence the predictor leaves the basic variables essentially unchanged while
rescaling the nonbasic variables according to the leading-order central-path
motion.
\end{remark}

We now turn to the comparison between BPDLP and PDLP.
We compare BPDLP with the original PDLP on LP relaxations from the MIPLIB 2017 benchmark set~\cite{MILP}, using the same test collection as in~\cite{google2022p}. Our implementation is built on the original PDLP codebase. Apart from the barrier term and the corresponding modifications to the outer and inner stopping framework, we retain the main acceleration components of PDLP, including adaptive step size selection, adaptive restart, primal weight updates, presolve, and problem scaling. For a fair comparison, all instances are processed using the same presolve procedure, and both methods are applied to exactly the same presolved problems. PDLP and BPDLP are initialized from the same primal and dual points, with the primal variable set to the all-ones vector and the dual variable set to zero.

As for the hyperparameters, for BPDLP, we set the barrier reduction factor to \(\theta=0.2\).
In the inner stopping criterion~\eqref{barrierinnerkkt}, we choose
\(\vartheta=0.05\), and the maximum number of inner iterations is set
to \(80{,}000\). If the inner stopping criterion is not satisfied
within this limit, the barrier parameter is reduced more conservatively
according to Step~\ref{update-innerfail}, with
\(\tau_{\mathrm{inner}}=0.9\). For both PDLP and BPDLP, we use the original LP termination
criterion~\eqref{eq:lp-stopping} with tolerance \(10^{-4}\).
The maximum total number of iterations is set to \(800{,}000\), and the time
limit is set to one hour per instance.

\begin{table}[htbp]
\centering
\caption{Summary of solved instances with tolerance $10^{-4}$.}
\label{tab:solved_instances}

\begin{tabular}{lc}
\hline
Category & Count \\
\hline
Total instances & 381 \\
Solved by PDLP & 338 \\
Solved by BPDLP & 333 \\
\hline
\end{tabular}

\vspace{0.5em}

\begin{minipage}{0.8\textwidth}
\footnotesize
The instances solved only by BPDLP are
\texttt{hgms30}, \texttt{neos-4292145-piako}, \texttt{rmine21}, and
\texttt{unitcal\_7}.
\end{minipage}
\label{table-381}
\end{table}

Table~\ref{table-381} summarizes the number of instances successfully solved by PDLP and BPDLP on the full dataset, separately. The original dataset contains
383 instances, but two failed during presolve, leaving 381
instances for the comparison.
Specifically, among the 381 test instances, BPDLP solves 333 instances, while PDLP solves 338 instances. 
{
Overall, BPDLP requires less computing time than PDLP on 109
instances and successfully solves four instances that are not solved
by PDLP within the prescribed limits. These results indicate that,
although the barrier modification does not improve the
performance uniformly, it can provide a meaningful advantage on a
subset of difficult instances.
}

{
To investigate the source of this advantage, we select six representative instances from the 109 instances on which BPDLP is faster than PDLP. For these instances, Figures~\ref{fig:selected-bpdlp-better} and~\ref{fig:selected-bpdlp-better2} show the KKT residual trajectories against computing time and iteration count. The purpose of this comparison is to examine whether the performance advantage of BPDLP is associated with the prolonged plateau behaviour of PDLP discussed above. 
}


{
Across these representative instances, as confirmed by Figures~\ref{fig:selected-bpdlp-better} and~\ref{fig:selected-bpdlp-better2}, PDLP exhibits an extended plateau, during which the KKT residual decreases very slowly over a large number of iterations. In contrast, BPDLP substantially shortens
this flat stage and subsequently achieves a more rapid decrease in the
KKT residual. These observations suggest that BPDLP can be particularly
effective on plateau-dominated instances.
} {It is important to note, moreover, that BPDLP could also exhibit a plateau-type behaviour; see e.g., } the instance \texttt{adult-max5features}. This observation highlights that the barrier term
does not necessarily eliminate the flat stage completely. Instead, its role is to keep the primal iterates in the interior and to moderate
their approach to the boundary, which may shorten a prolonged active set
identification phase. One possible explanation for the remaining plateau is that the barrier
parameter is reduced too rapidly. Indeed, the primal barrier update
satisfies
\[
x^{k+1}
=
\frac{z^k+\sqrt{(z^k)^2+4\tau\mu}}{2}
\longrightarrow
\max\{z^k,0\}
\qquad\text{as }\mu\to0.
\]
Hence, when \(\mu\) becomes small, the barrier update increasingly approaches to the projection step used in PDLP, and its smoothing effect
near the boundary becomes weaker. 
For these reasons, overall, the barrier mechanism should be understood as alleviating,
rather than necessarily eliminating, prolonged plateau behaviour. By
keeping the primal iterates strictly positive, it allows them to
approach the boundary more smoothly and can shorten the stagnation
phase of PDLP on some \textit{pathological} instances.



\begin{figure}[htbp]
    \centering

    \begin{subfigure}{0.45\linewidth}
        \centering
        \includegraphics[width=\linewidth]{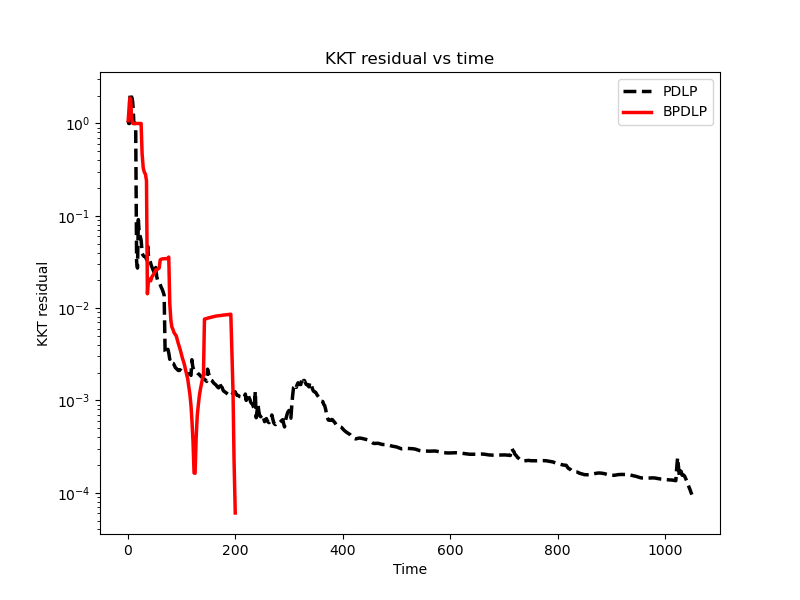}
        \caption{\texttt{neos-4976951-bunnoo}}
    \end{subfigure}
    \hfill
    \begin{subfigure}{0.45\linewidth}
        \centering
        \includegraphics[width=\linewidth]{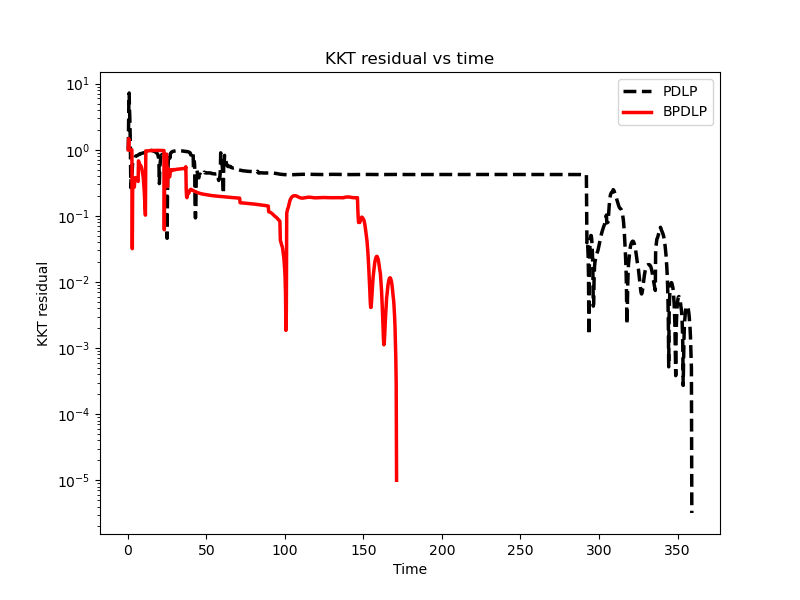}
        \caption{\texttt{adult-max5features}}
    \end{subfigure}

    \vspace{0.2cm}

    \begin{subfigure}{0.45\linewidth}
        \centering
        \includegraphics[width=\linewidth]{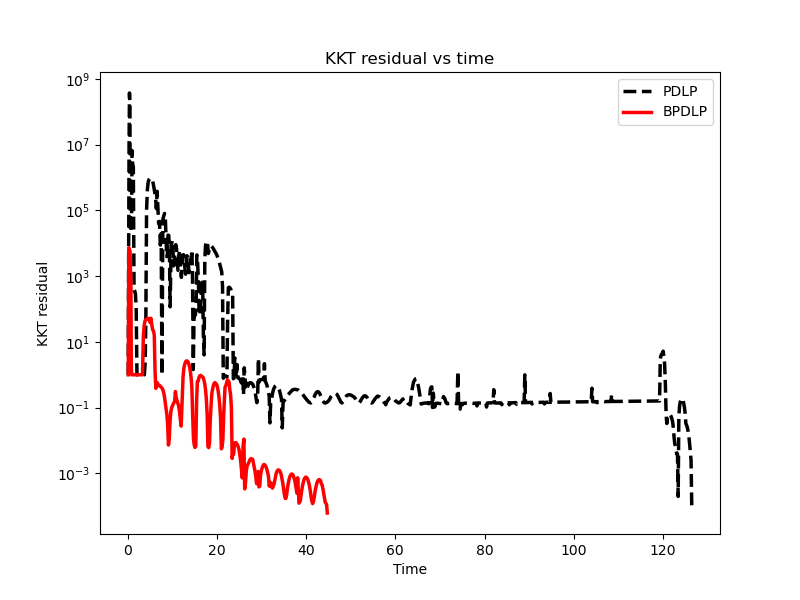}
        \caption{\texttt{rd-rplusc-21}}
    \end{subfigure}
    \hfill
    \begin{subfigure}{0.45\linewidth}
        \centering
        \includegraphics[width=\linewidth]{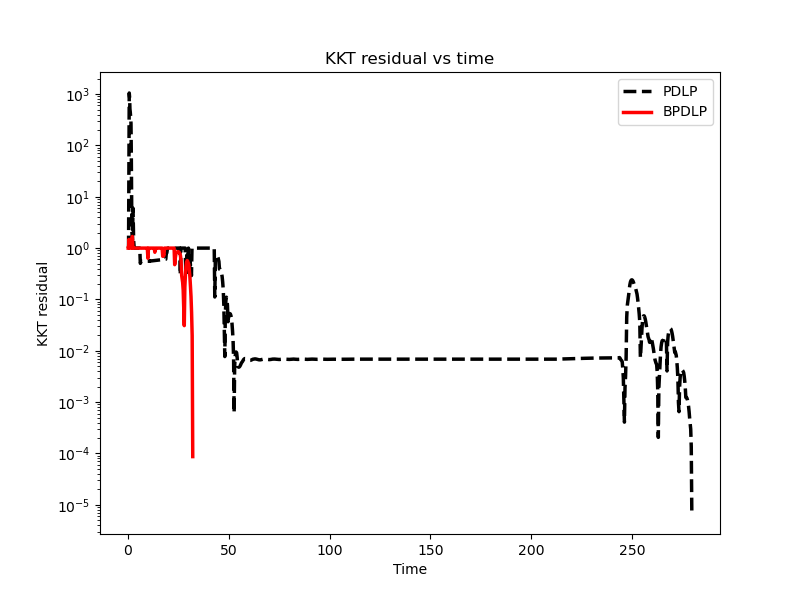}
        \caption{\texttt{radiationm40-10-02}}
    \end{subfigure}

    \vspace{0.2cm}

    \begin{subfigure}{0.45\linewidth}
        \centering
        \includegraphics[width=\linewidth]{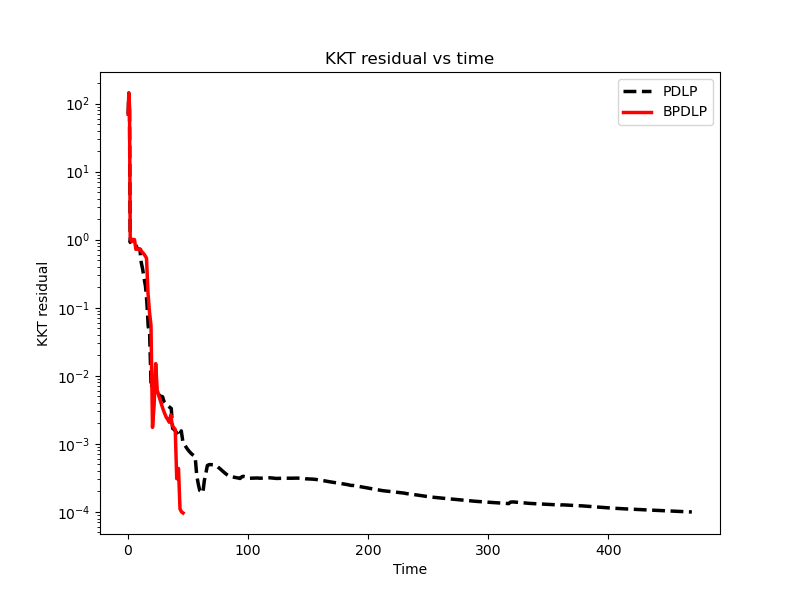}
        \caption{\texttt{ds-big}}
    \end{subfigure}
    \hfill
    \begin{subfigure}{0.45\linewidth}
        \centering
        \includegraphics[width=\linewidth]{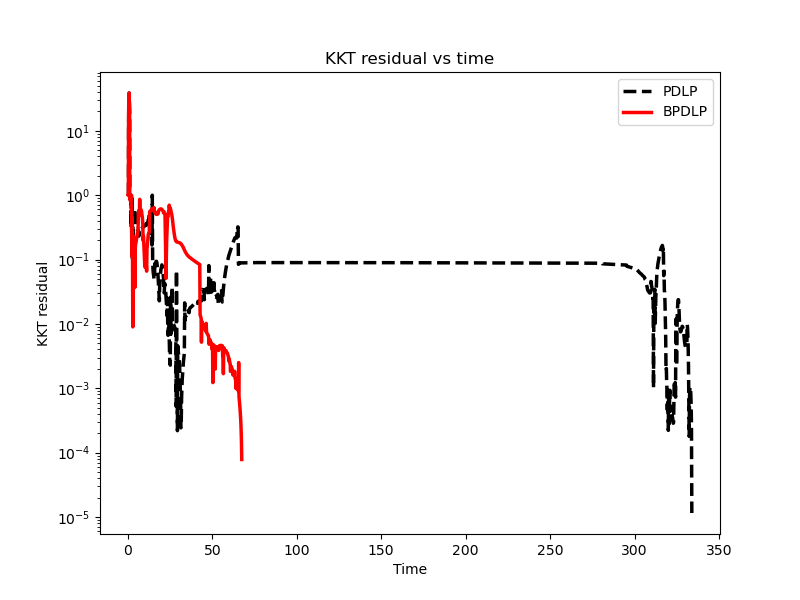}
        \caption{\texttt{dc1l}}
    \end{subfigure}

    \caption{KKT residual comparison with respect to time.}
    \label{fig:selected-bpdlp-better}
\end{figure}

\begin{figure}[htbp]
    \centering

    \begin{subfigure}{0.45\linewidth}
        \centering
        \includegraphics[width=\linewidth]{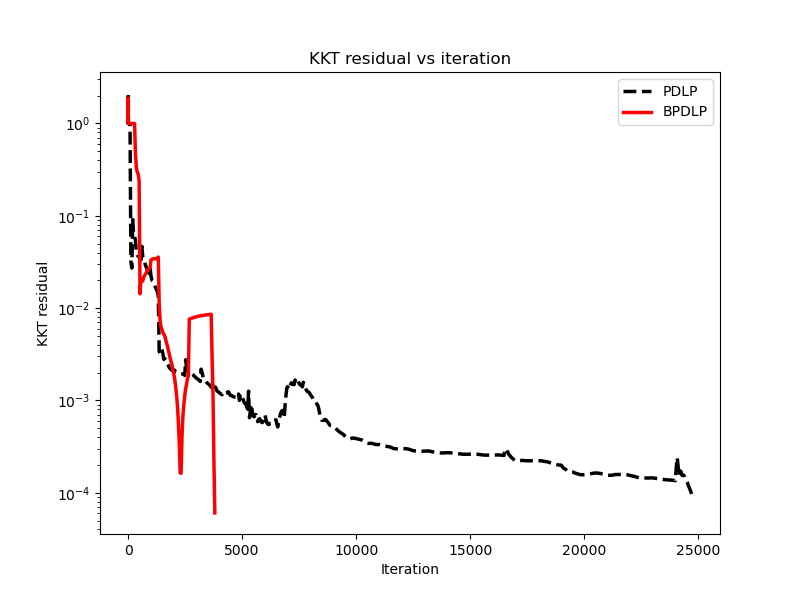}
        \caption{\texttt{neos-4976951-bunnoo}}
    \end{subfigure}
    \hfill
    \begin{subfigure}{0.45\linewidth}
        \centering
        \includegraphics[width=\linewidth]{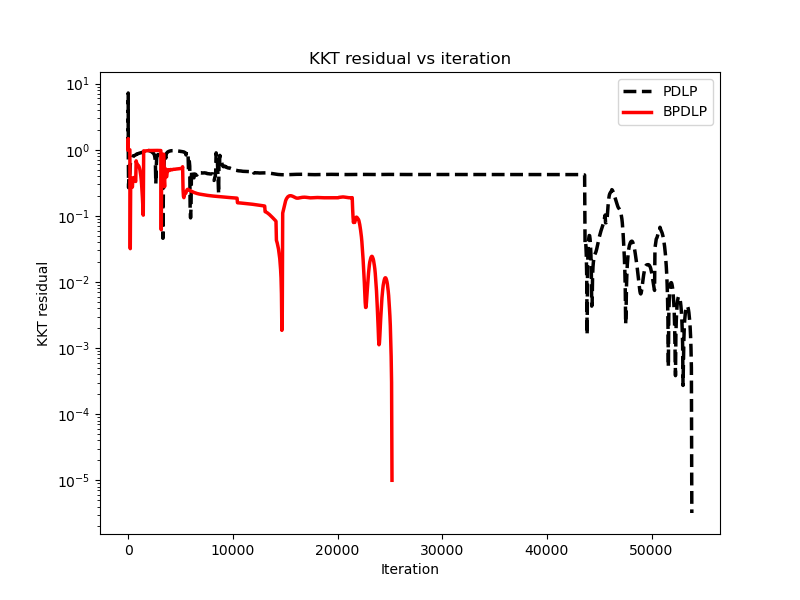}
        \caption{\texttt{adult-max5features}}
    \end{subfigure}

    \vspace{0.2cm}

    \begin{subfigure}{0.45\linewidth}
        \centering
        \includegraphics[width=\linewidth]{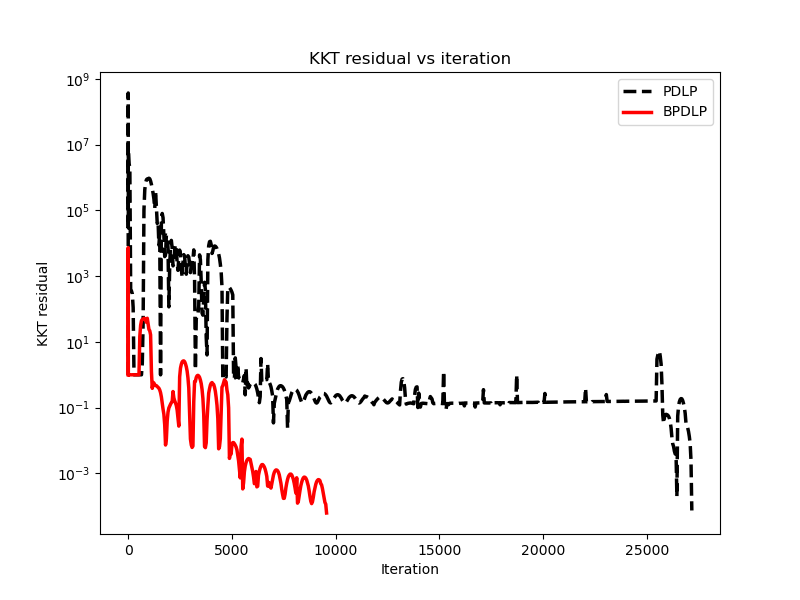}
        \caption{\texttt{rd-rplusc-21}}
    \end{subfigure}
    \hfill
    \begin{subfigure}{0.45\linewidth}
        \centering
        \includegraphics[width=\linewidth]{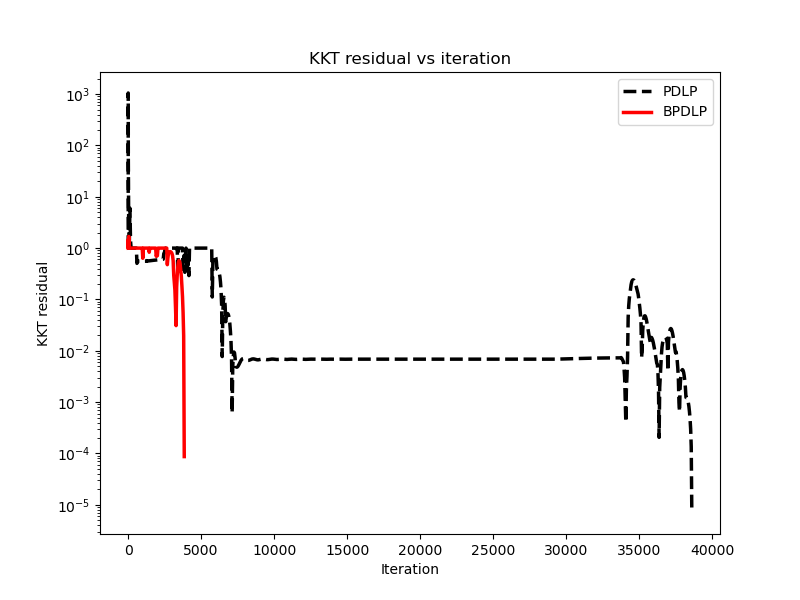}
        \caption{\texttt{radiationm40-10-02}}
    \end{subfigure}

    \vspace{0.2cm}

    \begin{subfigure}{0.45\linewidth}
        \centering
        \includegraphics[width=\linewidth]{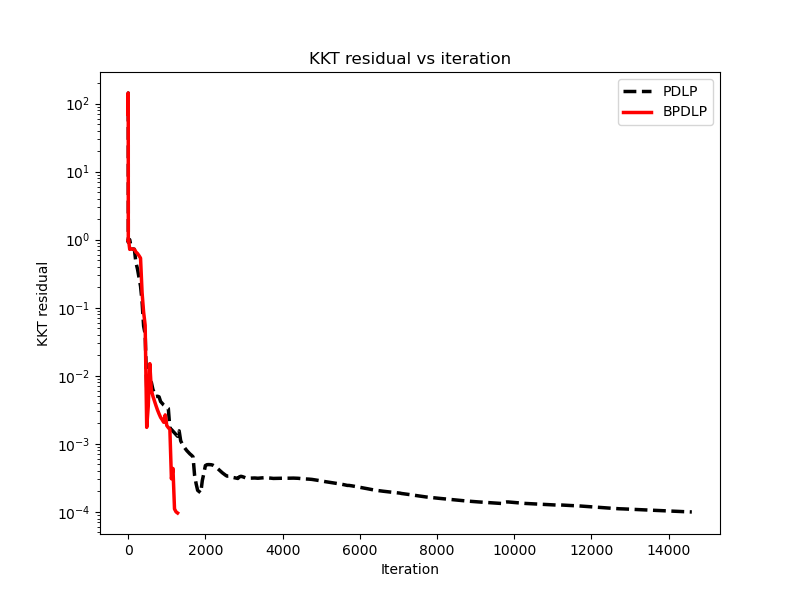}
        \caption{\texttt{ds-big}}
    \end{subfigure}
    \hfill
    \begin{subfigure}{0.45\linewidth}
        \centering
        \includegraphics[width=\linewidth]{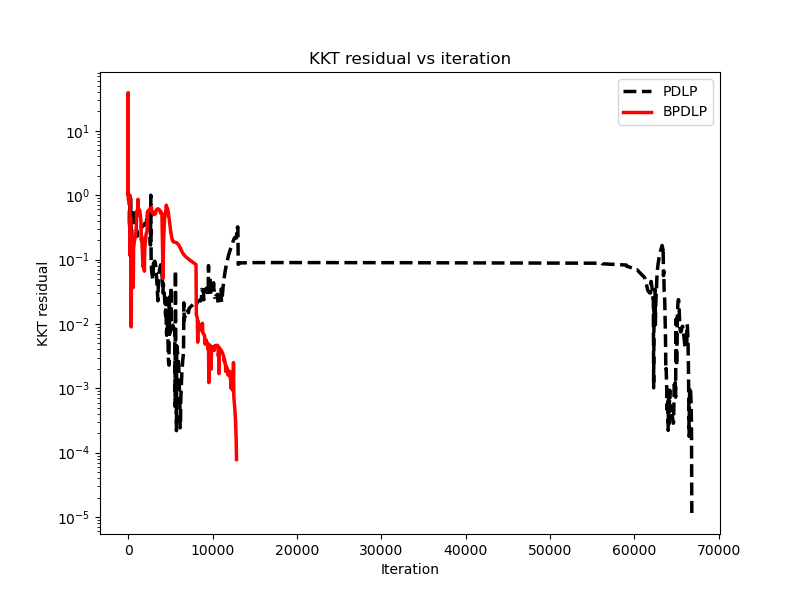}
        \caption{\texttt{dc1l}}
    \end{subfigure}

    \caption{KKT residual comparison with respect to iteration.}
    \label{fig:selected-bpdlp-better2}
\end{figure}

\subsection{When does BPDLP help?}\label{subse:EmpiricalCharacteristics}

{
The preceding experiments show that BPDLP is faster than PDLP on
approximately \(30\%\) of the MIPLIB instances solved by both methods.
However, BPDLP is not uniformly superior across the full benchmark set.
This naturally leads to the following question:
\[
\textit{Under what problem characteristics is BPDLP more likely to
outperform PDLP?}
\]
The purpose of this subsection is therefore to introduce a computable measure
\(v\) that can be used to identify problem instances on which BPDLP is
more likely to outperform PDLP. We first examine the qualitative behaviour of this proxy using small-scale synthetic examples. We then evaluate its ability to
distinguish the relative performance of PDLP and BPDLP on the larger instances, including MIPLIB, production-inventory and supply-chain datasets.
}

Our main hypothesis is that the relative performance of the two methods
is influenced by the local geometry of the original LP.  Since PDLP
operates directly on the original problem~\eqref{p:stand-LP}, its local
convergence behaviour can be affected by the sharpness of the LP near
the optimal solution set. When the sharpness is weak, a small residual
may provide only limited control of the distance to the solution set,
which may lead to slow local progress and prolonged plateau behaviour.


However, BPDLP is governed by a different local mechanism. For each fixed \(\mu\), BPDLP solves a smooth barrier
subproblem. Theorem~\ref{th:jmu-inverse-bound} establishes a local error
bound for these barrier subproblems. This error bound is different from
the Hoffman error bound \cite{Hoffman1952} associated with the original
LP. Although it still depends on the problem data through \(A\) and the
corresponding barrier solutions, the inner iterations of BPDLP do not
act directly on the nonsmooth boundary of the original LP. Instead,
they proceed through a sequence of smooth interior subproblems. Consequently, BPDLP is associated with a family of
\(\mu\) dependent error bound constants rather than a single fixed
sharpness measure. This makes the barrier error bound constant difficult to use as a practical indicator. By contrast, the sharpness of the original LP is a fixed geometric property determined by the problem itself. 


For this reason, we seek to construct a computable proxy for the
sharpness of the original LP, which can be used to characterize the
local difficulty faced by PDLP. Nevertheless, this choice does not
imply that the convergence of BPDLP is unrelated to the geometry of the
original LP. Indeed, as \(\mu\to0\), the barrier subproblems increasingly
approximate the original LP, and the geometry of the original problem
may again become important in the later stages. 


Based on the above discussion, we formulate the following {\textit{hypothesis}}:
\textit{BPDLP is more likely to gain a relative advantage on instances with
weak sharpness, for which PDLP may experience slow local progress.} To
examine this conjecture, we next construct a computable proxy for weak
sharpness and investigate whether it can distinguish instances
on which BPDLP performs better than PDLP.


{
To construct this proxy, we first recall the relationship between the
sharpness parameter of the original LP and the Hoffman constant of its
associated polyhedral system. By~\cite[Lemma~5]{fasterPDHG-23}, the sharpness parameter depends inversely on the Hoffman constant \(H(K)\), and is given by
\begin{align}\label{express-alpha}
    \alpha
    =
    \frac{\widetilde{C}}{H(K)}, \quad \text{where $\widetilde{C}>0$ is a constant}
\end{align}
and where \(H(K)\) denotes the Hoffman constant associated with problem~\eqref{p:stand-LP}, being
\begin{align*}K:=
    \begin{bmatrix}
        I \quad &0\\
        -A \quad& 0\\
        A \quad &0 \\
        0 \quad &-A^{\top}\\
        -c^{\top} \quad& b^{\top}
    \end{bmatrix}\in \mathbb{R}^{(2m+2n+1)\times (n+m)},\quad I\in \mathbb{R}^{n\times n}, \;A\in\mathbb{R}^{m\times n},\; c\in\mathbb{R}^{n},\;b\in\mathbb{R}^{m}
\end{align*} 
Moreover, Remark~2 of~\cite{fasterPDHG-23} provides the following lower
bound for the Hoffman constant:
\begin{align}\label{ineq:hoffmanbound}
    H(K)\geq \max_{\substack{J\subseteq \{1,2,\ldots, 2m+2n+1\} \\ K_J \text{ has full row rank}}}\dfrac{1}{\sigma_{\min}^+(K_J)},
\end{align}
where \(K_J\) denotes the submatrix of \(K\) formed by the rows indexed
by \(J\), and \(\sigma_{\min}^{+}(K_J)\) denotes its smallest positive
singular value.
}

{
Since evaluating the maximum in \eqref{ineq:hoffmanbound} requires considering different choices of full row rank submatrices, we
restrict our attention to a particular class of submatrices of \(K\)
that contain the final row associated with $b$ and $c$. This separation allows us to isolate the influence of \(b\) and
\(c\), which may help identify a more tractable and computable proxy.
}

Let
\[
\widehat K :=
\begin{bmatrix}
I & 0\\
-A & 0\\
A & 0\\
0 & -A^\top
\end{bmatrix}
\in \mathbb{R}^{(2n+2m)\times(n+m)},
\]
and define
\[
B_c :=
\begin{bmatrix}
-c^\top & b^\top
\end{bmatrix}.
\]
Then \(K\) can be written as
\[
K=
\begin{bmatrix}
\widehat K\\
B_c
\end{bmatrix}.
\]
For any index set
\[
J\subseteq\{1,\ldots,2n+2m\},
\]
let
\[
H_J:=\widehat K_J
\]
denote the submatrix of \(\widehat K\) formed by the rows indexed by
\(J\). 

We now restrict attention to index sets \(J\) for which \(H_J\) has
full row rank. By adding the final row \(B_c\), we obtain the augmented
submatrix
\[
K_J(c):=
\begin{bmatrix}
H_J\\
B_c
\end{bmatrix}.
\]
If \(K_J(c)\) has still full row rank, then it is an admissible submatrix in
\eqref{ineq:hoffmanbound}, and hence
\[
H(K)
\geq
\frac{1}{\sigma_{\min}^{+}(K_J(c))}.
\]
Therefore, a small upper bound on
\(\sigma_{\min}^{+}(K_J(c))\) yields a large computable lower bound on
\(H(K)\), and consequently a computable small upper bound on the sharpness
parameter \(\alpha\). 

Define
\[
\delta_J(c)
:=
\operatorname{dist}\!\left(B_c,\operatorname{row}(H_J)\right)
=
\min_{\lambda\in\mathbb{R}^{|J|}}
\left\|B_c-\lambda^\top H_J\right\|.
\]
The quantity $\delta_J(c)$ measures how far $B_c$ is from being represented by the selected rows of $\widehat K$, namely, those used to construct $H_J$. It can therefore be used to determine whether $K_J(c)$ has full row rank. Accordingly, we distinguish the following two cases:

\begin{enumerate}
    \item 
    \(
    B_c\in\operatorname{row}(H_J)
    \)
    so that \(\delta_J(c)=0\). In this case, the matrix
    \[
    K_J(c)
    =
    \begin{bmatrix}
        H_J\\
        B_c
    \end{bmatrix}
    \]
    is not full row rank because \(B_c\) is redundant. Therefore,
    \(K_J(c)\) is not an admissible submatrix in
    \eqref{ineq:hoffmanbound}. So removing the redundant row \(B_c\)
    leaves a full row rank, hence an admissible submatrix of
    \(K\). Consequently, when \(\delta_J(c)=0\), the augmented matrix \(K_J(c)\) cannot be used
directly in \eqref{ineq:hoffmanbound}, and the resulting candidate bound must instead be obtained by further analyzing a full row rank submatrix of \(H_J\).


    \item 
    \(
    B_c\notin\operatorname{row}(H_J)
    \)
    so that \(\delta_J(c)>0\). Since \(H_J\) has full row rank and \(B_c\)
    does not belong to its row space, \(K_J(c)\) also has full row
    rank. We next show that
    \[
    \sigma_{\min}^{+}(K_J(c))
    \leq
    \delta_J(c).
    \]

    Let
    \(
    \lambda^*
    \in
    \arg\min_{\lambda\in\mathbb{R}^{|J|}}
    \left\|
    B_c-\lambda^\top H_J
    \right\|.
    \)
    Then
    \(
    \delta_J(c)
    =
    \left\|
    B_c-(\lambda^*)^\top H_J
    \right\|.
    \)
    Define
    \[
    \nu
    :=
    \begin{bmatrix}
        -\lambda^*\\
        1
    \end{bmatrix}.
    \]
    It follows that
    \[
    K_J(c)^\top \nu
    =
    B_c^\top-H_J^\top\lambda^*,
    \]
    and hence
    \[
    \left\|K_J(c)^\top\nu\right\|
    =
    \delta_J(c).
    \]

    Now define
    \[
    \widehat{\nu}
    :=
    \frac{\nu}{\|\nu\|}
    =
    \frac{\nu}{\sqrt{1+\|\lambda^*\|^2}}.
    \]
    Since \(\|\widehat{\nu}\|=1\), we have
    \begin{align}
    \label{ineq:positive_eig}
    \sigma_{\min}^{+}(K_J(c))
    &=
    \min_{\|v\|=1}
    \left\|K_J(c)^\top v\right\|
    \leq
    \left\|K_J(c)^\top\widehat{\nu}\right\|
    =
    \frac{\delta_J(c)}
    {\sqrt{1+\|\lambda^*\|^2}}
    \leq
    \delta_J(c).
    \end{align}
\end{enumerate}
We now show how $\delta_J(c)>0$ can be used to derive an upper bound on the sharpness parameter.

If \(\delta_J(c)>0\), then \(K_J(c)\) is admissible and
\[
H(K)
\geq
\frac{1}{\sigma_{\min}^{+}(K_J(c))}
\geq
\frac{\sqrt{1+\|\lambda^*\|^2}}{\delta_J(c)}
\geq
\frac{1}{\delta_J(c)}.
\]
Combining this estimate with the expression for the sharpness
parameter \eqref{express-alpha} yields
\[
\alpha
\leq
\frac{\widetilde{C}\delta_J(c)}
{\sqrt{1+\|\lambda^*\|^2}}
\leq
{\widetilde{C}\delta_J(c)}.
\]
Therefore, in this case, a small positive value of \(\delta_J(c)\) gives a small
upper bound on the sharpness parameter and may indicate potentially
weak sharpness of the original LP.

We next exploit the block structure of $\widehat K$ to construct specific row selections and derive explicit candidate bounds for the two cases.

To describe the possible row selections from \(\widehat K\), let \(R\)
denote an index set associated with the rows selected from the block
$B:=\begin{bmatrix}
    I \\ -A \\ A
\end{bmatrix}$
and let
\(Q\subseteq\{1,\ldots,n\}\) denote the indices associated with the rows
selected from the block \(-A^\top\). After a suitable ordering of the
selected rows, the corresponding submatrix can be written as
\[
H_{R,Q}
:=
\begin{bmatrix}
B_R & 0\\
0 & -A_{:,Q}^{\top}
\end{bmatrix},
\]
where \(B_R\) is formed by selecting the rows $R$ from $B$ and $A_{:,Q}$ is the submatrix of \(A\) formed by the columns
indexed by \(Q\). The index sets \(R\) and \(Q\) are allowed to be empty.
Whenever the corresponding block is nonempty, we assume that \(B_R\)
has full row rank and \(A_{:,Q}\) has full column rank. Hence, given its block structure,
\(H_{R,Q}\) has full row rank whenever \(R\) and \(Q\) are not
simultaneously empty.

The preceding discussion applies to any index set \(J\) for which
\(H_J\) has full row rank. Once \(J\) is fixed, the corresponding
quantity
\[
\delta_J(c)
=
\operatorname{dist}
\left(
B_c,\operatorname{row}(H_J)
\right)
\]
is also fixed. Hence, the selected index set \(J\) necessarily belongs
to exactly one of the two cases discussed above:
\[
\delta_J(c)=0
\qquad\text{or}\qquad
\delta_J(c)>0.
\]
We now consider several particular choices of \(J\) induced by the
block structure of \(\widehat K\). For each choice, we determine the
corresponding value of \(\delta_J(c)\), identify which of the two cases
applies, and derive the resulting candidate bound.

\begin{enumerate}
    \item\label{case-1}
If
\[
B_c\in\operatorname{row}(H_{R,Q}),
\]
then \(\delta_J(c)=0\). By the preceding discussion, \(B_c\) is
redundant in \(K_J(c)\), and the remaining full row rank admissible
submatrix is \(H_{R,Q}\). If both diagonal blocks are nonempty, then
\[
\sigma_{\min}^{+}(H_{R,Q})
=
\min\left\{
\sigma_{\min}^{+}(B_R),
\sigma_{\min}^{+}(A_{:,Q}^{\top})
\right\}.
\]
If one block is empty, the corresponding term is omitted.
    


    In particular, if the selected rows forming \(B_R\) are nearly linearly dependent, or if the selected
    columns of \(A_{:,Q}\) are nearly linearly dependent, then
    \(\sigma_{\min}^{+}(H_{R,Q})\) is small. This yields a large
candidate lower bound on the Hoffman constant and, consequently, a small
candidate upper bound on the sharpness parameter. Such a small upper bound
may indicate potentially weak sharpness, which may be associated with
slower local progress of PDLP \cite[Remark 3]{fasterPDHG-23}.
    \item\label{case-2}
    Choose \(J\) such that
\[
H_J=
\begin{bmatrix}
I&0
\end{bmatrix}.
\]
    
    Then,
    \[
    K_J(c)
    :=
    \begin{bmatrix}
    I & 0\\
    -c^{\top} & b^{\top}
    \end{bmatrix}.
    \]
    In this case,
    \[
    \delta_J(c)^2
    =
    \min_{\lambda\in\mathbb{R}^n}
    \left(
    \|-c-\lambda\|^2+\|b\|^2
    \right)>0.
    \]
    The minimum is attained at \(\lambda^*=-c\), and hence
    \(
    \delta_J(c)=\|b\|.
    \)
    It follows from \eqref{ineq:positive_eig} that
    \begin{align}
    \label{special-case-I}
    \sigma_{\min}^{+}(K_J(c))
    \leq
    \frac{\delta_J(c)}
    {\sqrt{1+\|\lambda^*\|^2}}
    =
    \frac{\|b\|}
    {\sqrt{1+\|c\|^2}}.
    \end{align}


        \item\label{case-3}  {  Consider the general block selection of $J$ such that,
    \[
    H_J
    =
    H_{R,Q}
    =
    \begin{bmatrix}
    B_R & 0\\
    0 & -A_{:,Q}^{\top}
    \end{bmatrix}.
    \]
    For this choice, we have
    \[
    \begin{aligned}
    \delta_{R,Q}(c)^2
    &=
    \operatorname{dist}\!\left(
    c,\operatorname{range}(B_R^\top)
    \right)^2
    +
    \operatorname{dist}\!\left(
    b,\operatorname{range}(A_{:,Q})
    \right)^2=
    \|(I-\mathcal{P}_R)c\|^2
    +
    \|(I-\mathcal{P}_Q)b\|^2>0,
    \end{aligned}
    \]
    } where \(\mathcal{P}_R\) is the orthogonal projector onto
    \(\operatorname{range}(B_R^\top)\), and \(\mathcal{P}_Q\) is the orthogonal
    projector onto \(\operatorname{range}(A_{:,Q})\).

    From \eqref{ineq:positive_eig}, we have
    \begin{align}\label{ineq:general-projection-bound}
    \sigma_{\min}^{+}(K_J(c))
    \leq
    \frac{
    \sqrt{
    \|(I-\mathcal{P}_R)c\|^2+\|(I-\mathcal{P}_Q)b\|^2
    }}
    {
    \sqrt{
    1+\|\lambda_1^*\|^2+\|\lambda_2^*\|^2
    }}
  \leq
    \frac{
    \sqrt{
    \|(I-\mathcal{P}_R)c\|^2+\|(I-\mathcal{P}_Q)b\|^2
    }}
    {
    \sqrt{
    1+
    \dfrac{\|\mathcal{P}_Rc\|^2}{\|B_R\|^2}
    +
    \dfrac{\|\mathcal{P}_Qb\|^2}{\|A_{:,Q}\|^2}
    }
    },
    \end{align}
    where
    \[
    B_R^{\top}\lambda_1^*=\mathcal{P}_Rc,
    \qquad
    A_{:,Q}\lambda_2^*=\mathcal{P}_Qb.
    \]
    The last inequality follows from
    \[
    \|\mathcal{P}_Rc\|
    =
    \|B_R^{\top}\lambda_1^*\|
    \leq
    \|B_R\|\|\lambda_1^*\|
    \]
    and
    \[
    \|\mathcal{P}_Qb\|
    =
    \|A_{:,Q}\lambda_2^*\|
    \leq
    \|A_{:,Q}\|\|\lambda_2^*\|.
    \]

\end{enumerate}
We {are now ready to present } a series of numerical experiments based on the three
candidate bounds derived above. In these experiments, we compare the
performance of PDHG, rPDHG, and BPDHG under different problem settings.
Our main purpose is to examine whether BPDHG exhibits a larger
advantage when the sharpness of the original LP is potentially weak,
and hence when PDHG type methods may experience slower convergence stage. These experiments are not intended to compute the exact
sharpness parameter, but rather to provide preliminary numerical
evidence supporting the qualitative implications of the bounds derived
above.

We first examine the mechanism identified in Case~\ref{case-1},
where the candidate bound is governed by near linear dependence among
selected rows or columns of \(A\). Consider the following toy linear
program:
\begin{example}
\[
\begin{aligned}
\min_{x\geq 0}\quad & -2x_1-x_2\\
\text{s.t.}\quad &
\begin{bmatrix}
1 & 2\\
1 & 2+\varepsilon
\end{bmatrix}x
=
\begin{bmatrix}
3\\
3
\end{bmatrix}.
\end{aligned}
\]
\end{example}
This example corresponds to the representable case discussed above.
Indeed, by choosing \(Q=\{1\}\) and \(B_R=A\), we obtain
\[
H_{R,Q}
=
\begin{bmatrix}
1 & 2+\varepsilon & 0 & 0 \\
1 & 2             & 0 & 0 \\
0 & 0             & -1 & -1
\end{bmatrix}.
\]
Since
\[
B_c=
\begin{bmatrix}
2 & 1 & 3 & 3
\end{bmatrix}
\in \operatorname{row}(H_{R,Q}),
\]
we have \(\delta_J(c)=0\). We can thus investigate this candidate bound by varying the
degree of linear dependence in the constraint matrix \(A\) according to the discussion above.

As \(\varepsilon\) decreases, the two rows of the constraint matrix
become increasingly close to linear dependence, causing its smallest
positive singular value to decrease. We consider
\[
\varepsilon\in\{10^{-8},10^{-2},1,2\}.
\]
Except for \(\varepsilon\), all other experimental settings are kept
fixed. Specifically, we use a step size of \(0.2\), a stopping
tolerance of \(10^{-8}\), and a maximum of \(5{,}000\) iterations.

The results are shown in Figure~\ref{fi:rd}. The figure suggests that
the relative advantage of BPDHG becomes more pronounced as the problem
approaches rank deficiency. However, once the smallest positive singular
value becomes sufficiently small, further decreases appear to have only
a limited additional effect on the observed convergence behavior.
Overall, this experiment provides preliminary numerical evidence
supporting the mechanism described in the first case.

\begin{figure}
    \centering
    \includegraphics[width=0.8\linewidth]{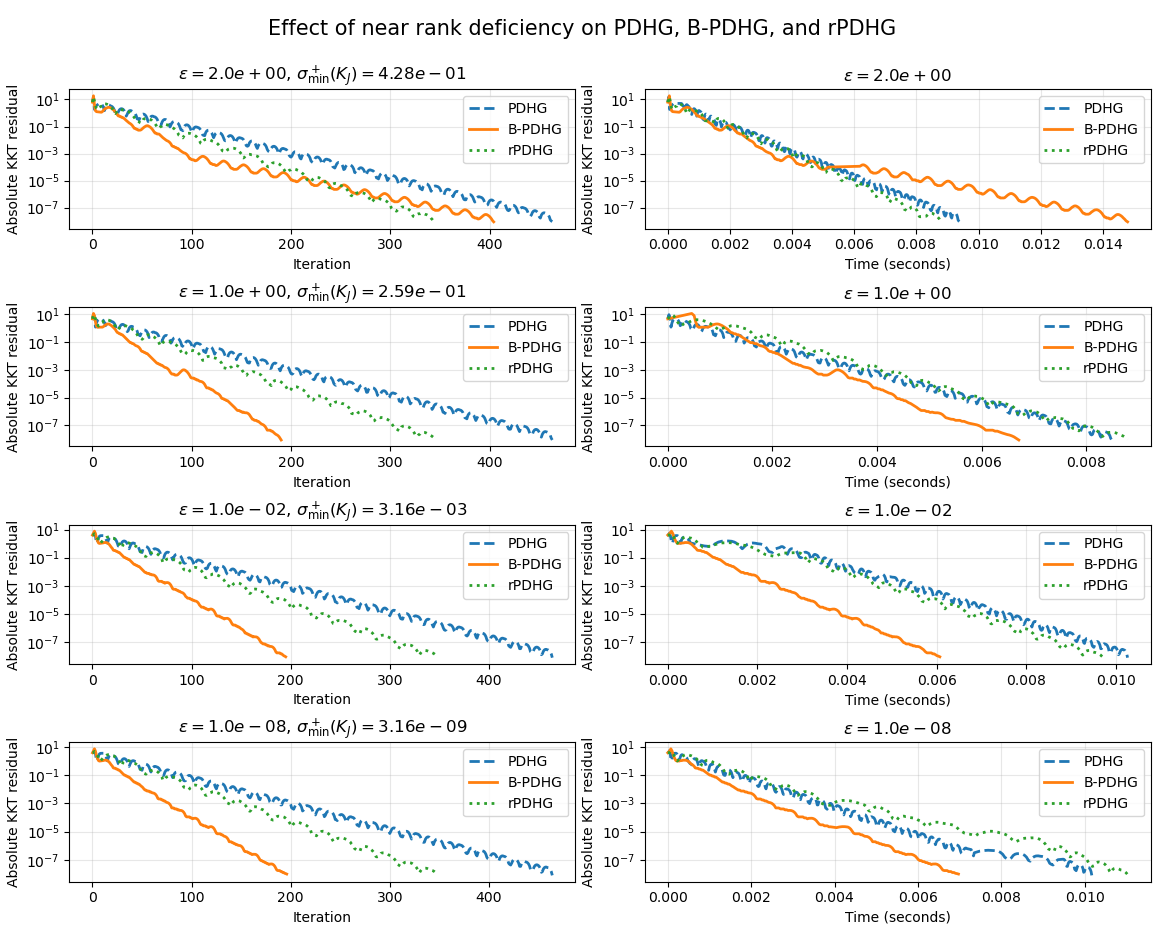}
    \caption{Effect of near rank deficiency on the convergence of PDHG, BPDHG, and rPDHG. The left column shows the KKT residual against iteration count, while
the right column shows the KKT residual against computing time.
Smaller values of \(\varepsilon\) make the two rows of the constraint
matrix increasingly close to linear dependence.}
    \label{fi:rd}
\end{figure}

We next consider Case~\ref{case-2}, namely the bound given by \eqref{special-case-I}, which postulates that the upper bound
on \(\sigma_{\min}^{+}(K_J(c))\) decreases either when \(\|c\|\)
increases with \(\|b\|\) fixed or comparable, or when \(\|b\|\)
decreases with \(\|c\|\) fixed or comparable. Consequently, both cases
may lead to a smaller upper bound on the sharpness parameter. To examine the first mechanism while isolating the effect of \(c\), we fix
\(b\) and vary only the magnitude of \(c\). This experiment corresponds to
the special choice
\[
H_J=
\begin{bmatrix}
I_2 & 0
\end{bmatrix}.
\]
Equivalently, this is obtained by taking \(Q=\emptyset\) and choosing \(R\)
so that
\[
H_{R,Q}
=
\begin{bmatrix}
1&0&0&0\\
0&1&0&0
\end{bmatrix}.
\]
Under this choice, the bound in~\eqref{special-case-I} applies directly.
We therefore consider the following example:
\begin{example}
\[
\min_{x\geq0} c(\chi)^{\top}x
\qquad
\text{s.t.}
\qquad
\begin{bmatrix}
1&2\\
1&1
\end{bmatrix}x
=
\begin{bmatrix}
3\\
3
\end{bmatrix},
\]
where
\[
c(\chi)
=
\frac{\chi}{\sqrt{5}}
\begin{bmatrix}
2\\
1
\end{bmatrix},
\qquad
\chi
\in
\left\{
0.005,
0.5,
2,
5
\right\} \times \|b\| \text{ to control the norm of $c$}.
\]
\end{example}


This example has a unique feasible point 
\(
x^*=(3,0)^{\top}.
\) For the experimental settings, we use a step size of \(0.1\), a stopping
tolerance of \(10^{-8}\), and a maximum of \(5{,}000\) iterations.
The results are shown in Figure~\ref{fi:special-I}.
\begin{figure}
    \centering
    \includegraphics[width=0.8\linewidth]{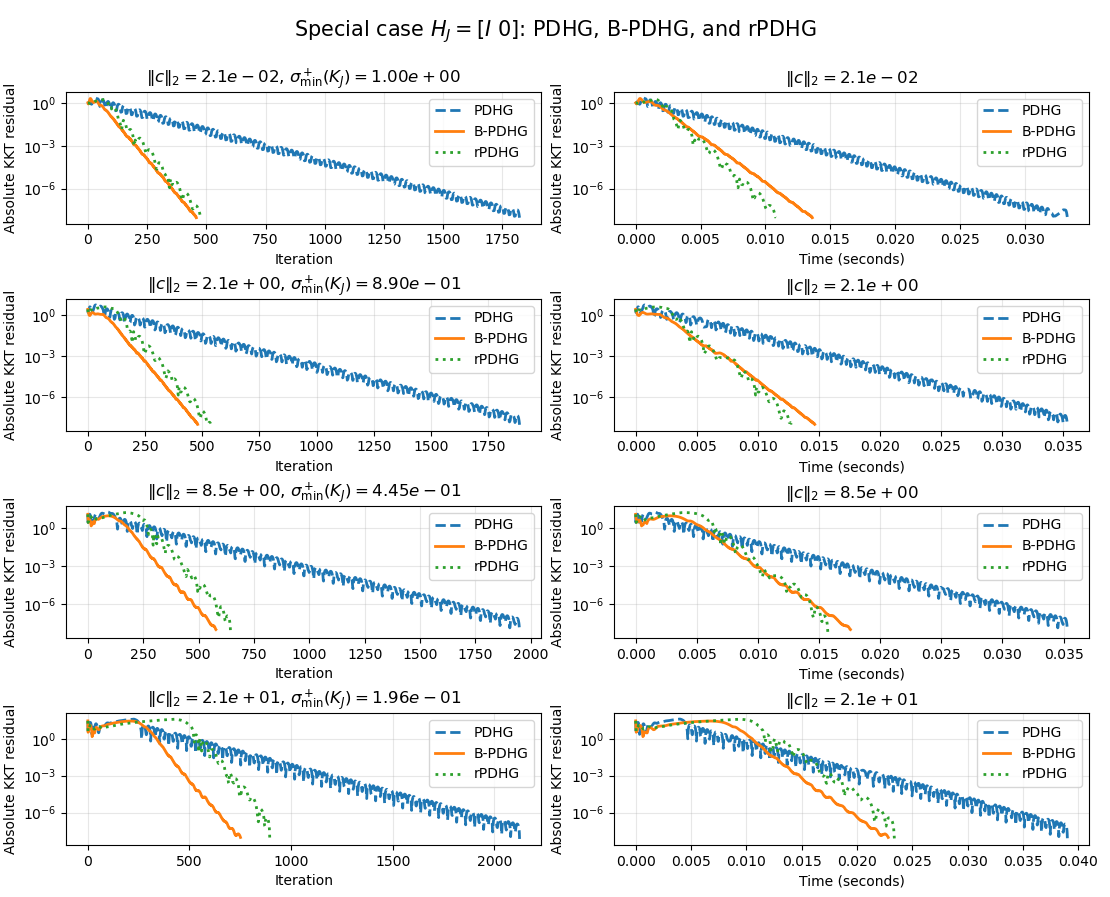}
    \caption{Verification of the predicted relationship between
\(\|c\|\) and the KKT residual decay of the three methods under the
special choice \(H_J=[\,I\;\;0\,]\).}\label{fi:special-I}
\end{figure} Also in this case, the presented results provide preliminary numerical support of the fact that larger objective norms may be associated with smaller $\sigma_{\min}^+(K_J(c))$, and the iteration advantage of BPDHG over rPDHG becomes more prominent.

We finally consider the slightly more complicated case, i.e., the general block selection \eqref{ineq:general-projection-bound}. Since in this case the bound involves the projections of both \(c\) and \(b\) onto
their corresponding subspaces, we construct the problem so as to reduce
the influence of different vectors. In particular, we choose
\(Q\) such that
\[
\mathcal{P}_Qb=b,
\qquad
(I-\mathcal{P}_Q)b=0.
\]
Meanwhile, we construct \(c\) through an orthogonal decomposition into
a component in \(\operatorname{range}(B_R^\top)\) and a component
orthogonal to this subspace.

Specifically, we write
\begin{align}\label{eq:ctheta}
c(\theta)
=
\rho\bigl(\cos(\theta)u+\sin(\theta)v\bigr),
\end{align}
where
\[
u\in\operatorname{range}(B_R^\top),
\qquad
v\perp\operatorname{range}(B_R^\top),
\qquad
u\perp v,
\qquad
\|u\|=\|v\|=1.
\]
It follows that
\[
\mathcal{P}_Rc(\theta)=\rho\cos(\theta)u,
\qquad
(I-\mathcal{P}_R)c(\theta)=\rho\sin(\theta)v,
\]
while
\[
\|c(\theta)\|=\rho
\]
remains fixed. Therefore, varying \(\theta\) changes only the relative
sizes of the projected and orthogonal components of \(c\), while the
terms associated with \(b\) and the norm of \(c\) remain unchanged.
Then from \eqref{ineq:general-projection-bound}, we have
\[
\sigma_{\min}^{+}(K_J(c(\theta)))
    \leq \dfrac{\sqrt{\|(I-\mathcal{P}_{R})c\|^2 + \|((I-\mathcal{P}_Q)b)\|^2}}{\sqrt{1+\dfrac{\|\mathcal{P}_Rc\|^2}{\|B_R\|^2}+\dfrac{\|\mathcal{P}_Qb\|^2}{\|A_{:,Q}\|^2}}}=
\dfrac{\rho|\sin(\theta)|}{\sqrt{2+\dfrac{\rho^2 \cos^2(\theta)}{2}}}.
\]
Hence, by decreasing $\theta\in (0,\dfrac{\pi}{2})$, we would get the smaller upper bound of $\sigma_{\min}^+(K_J)$. 

To summarize, we consider the following problem:
\begin{example}
    \[
    \min_{x}\ c(\theta)^\top x\quad\text{s.t.}\quad \begin{bmatrix}1&1\end{bmatrix}x=1,\ x\ge0,
    \]
    where $c(\theta)$ has the form of \eqref{eq:ctheta}.
    We set $\rho=10$,\; $u=\dfrac{1}{\sqrt{2}}[1 \quad 1]^{\top}$, $v =\dfrac{1}{\sqrt{2}}[1 \quad -1]^{\top}$.
\end{example}
For the row and column selection used in Case~\ref{case-3}, we take \(B_R=[\,1\;\;1\,]\) and \(Q=\{1\}\), so that
\(
H_{R,Q}
=
\begin{bmatrix}
1 & 1 & 0\\
0 & 0 & -1
\end{bmatrix}.
\)

We test
\(
\theta\in\{5^\circ,25^\circ,45^\circ,80^\circ\}
\) for this example. For other experimental settings, we use a step size of \(0.2\), a stopping
tolerance of \(10^{-8}\), and a maximum of \(5{,}000\) iterations. The results are shown in Figure~\ref{fi:sig-U_RQ}. Such results support the observation that, when all other variables are held fixed, the iteration advantage of BPDHG over PDHG and rPDHG becomes more
pronounced as \(\|\mathcal{P}_Rc\|\) increases.
\begin{figure}
    \centering
    \includegraphics[width=0.8\linewidth]{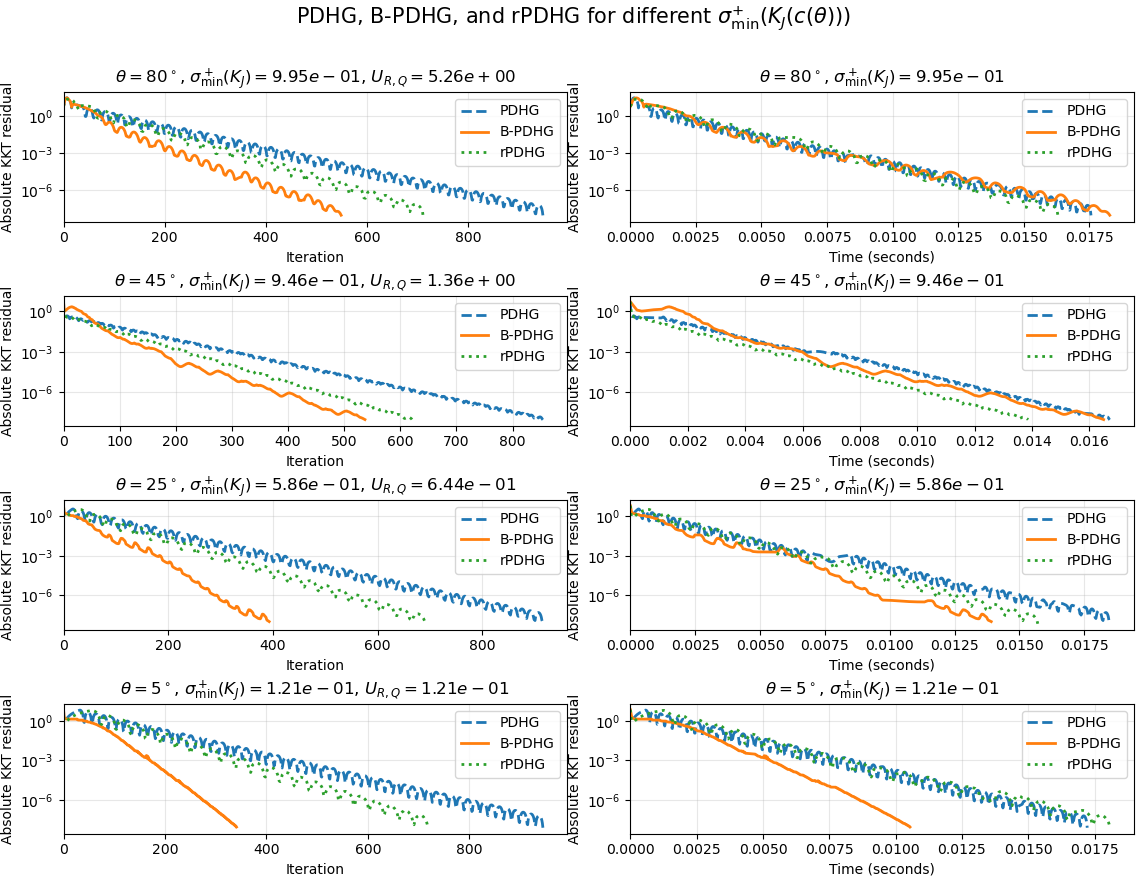}
    \caption{Performance of the three algorithms under different choices of $\theta$.
    }\label{fi:sig-U_RQ}
\end{figure}

Overall, the three simple numerical experiments presented suggest that a smaller value of \(\sigma_{\min}^{+}(K_J)\), and hence potentially weaker sharpness of the original LP, may be associated with a larger relative advantage of BPDHG over PDHG and rPDHG. However, for large-scale datasets, the bounds in Case~\ref{case-1} and Case~\ref{case-3} are difficult to evaluate
effectively in practice, since they involve problem specific row and
column selections, singular values, and projection quantities. In
contrast, the bound in \eqref{special-case-I} depends only on
\(\|b\|\) and \(\|c\|\), and can therefore be computed efficiently for
all instances in the dataset.

Although this bound might be loose, it provides a simple one-sided proxy
for potentially weak sharpness. Specifically, we define
\begin{align}\label{possible-measure}
  v:=\frac{\|b\|}{\sqrt{1+\|c\|^2}}.
\end{align}
Therefore, a small value of \(v\) guarantees a small candidate upper
bound on \(\sigma_{\min}^{+}(K_J(c))\), which in turn yields a small candidate upper bound on the sharpness parameter. However, the converse does not necessarily hold.

We next investigate whether \(v\) is informative for the relative
performance of BPDLP and PDLP on a larger dataset. For each LP instance,
we compute \(v\) and examine whether smaller values are associated with
better performance of BPDLP. If such a relationship is
consistently observed across the dataset, \(v\) may serve as a simple
empirical indicator of instances on which BPDLP is more likely to
outperform PDLP.

Our test set consists of the MIPLIB instances described in
Section~\ref{subse:Comparison with PDLP on the MIPLIB Benchmark}, together with production-inventory and supply-chain instances generated using the code provided by the authors of~\cite{PDLPnew2025}. Following their generation procedures, we generate 20 medium to large instances from each class, resulting in 40 additional instances in total. A presolve procedure is applied to all instances before they are solved. Together with the 381 MIPLIB instances, the final test set contains 421 LP instances.

For BPDLP, the parameter configuration is fixed at the dataset level.
We use one common parameter setting for all MIPLIB instances, one for
all production-inventory instances, and one for all supply-chain
instances. Thus, the parameters are not tuned separately for individual
test instances. More specifically, for the MIPLIB dataset, we use the
same parameter configuration as in
Section~\ref{subse:Comparison with PDLP on the MIPLIB Benchmark}. For
the production-inventory instances, we set
\(
\theta=0.1,\;\vartheta=0.05,\;
\tau_{\mathrm{inner}}=0.8,
\)
whereas for the supply-chain instances, we set
\(
\theta=0.1,\;
\vartheta=0.2,\;
\tau_{\mathrm{inner}}=0.85.
\)
For PDLP, we use the default parameter settings throughout.

For every instance, we compute \(v\) and compare it with the relative
performance of BPDLP and PDLP. We evaluate the two methods using
performance profiles based on iteration count and solution time,
following the methodology in~\cite{ASPerprof2016}. Instances that are
not solved within the prescribed time or iteration limit are assigned
the corresponding limiting value. 

We then group the instances according to the magnitude of \(v\) and
construct a separate performance profile for each group. This allows us
to examine whether smaller values of \(v\) are associated
with a greater relative advantage of BPDLP over PDLP. The results are provied in Figure~\ref{fi:pp-8-it-time}. The intervals in Figure~\ref{fi:pp-8-it-time} are
chosen to provide a finer division in the small \(v\) region, since
our main interest is in determining whether particularly small values of
\(v\) are informative.


\begin{figure}
\centering
\includegraphics[width=0.9\linewidth]{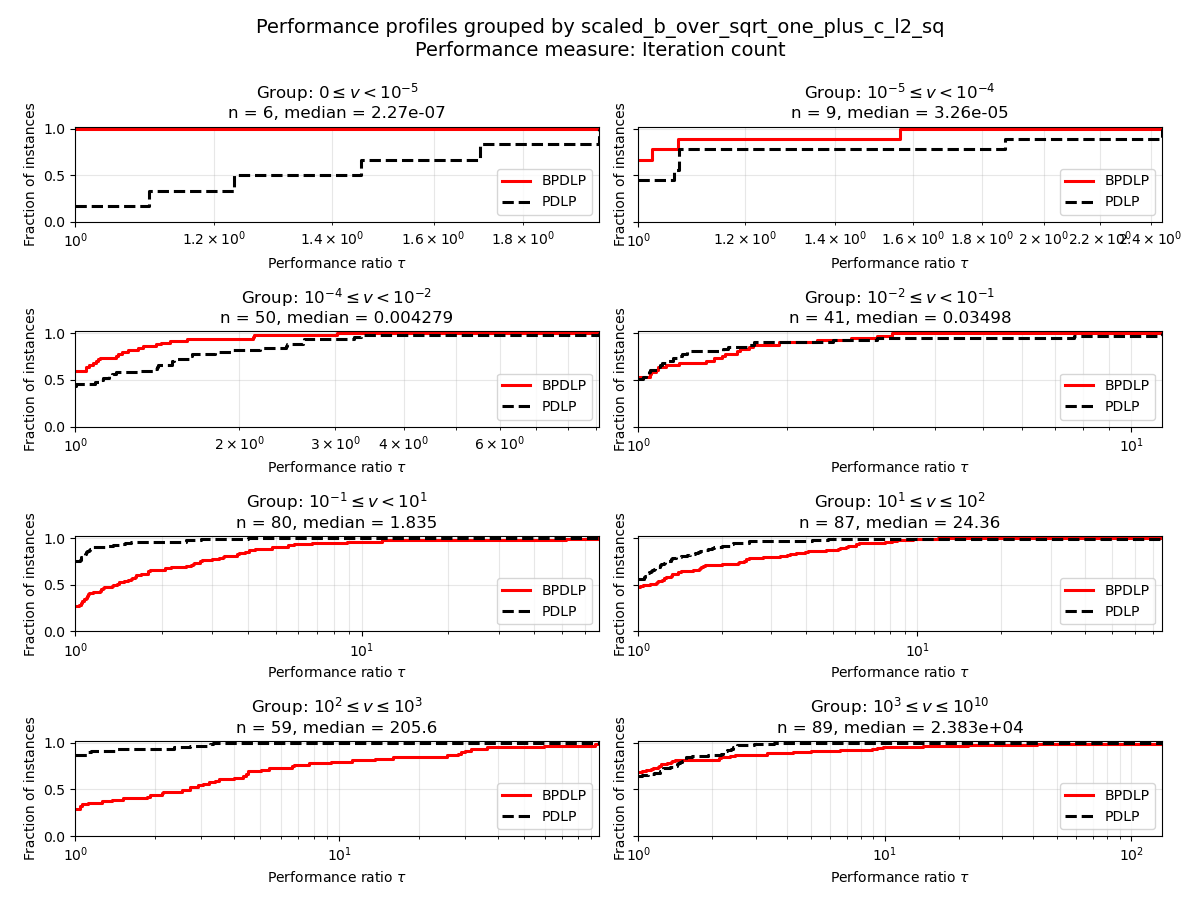}
\includegraphics[width=0.9\linewidth]{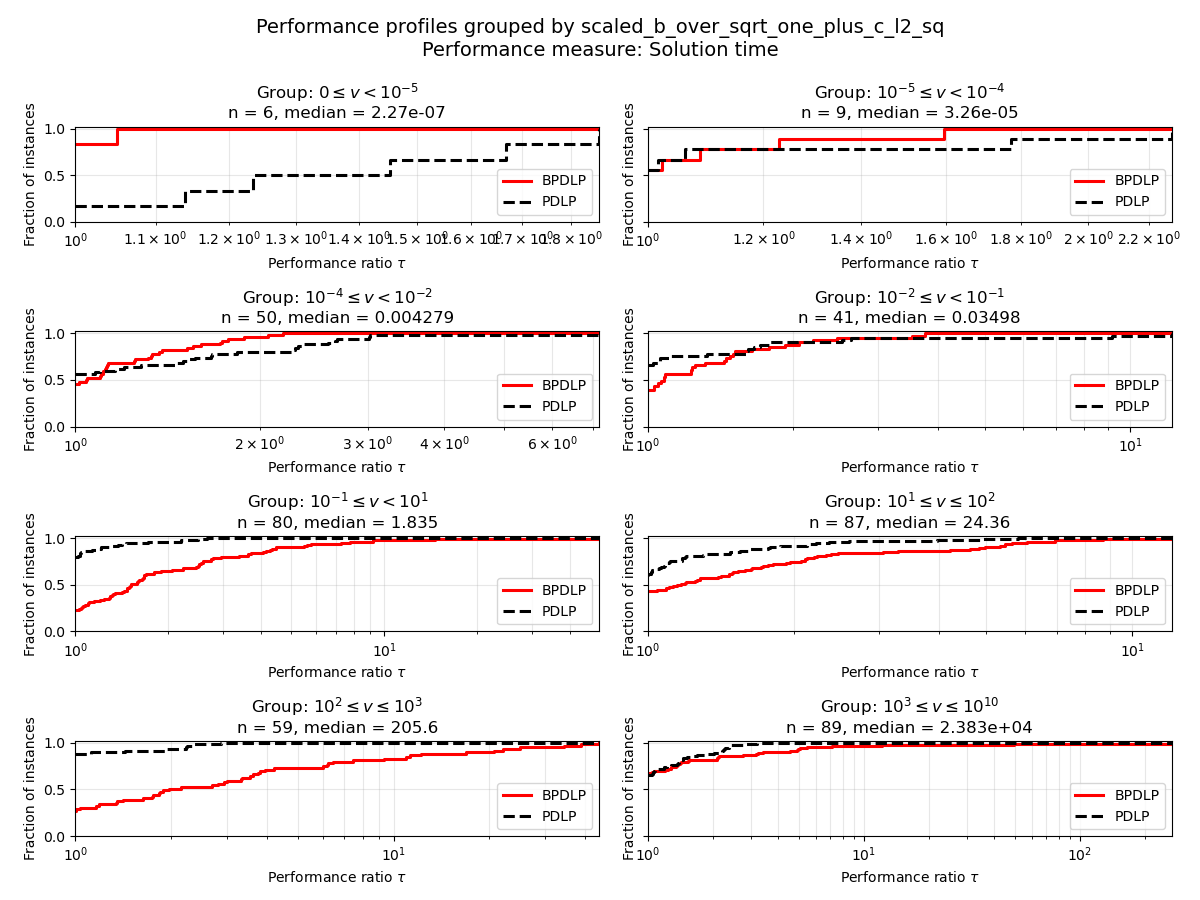}
\caption{Performance profiles of BPDLP and PDLP grouped by the value of $v$. The upper figure use
iteration count as the performance measure, while the lower use
solution time. \(n\) denotes the number of tested instances in each group, and the median is the median value of $v$ among tested instances.}\label{fi:pp-8-it-time}
\end{figure}

Figure~\ref{fi:pp-8-it-time} suggests that BPDLP tends to perform better than PDLP in terms of both iteration count and
solution time when \(v\) is small. However, the
converse does not necessarily hold.

To provide a more direct visualization of this relationship, we also
plot the iteration ratio between BPDLP and PDLP against \(v\) in
Figure~\ref{fi:scatter-v-iteration}. Instances for which either method
reaches the maximum iteration limit are excluded from this plot. The figure shows that instances
with smaller values of \(v\) are more likely to satisfy
\[
\frac{\text{BPDLP iterations}}{\text{PDLP iterations}}<1,
\]
which implies that BPDLP requires fewer iterations than PDLP on these instances.

\begin{figure}
    \centering
    \includegraphics[width=0.8\linewidth]{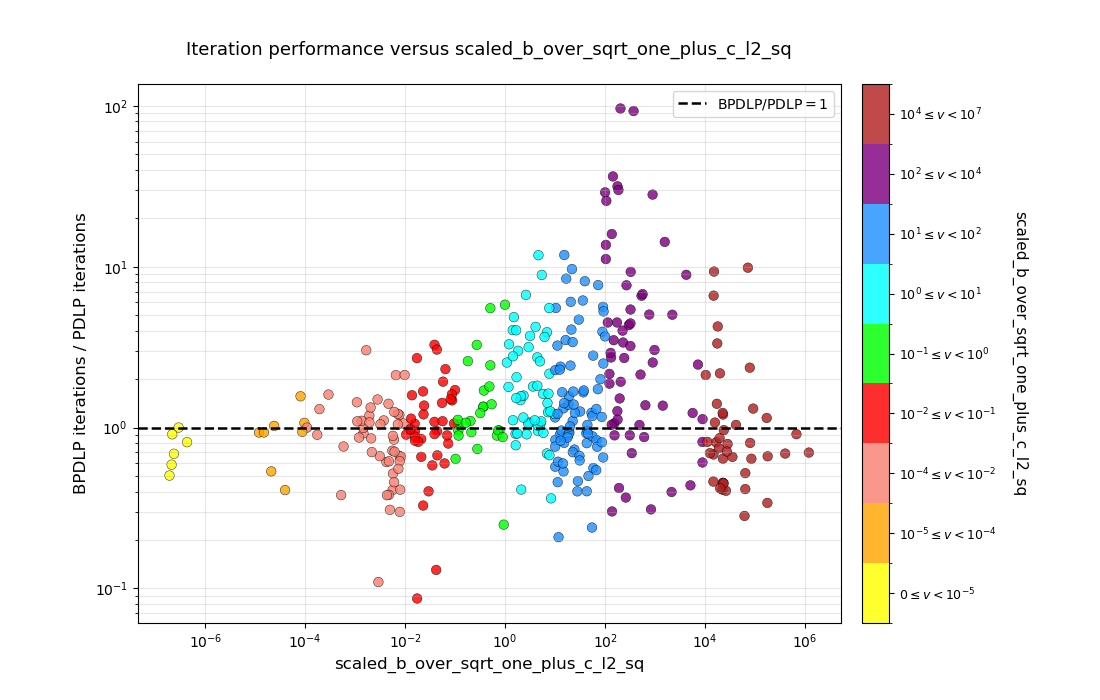}
    \caption{BPDLP$/$PDLP iteration ratio against the
proposed metric \(v\). The horizontal dashed line indicates equal
iteration counts for the two methods. Points below the line correspond
to instances on which BPDLP requires fewer iterations than PDLP.}
    \label{fi:scatter-v-iteration}
\end{figure}

Overall, although \eqref{possible-measure} gives a  potentially loose upper bound,
it remains easy to compute and provides a practical quantity that can
be monitored across datasets. This conclusion also answer the question proposed in the beginning of
Section~\ref{subse:EmpiricalCharacteristics}. Our results suggest that, when a small
value of \(v\) is observed for an LP instance, using BPDLP is more likely
to yield a better performance than PDLP.

\subsection{Secondary Empirical Observation: C-Zero-Ratio}

{
We further identify an empirical quantity that appears to distinguish
instances on which BPDLP is more likely to outperform PDLP. This
quantity is the proportion of zero entries in the objective coefficient
vector \(c\), which we refer to as the C-zero-ratio. At present, we do not have a theoretical explanation for why this
quantity is informative or an analytical result supporting the observed
relationship. We therefore present it as a secondary empirical
observation and report the corresponding performance profiles.
}

The C-zero-ratio is defined by
\[
\text{C-zero-ratio}
:=
\frac{\left|\left\{i:c_i=0\right\}\right|}{n},
\]
where \(n\) is the number of variables. 

\begin{figure}
\centering
\includegraphics[width=0.9\linewidth]{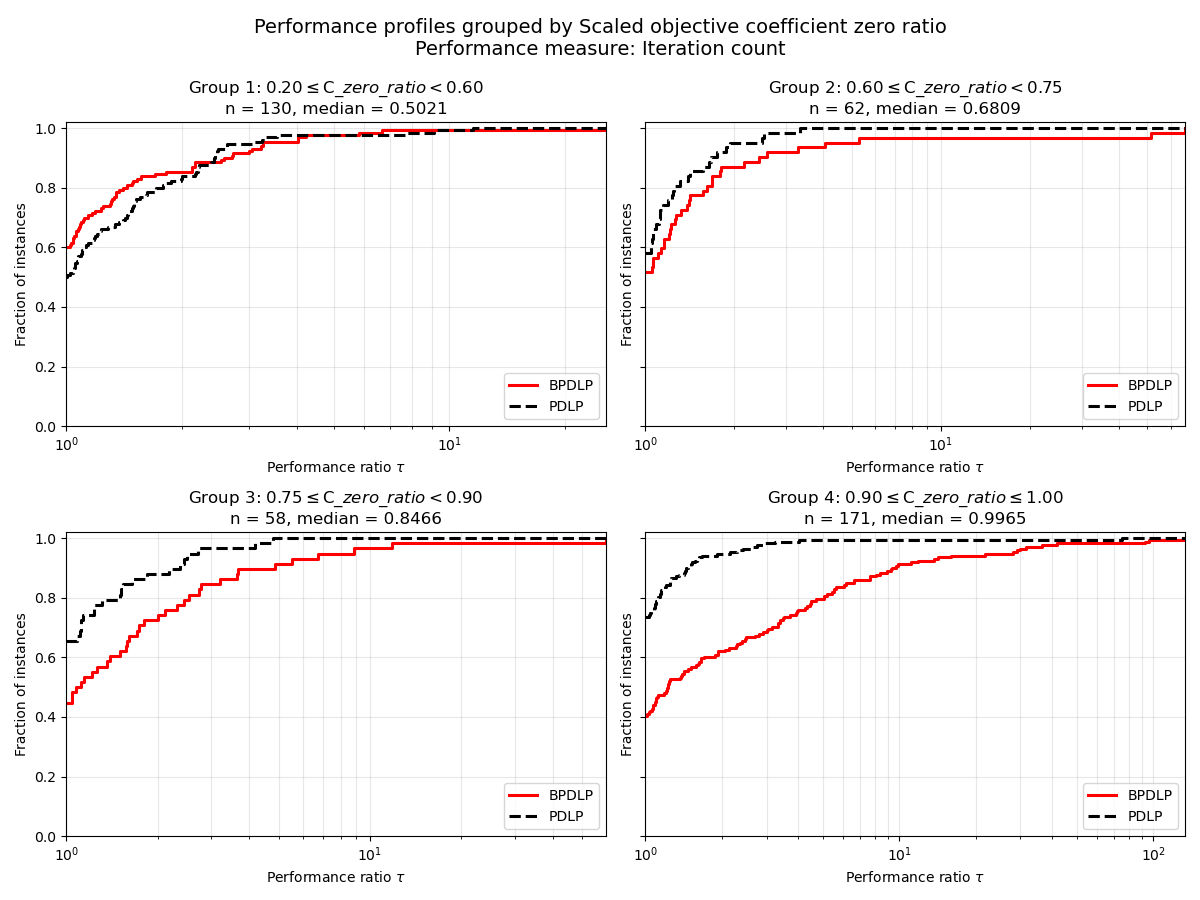}
\includegraphics[width=0.9\linewidth]{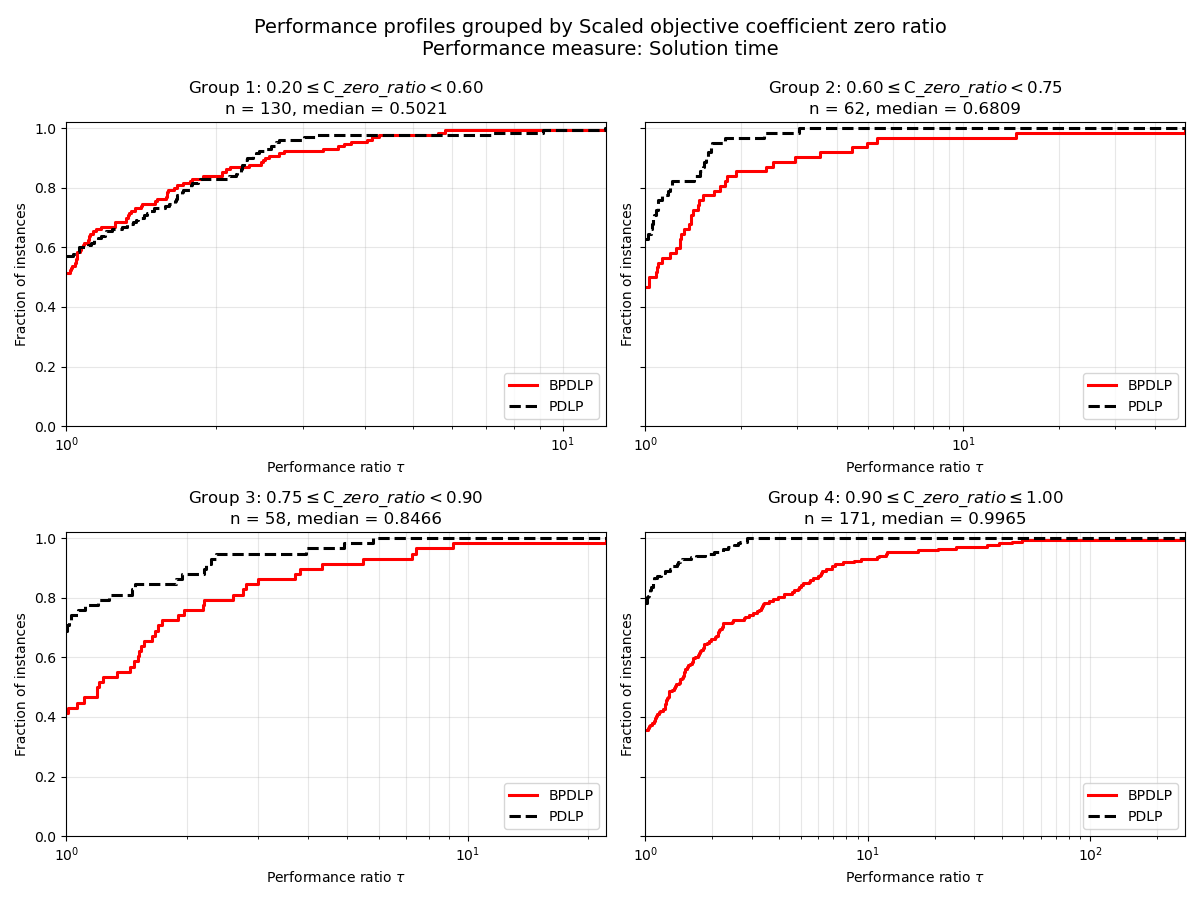}
\caption{Performance profiles of BPDLP and PDLP grouped by the value of $C$-zero-ratio. The upper figure use
iteration count as the performance measure, while the lower use
solution time. \(n\) denotes the number of tested instances in each group, and the median is the median value among tested instances.}
\label{fig:pp-zero-it-time}
\end{figure}

Figure~\ref{fig:pp-zero-it-time} indicates that BPDLP is more likely to outperform PDLP on instances with smaller values of the C-zero-ratio. Compared with \(v\), the C-zero-ratio provides a clearer empirical
separation between instances that are more or less suitable for BPDLP.
As the C-zero-ratio increases, the relative performance advantage of
BPDLP over PDLP generally decreases.

We next examine whether this empirical observation is consistent with
the behavior of a particular class of instances.

In our numerical experiments, BPDLP outperforms PDLP on the
production-inventory instances considered in this study. We therefore
report the corresponding performance profiles in
Figure~\ref{fig:pp_production_inv-all}.

\begin{figure}
    \centering
    \includegraphics[width=1\linewidth]{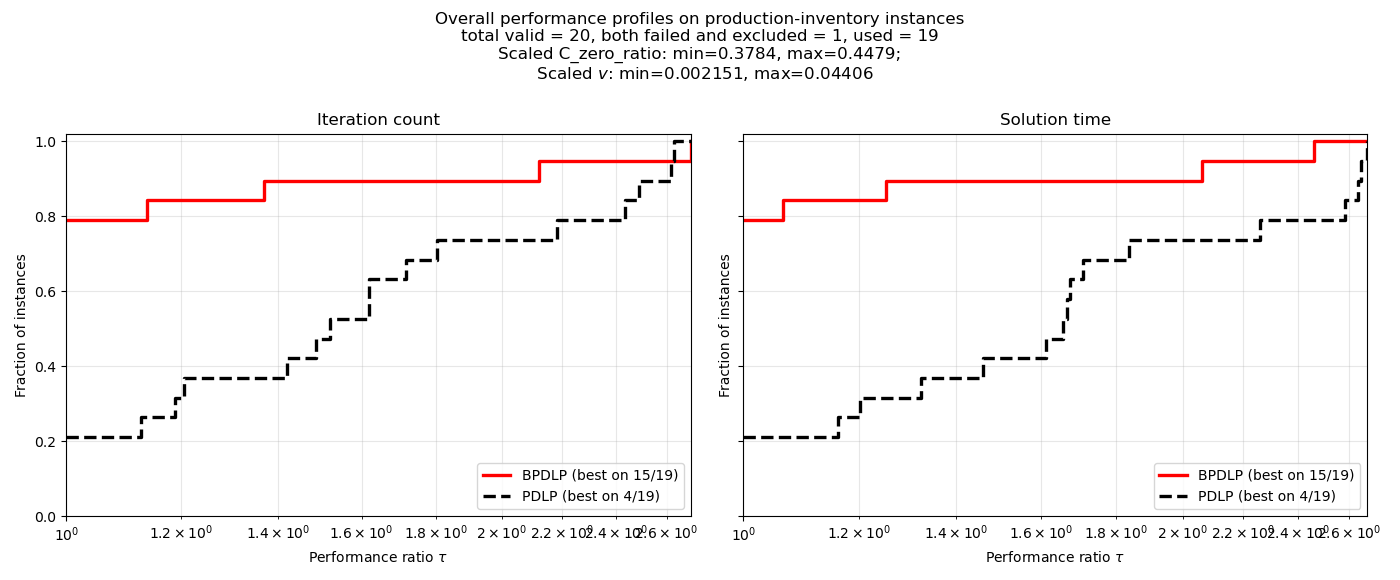}
    \caption{Performance profiles of BPDLP and PDLP on the
production-inventory instances. 
Left: comparison with iteration count. Right: comparison with solution time.}
    \label{fig:pp_production_inv-all}
\end{figure}

Figure~\ref{fig:pp_production_inv-all} shows that BPDLP substantially
outperforms PDLP on the production-inventory instances considered here.
For this class of problems, the C-zero-ratio is relatively small, with
most instances falling into Group~1 of
Figure~\ref{fig:pp-zero-it-time}. This is consistent with the empirical
relationship observed above. In addition, the values of \(v\) for these
instances are also relatively small. Taken together, these observations
suggest that production-inventory instances with such characteristics
may be particularly well suited to BPDLP.

\section{Conclusions}
In this paper, we propose Barrier PDHG, which incorporates a logarithmic barrier into the PDHG framework. The method replaces the nonsmooth projection operator with a smooth
interior update, with the aim of alleviating the prolonged active set
identification phase of PDHG. We further incorporate the barrier technique into the PDLP framework
and develop the corresponding Barrier PDLP (BPDLP) method. Numerical
experiments comparing BPDLP with PDLP show that BPDLP can achieve improvements on instances where PDLP exhibits a pronounced plateau in its KKT residual trajectory. Furthermore, we investigate the problem characteristics associated with
the relative performance of BPDLP and PDLP, and propose an empirical
indicator to identify instances on which BPDLP is more likely to be
advantageous. 


\bibliographystyle{unsrt}
\bibliography{sample}

\end{document}